\documentclass{amsart}

\usepackage{graphicx,color}
\usepackage{euscript}
\usepackage{amsmath}
\usepackage{amsfonts}
\usepackage{amssymb}
\usepackage{amsthm}
\usepackage{hyperref}
\usepackage{epsfig}
\usepackage{epsf}
\usepackage{epstopdf}
\usepackage{caption}
\usepackage{url}
\usepackage{array}
\usepackage{amssymb}\usepackage{amscd}
\usepackage{epic}
\usepackage{eufrak}
\usepackage{tikz,underscore,pgfplots}
\usepackage{tikz-cd}
\usepackage{bm}
\usepackage{float}
\usepackage{anysize}
\usepackage{amsfonts}
\usepackage{pinlabel} 
\usepackage{combelow}
\usepackage{comment}

\newtheorem{theorem}{Theorem}[section] 
\newtheorem{corollary}[theorem]{Corollary} 
\newtheorem{lemma}[theorem]{Lemma} 
\newtheorem{proposition}[theorem]{Proposition} 
\newtheorem{example}[theorem]{Example}
\newtheorem{conjecture}[theorem]{Conjecture}
\newtheorem{problem}[theorem]{Problem}  
\newtheorem{remark}[theorem]{Remark}  
\newtheorem{definition}[theorem]{Definition}

\newcommand{\Kh}{\mathrm{Kh}}

\newcommand{\CA}{\mathcal{A}} 
\newcommand{\CH}{\mathcal{H}}
\newcommand{\Cab}{\mathrm{Cab}}
\newcommand{\col}{\mathrm{col}}
\newcommand{\Tw}{\mathcal{TW}}

\newcommand{\F}{\mathbb{F}}

\newcommand{\C}{\mathbb{C}}

\newcommand{\cCFL}{\mathcal{C\!F\!L}}

\newcommand{\cHFL}{\mathcal{H\!F\! L}}

\newcommand{\HFL}{\mathit{HFL}}

\newcommand{\s}{\mathfrak{s}}
\newcommand{\ee}{\mathbf{e}}

\newcommand{\lk}{\mathrm{lk}}

\newcommand{\Spin}{\text{Spin}^{c}}
\newcommand{\alphas}{\boldsymbol{\alpha}}
\newcommand{\betas}{\boldsymbol{\beta}}
\newcommand{\gammas}{\boldsymbol{\gamma}}

\newcommand{\ovl}{\overline}

\newcommand{\HD}{\mathcal{H}}
\newcommand{\G}{\mathcal{G}}

\newcommand{\tw}{\mathbf{tw}}
\newcommand{\w}{\mathbf{w}}
\newcommand{\x}{\mathbf{x}}
\newcommand{\y}{\mathbf{y}}

\newcommand{\newword}[1]{\emph{\textbf{#1}}}

\newcommand{\spinc}{\text{Spin}^{c}}

\newcommand{\HY}{\mathrm{HY}}

\def\L{\mathcal{L}}
\def\gr{\textup{gr}}
\def\w{{\bf{w}}}
\def\z{{\bf{z}}}

\def\x{\mathbf{x}}
\def\y{\mathbf{y}}

\newcommand{\HHH}{\mathrm{HHH}}

\newcommand{\FT}{\mathrm{FT}}

\newcommand{\HH}{\mathbb{H}}
\newcommand{\Z}{\mathbb{Z}}
\newcommand{\R}{\mathbb{R}}
\newcommand{\UU}{\mathbf{U}}

\newcommand{\Br}{\mathrm{Br}}

\newcommand{\cc}{\mathbf{c}}

\newcommand{\ttt}{\mathbf{t}}

\newcommand{\Id}{\mathrm{Id}}

\newcommand{\spann}{\mathrm{span}}

\newcommand{\stab}{\mathrm{stab}}

\renewcommand{\ss}{\mathbf{s}}

\newcommand{\AAA}{\mathsf{A}}

\title{Colored knot Floer homology: structures and examples}

\author{Akram Alishahi}
\thanks{\texttt{AA was partly supported by NSF Grant DMS-2238103}}
\address{Department of Mathematics, University of Georgia, Athens, GA 30602 USA}
\email{\href{mailto:akram.alishahi@uga.edu}{akram.alishahi@uga.edu}}

\author{Eugene Gorsky}
\thanks{\texttt{EG was partly supported by NSF Grant DMS-2302305}}
\address{Department of Mathematics, University of California Davis\\ One Shields Avenue, Davis CA 95616 USA}
\email{\href{mailto:egorskiy@ucdavis.edu}{egorskiy@ucdavis.edu}}

\author{Beibei Liu}
\thanks{\texttt{BL was partly supported by NSF Grant DMS-2417229}}
\address{Department of Mathematics, The Ohio State University, 100 Math Tower, 231 W 18th Ave, Columbus, OH 43201 USA}
\email{\href{mailto:liu.11302@osu.edu}{liu.11302@osu.edu}}

\begin{document}

\begin{abstract}
Inspired by the $S^n$ colored version of Khovanov and Khovanov-Rozansky homology, we define a colored version of knot Floer homology by studying the colimit of a directed system of link Floer homology with infinite full twists. Specifically,  our $n$-colored knot Floer homology of a knot $K$ is then defined as the colimit of the link Floer homology of $(n, mn)$-cables of  $K$ by fixing $n$ and letting $m$ goes to infinity. We show that the colimit of the infinite full twists is a module over the colored knot Floer homology of the unknot.  In addition,  we give an explicit description of colored  Heegaard Floer homology for L-space knots, and maps for colored knot Floer homology of crossing changes. 
\end{abstract}

\maketitle

\section{Introduction}


\subsection{Motivation: colored invariants via cables}

One of the central ideas in quantum topology is the  notion of colored knot invariants \cite{RT}. Given a knot $K\subset S^3$, a semisimple Lie algebra $\mathfrak{g}$ and its representation $V$, one can define the Laurent polynomial $P_{\mathfrak{g},V}(K)\in \Z[q,q^{-1}]$ which is a topological invariant of $K$. For example, for $\mathfrak{g}=\mathfrak{sl}_2$ and $V=\C^2$ one gets Jones polynomial, while  $\mathfrak{g}=\mathfrak{sl}_2$ and $V=S^n\C^2\simeq \C^{n+1}$ corresponds to the so-called colored Jones polynomial \cite{MM}. More generally, one refers to $P_{\mathfrak{g},V}(K)$ as the invariant of $K$ colored by the representation $V$. 

For $\mathfrak{g}=\mathfrak{sl}_N$ Rosso and Jones \cite{RJ} proved that the invariants of any cable $K_{n,p}$ of the knot $K$ are related to colored invariants of $K$. In particular, one has the following identity
\begin{equation}
\label{eq: colored as limit RJ}
P_{\mathfrak{sl}_N,S^n\C^N}(K)=\lim_{p\to \infty} P_{\mathfrak{sl}_N,\C^N}(K_{n,p}).
\end{equation}
On the left hand side of \eqref{eq: colored as limit RJ} we get $S^n\C^N$-colored (or simply $S^n$-colored) invariant of $K$, while on the right hand side we get the limit of uncolored invariants of cables $K_{n,p}$ normalized by a certain monomial in $q$ which we omit. The limit means that for all $i$ and sufficiently large $p$ the coefficient at $q^i$ in the right hand side stabilizes. 

In recent decades, there was a lot of work \cite{Cautis,CK,Hogancamp, Hogancamp2,Rozansky} towards categorifying \eqref{eq: colored as limit RJ} in the context of Khovanov and Khovanov-Rozansky homology \cite{Kh,KR1,KR2}. In particular, for $N=2$ the (uncolored) Jones polynomial is categorified by Khovanov homology $\Kh(K)$  and one can {\bf define} the $S^n$-colored Khovanov homology as the colimit
\begin{equation}
\label{eq: colored as limit Khovanov}
\Kh_{S^n}(K):=\varinjlim\left[\Kh(K_{n,0})\xrightarrow{\alpha_n}\Kh(K_{n,n})\xrightarrow{\alpha_n}\Kh(K_{n,2n})\xrightarrow{\alpha_n} \cdots\right]
\end{equation}
The terms on the right hand side differ by adding one full twist, and the colimit is often referred to as ``infinite full twist" \cite{Rozansky}.

The key novel feature of \eqref{eq: colored as limit Khovanov} are the {\bf connecting maps} $\alpha_n:\Kh\left(K_{n,mn}\right)\to \Kh\left(K_{n,(m+1)n}\right)$   which are required for the definition of colimit. The colimit on the right hand side of \eqref{eq: colored as limit Khovanov} is a bigraded vector space which is typically infinite-dimensional, but one can prove that it is finite-dimensional in each bidegree. See \cite{CK,Hogancamp, Hogancamp2,Rozansky} and references therein for more details and an explicit description of the maps $\alpha_n$, and related computations in Khovanov-Rozansky homology. See Section \ref{sec: HY comparison} for some computations of colored homology.

In this paper, we define and study some analogues of \eqref{eq: colored as limit RJ} and \eqref{eq: colored as limit Khovanov} in the context of Alexander polynomial and Heegaard Floer homology.  In fact, the analogue of
\eqref{eq: colored as limit RJ} is not too interesting, as the Alexander polynomials for cables of $K$ are proportional to the Alexander polynomial of $K$ evaluated at $t_1\cdots t_n$ (see Section \ref{sec: alex cable} for more details). However, an analogue of \eqref{eq: colored as limit Khovanov}, which we are about to define, has a rich and interesting structure.

\subsection{Colored knot Floer homology}
Analogously to \eqref{eq: colored as limit Khovanov}, we define a colimit invariant under full twists by using link Floer homology. Let $L\subset S^3$ be an oriented $n$-component link and $M$ be an unknot bounding a disk $D$ that intersects every component positively at exactly one point. Let $L_m$ denote the link obtained by inserting $m$ full twists in $L$. Then we can define a colimit of the following directed system of $\cHFL(L_m)$: 
\begin{equation}
\label{eq: def HD}
\HD_D(L)=\varinjlim\left[\cHFL(L)\xrightarrow{\phi_0}\cHFL(L_1)\xrightarrow{\phi_0}\cHFL(L_2)\to\cdots\right]
\end{equation}
where the connecting map $\phi_0$ is a cobordism map induced by blowing down the $(-1)$-framed unknot $M$ in some $\Spin$-structure (see Proposition \ref{prop:gradingshifts}), and  we use the ``full" version of link Floer homology  developed in  \cite{Zemke, Zemke2}. In particular, $\cHFL(L_m)$ is a module over $\F[U_1,\ldots,U_n,V_1,\ldots,V_n]$.

In our first main result, we prove that $\HD_{D}(L)$ is well defined. 

\begin{theorem}
\label{thm: intro well defined}
The $\F$-vector space  $\HD_{D}(L)$ has well defined Alexander $\Z^n$-grading, and Maslov $\Z$-grading. The homogeneous component of each $\Z^n\oplus \Z$-degree is finite-dimensional. 
\end{theorem}

To prove Theorem \ref{thm: intro well defined}, we normalize the Alexander degrees such that they are preserved by the maps $\phi_0$. More precisely, we define 
\begin{equation}
\label{eq: def normalized Alexander}
\cHFL^{\stab}(L_m;\overline{\ss}):=\cHFL(L_m,\ss)
\end{equation}
where 
$\overline{\ss}=\ss-(\cc_m,\ldots,\cc_m)$ and $\cc_m=m(n-1)/2$. We then  prove that for any fixed normalized Alexander degree $\overline{\ss}$ the dimension of $\cHFL^{\stab}(L_m;\overline{\ss})$ stabilizes for sufficiently large $m$:  
\begin{theorem}
\label{thm: HFL stabilization intro}
For any $\ovl{\ss}=(\ovl{s}_1,\ovl{s}_2,\cdots,\ovl{s}_n)$ with $\ovl{s}_i\ge C-m$ for all $i$ we have
\[\cHFL^{\stab}(L_m,\ovl{\ss})\cong \cHFL^{\stab}\left(L_{m+1},\ovl{\ss}\right)\]
Here, $C$ is some constant depending on $L$ but independent of $m$. 
\end{theorem}
The isomorphism in Theorem \ref{thm: HFL stabilization intro} is obtained by
analyzing certain special Heegaard diagrams of $L_{m}$ and $L_{m+1}$, constructing an explicit bijection on generators and verifying that it is a chain map, see Theorem \ref{thm: HFL stabilization}.
Theorem \ref{thm: HFL stabilization intro} gives an upper bound for the dimension of the limit (see Lemma \ref{lem: colimit bounded}) for each Maslov grading $j$:
\begin{equation}
\label{eq: colored dim estimate}
\dim \HD_{D,j}(L; \overline{\ss})\le \dim \cHFL^{\stab}_j(L_m;\overline{\ss}),\quad m\gg 0.
\end{equation}
In fact, we conjecture a stronger statement. 

\begin{conjecture}
\label{conj: isomorphism}
Given a normalized Alexander degree $\overline{\ss}$, the maps
$$
\phi_0:\cHFL^{\stab}\left(L_m;\overline{\ss}\right)\to \cHFL^{\stab}\left(L_{m+1};\overline{\ss}\right)
$$
are isomorphisms for sufficiently large $m$. As a consequence, $\dim \HD_{D,j}(L;\overline{\ss})=\dim \cHFL^{\stab}_j\left(L_m;\overline{\ss}\right)$ for $m\gg 0$.
\end{conjecture}

\begin{remark}
Based on Theorem \ref{thm: HFL stabilization intro}, one might be tempted to define a ``naive" version of $\HD_D(L)$ by simply computing $\cHFL^{\stab}\left(L_m;\overline{\ss}\right)$ for $m\gg 0$. However, it is unclear whether the isomorphisms constructed in the proof of Theorem \ref{thm: HFL stabilization intro} have any nice topological properties.

On the other hand, the map $\phi_0$ does interact well with the cobordism maps in Heegaard Floer homology which yields nice properties of the colimit \eqref{eq: def HD}. See, in particular, Theorems \ref{thm: intro module} and \ref{thm:colored crossing change} below.

\end{remark}

\begin{remark}\label{rmk:orientation}
Throughout this paper, for simplicity, we focus on the case that every component of the link intersects the disk $D$ positively in one point. However, an analogous version of Theorem \ref{thm: HFL stabilization intro} holds in general. More precisely, suppose there exists a subset $I\subset\{1,2,\cdots,n\}$ such that every component of $L_I=\amalg_{i\in I}L_i$ (resp. $L_{\bar{I}}=\amalg_{i\in \bar{I}}L_i$) intersects $D$ negatively (resp. positively) at one point, where $\bar{I}=\{1,2,\cdots,n\}\setminus I$. Let $|I|=l$. If we use $\phi_{-l}$ to define $\HD_{D}(L)$ instead of $\phi_0$ then there is a constant $C$ independent of $m$ such that for any $\bar{\ss}$ with $\bar{s}_i\le -(C-m)$ for $i\in I$ and $\bar{s}_i\ge C-m$ for $i\in\bar{I}$ we have $\cHFL^{\stab}(L_{m};\bar{\ss})\cong\cHFL^{\stab}(L_{m+1};\bar{\ss})$. Here, Alexander grading is renormalized such that $\bar{s}_i=s_i+\cc_m$ for $i\in I$ and $\bar{s}_{i}=s_i-\cc_m$ for $i\in\bar{I}$. This is a direct consequence of Theorem \ref{thm: HFL stabilization intro} and the symmetry of link Floer homology, that is $\cHFL(L,\ss)\cong\cHFL(L';\ss')$ where $L'=(-L_{I})\amalg L_{\bar{I}}$ and $\ss'$ is obtained from $\ss$ by multiplying every $i$-th entry with $i\in I$ by $-1$. Consequently, we can define $\HD_{D}(L)$ with connecting homomorphisms $\phi_{-l}$ and after proper renormalization of the Maslov grading, the limit is graded by Maslov and Alexander grading and is finite dimensional in each degree.
\end{remark}

Given a knot $K\subset S^3$ and an integer $n>0$ we consider the family of $n$-component cables $K_{n,mn}$ for $m\ge 0$ and their link Floer homology $\cHFL(K_{n,mn})$. All components of $K_{n,mn}$ are isotopic to $K$ and their pairwise linking numbers are all equal to $m$. Applying the colimit definition to the link $K_{n, mn}$, we define the $n$-colored knot Floer homology of $K$.

\begin{definition}
\label{def: colored homology intro}
We define the $n$-colored knot Floer homology of $K$ as the colimit of the directed system
\begin{equation}
\label{eq: colored def intro}
\CH_{n}(K)=\varinjlim\left[\cHFL(K_{n,0})\xrightarrow{\phi_0}\cHFL(K_{n,n})\xrightarrow{\phi_0}\cHFL(K_{n,2n})\to\cdots\right].
\end{equation}
This is a special case of $\HD_D(L)$ where $L=K_{n,0}$ and the disk $D$ bounds the meridian of $K$.
\end{definition}

For $n=1$ all cables $K_{1,m}$ coincide with $K$ and  we get $\CH_1(K)=\cHFL(K)$. In Lemma \ref{lem: stable h} we prove that Conjecture \ref{conj: isomorphism} holds whenever $K$ is an L-space knot. Assuming Conjecture \ref{conj: isomorphism}, the inequality \eqref{eq: colored dim estimate} becomes an equality, and we get the ranks of $\CH_n(K)$ for all (normalized) Alexander degrees, as well as its Euler characteristic.

\begin{theorem}
Assuming Conjecture \ref{conj: isomorphism}, the Euler characteristic of $\CH_{n}(K)$ is given by 
$$
\chi(\CH_{n}(K))=\lim_{m\to \infty}(t_1\cdots t_n)^{-m(n-1)/2}\chi_{K_{n,mn}}(t_1,\ldots,t_n)=(t_1\cdots t_n)^{1/2}\chi_K(t_1\cdots t_n)
$$
where $\chi_K(t)=\frac{\Delta_K(t)}{1-t^{-1}}$ is the Euler characteristic of $\cHFL(K)$.
\end{theorem}

\begin{remark}
\label{rem: framing dependence}
We can choose some framing $f$ on the knot $K$ and consider the $n$-component link $L_f=K_{n,fn}$ whose components are pushoffs of $K$ along the framing. Adding a full twist changes $f$ to $f+1$, and we obtain  the colimit
$$
\CH_D(L_f)=\varinjlim\left[\cHFL(K_{n,fn})\xrightarrow{\phi_0}\cHFL(K_{n,(f+1)n})\xrightarrow{\phi_0}\cHFL(K_{n,(f+2)n})\to\cdots\right].
$$
It is easy to see that as a vector space this does not depend on $f$ and is isomorphic to $\CH_n(K)$. The Maslov grading does not depend on $f$, but the normalized Alexander grading is shifted by $\left(\frac{f(n-1)}{2},\ldots,\frac{f(n-1)}{2}\right)$.
\end{remark}

\begin{remark}
  One can define a colimit using different choices of connecting maps $\phi_k$ corresponding to different $\Spin$ structures in the surgery cobordism. We expect that other $\phi_k$  correspond to different ``colors" of the homology. For simplicity of computation, we focus on the choice of $\phi_0$ throughout the paper. By conjugation, the colimit for the choice of $\phi_{n-1}$ is the same as the choice of $\phi_0$ up to some regrading and interchanging $U_i$ with $V_i$. 
\end{remark}

\begin{remark}
One can define a similar colimit when components of the cables are oriented differently, see Remark \ref{rmk:orientation}.
\end{remark}

Instead of considering the family of $(n,mn)$ cables of $K$ with $n$ components, we can instead consider the family of $(n,mn+r)$-cables  of $K$ for fixed remainder $r$. These links have $\gcd(n,r)$ connected components; in particular, for $r=1$ we get the family of knots $K_{n,mn+1}$. The maps 
$\phi_0:\cHFL(K_{n,mn+r})\to \cHFL(K_{n,(m+1)n+r})$ can be defined as above, and  we can consider the directed system 
\begin{equation}
\label{eq: directed system r}
\varinjlim\left[\cHFL(K_{n,r})\xrightarrow{\phi_0}\cHFL(K_{n,n+r})\xrightarrow{\phi_0}\cHFL(K_{n,2n+r})\to\cdots\right].
\end{equation}
The following conjecture, if true, would imply that the colimit of \eqref{eq: directed system r} is well defined. By taking into account the Alexander grading shift of $\phi_0$, one can define the normalized Alexander grading $\widetilde{\ss}$ which generalizes $\overline{\ss}$ and is preserved by $\phi_0$.

\begin{conjecture}
\label{conj: intro r}
a) For a given normalized Alexander degree $\widetilde{\ss}$, the dimension of $\cHFL^{\stab}(K_{n,mn+r},\widetilde{\ss})$ stabilizes for sufficiently large $m$. 

b) The maps
$$
\phi_0:\cHFL^{\stab}\left(K_{n,mn+r};\widetilde{\ss}\right)\to \cHFL^{\stab}\left(K_{n,(m+1)n+r};\widetilde{\ss}\right)
$$
are isomorphisms for sufficiently large $m$.
\end{conjecture}
 
At the moment, we cannot generalize our proof of Theorem \ref{thm: HFL stabilization intro} to the case $r\neq 0$. However, for $r=1$ Hedden proved Conjecture \ref{conj: intro r}(a) in \cite{Hedden,Hedden2} for minus and hat versions of Heegaard Floer homology, but part (b) of the conjecture appears to be open in these cases as well. 

\begin{remark}
Assuming Conjecture \ref{conj: intro r}(a), it appears that many results of the paper can be applied to \eqref{eq: directed system r} for arbitrary $r$, with some small modifications. We plan to investigate the properties of more general colored homology \eqref{eq: directed system r} in future work.
\end{remark}

Next, we study the module structure of $\HD_{D}(L)$. Since the connecting maps $\phi_0$ commute with the action of $U_i,V_i$, the limit is clearly a module over $\F[U_1,\ldots,U_n,V_1,\ldots,V_n]$. However, $\HD_{D}(L)$ has a richer module structure over a commutative algebra $\CA_n^{\col}$, which is a localization of the cable algebra $\CA_n$ defined in subsection \ref{cabled homology}. 

\begin{theorem}
\label{thm: intro module}
Consider the commutative algebra $\CA_n^{\col}$ with generators $U_1,\ldots,U_n,V_1,\ldots,V_n,\AAA$ and relations
\begin{equation}
\label{eq: colored algebra intro}
U_i=\AAA\prod_{j\neq i}V_j, \quad U_iV_i=U_{i'} V_{i'}\quad i, i'=1,\ldots,n.
\end{equation}
Then for an arbitrary $n$-component link $L$ there is an action of $\CA_n^{\col}$ on $\HD_{D}(L)$ (in particular, on $\HD_n(K)$ for any knot $K$).
The generator $\AAA$ changes Alexander degree by $(-1,\ldots,-1)$ and Maslov degree by $-2$.
\end{theorem}

\begin{problem}
\label{prob: intro finitely generated}
Is $\HD_{D}(L)$ a finitely generated module over $\CA_n^{\col}$?
\end{problem}

\subsection{Colored knot Floer homology and crossing changes}

Let $K^+, K^-$ denote two knot diagrams that differ at a single crossing, and $K^+$ represents the one with a positive crossing, $K^-$ represents the one with a negative crossing. There are cobordism maps between $\cHFL(K^+)$ and $\cHFL(K^-)$ induced by blowing down a $(-1)$-framed unknot. We study the colored knot Floer homology of $K^+, K^-$ respectively, and find analogous cobordism maps between them. 

\begin{theorem}\label{thm:colored crossing change}
For any $j\in \Z$, there are maps 
$$
G_j^{\col}:\CH_n(K^{-})\to \CH_n(K^+), \quad \ F_j^{\col}: \CH_n(K^+)\to \CH_n(K^-)
$$
with 
$$
\gr_{\w}(G_j^{\col})=\gr_{\w}(F_j^{\col})=-j^2-j,\quad A_i(G_j^{\col})=-2j+1, \quad A_i(F_j^{\col})=n-1.
$$
All maps $G_j^{\col}, F_j^{\col}$ commute with the action of the algebra $\CA_n^{\col}$. 
\end{theorem}

\subsection{Examples}

We compute a lot of examples of colored homology, starting with the unknot. In this case, $K_{n,mn}=T(n,mn)$ is a torus link, and its ``full" link Floer homology  was computed in \cite{BLZ}. This computation can be used to prove the following:

\begin{theorem}
\label{thm: intro unknot}
If $K=O_1$ is the unknot then $\CH_{n}(O_1)$ is a free rank one module over $\CA_n^{\col}$. In particular, it is isomorphic to $\CA_n^{\col}$ as an $\F$-vector space.
\end{theorem}

More generally, we can compute the colored homology when $K$ is an L-space knot. Recall that an L-space is a 3-manifold with minimal possible rank of Heegaard Floer homology \cite{OSlens}. A knot (resp. link) $K$ is called an L-space knot (resp. L-space link) if all sufficiently large Dehn surgeries of $S^3$ along $K$ yield L-spaces \cite{Yajing}. If $K$ is an L-space knot then by \cite{GH} for $m\gg 0$ the cable $K_{n,mn}$ is an L-space link, and $\cHFL(K_{n,mn})$ is determined by $\cHFL(K)$ and, in fact, by the Alexander polynomial of $K$. This allows us to prove the following.

\begin{theorem}
\label{thm: intro L space}
Consider the homomorphism $\varepsilon_n:\F[U,V]\to \CA_n^{\col}$ defined by 
$$
\varepsilon_n(U)=\AAA, \varepsilon_n(V)=V_1\cdots V_n.
$$
Assume $K$ is an L-space knot, then 
$$
\CH_{n}(K)\simeq \cHFL(K)\bigotimes_{\F[U,V]}\CA_n^{\col}
$$
where we regard $\CA_n^{\col}$ as a module over $\F[U,V]$ via the homomorphism $\varepsilon_n$.
\end{theorem}

See Section \ref{sec: L space} for more details. In particular, we give an explicit description of $\CH_{n}(K)$ by generators and relations in Theorem \ref{thm: colored L space}.

\subsection{Cabled homology}
\label{cabled homology}
In order to prove Theorem \ref{thm: intro module}, we need to study the relation between $\cHFL(L_m)$ for different $m$ in more detail. 

\begin{definition}
The $n$-strand \newword{cable algebra} $\CA_n$ is defined as a $\Z^n\oplus \Z\oplus \Z$ graded algebra  over $\F[U_1,
\ldots,U_n,V_1,\ldots,V_n,\UU]$ with commuting generators $a_0,\ldots,a_{n-1}$ and the following relations:
\begin{itemize}
\item $U_iV_i=\UU,\ i=1,\ldots,n$
\item Linear relations
\begin{equation}
\label{eq: linear intro}
U_Ia_{k-1}=V_{\overline{I}}a_k
\end{equation}
where $I\subset \{1,\ldots,n\}$ is an arbitrary subset with $|I|=k$ and $\overline{I}=\{1,\ldots,n\}\setminus I$. Here, $U_I=\prod_{i\in I} U_i$ and $V_{\overline{I}}=\prod_{i\notin I}V_i$.
\item Quadratic relations
\begin{equation}
\label{eq: quadratic intro}
a_ia_j=\UU^{k\ell-ij}a_{k}a_{\ell}
\end{equation} 
whenever $i+j=k+\ell$ and $i\le k\le \ell\le j$.
\end{itemize}
The gradings correspond to the Alexander ($\Z^n$) and Maslov ($\Z$) gradings, along with an additional  $\Z$ ``twist grading". The generator $a_k$ has Alexander grading 
$$
A(a_k)=\left(\frac{n-1}{2}-k,\ldots, \frac{n-1}{2}-k\right),
$$
Maslov grading $\gr_{\w}(a_k)=-k^2-k$ and twist grading $\tw(a_k)=1$.
\end{definition}

\begin{theorem}
\label{thm: intro cabled module}
Consider the direct sum
$$
\Tw_D(L)=\bigoplus_{m=0}^{\infty}\cHFL(L_m)
$$
with its Alexander and Maslov grading together with additional ``twist'' grading given by $m$. Then $\Tw_D(L)$ is a $\Z^n\oplus\Z\oplus\Z$-graded module over $\CA_n$ where the generators $a_k$ act by certain cobordism maps 
$$
\phi_k:\cHFL(L_m)\to \cHFL(L_{m+1}).
$$
\end{theorem}

The maps $\phi_k$ are defined geometrically. Actually they are the cobordism maps induced from blowing down the $(-1)$-framed unknot corresponding to different $\Spin$ structures, see Proposition \ref{prop:gradingshifts}.  In order to
 prove Theorem \ref{thm: intro cabled module} one needs to verify the relations \eqref{eq: linear intro} and \eqref{eq: quadratic intro} for them. We first verify these equations in Theorem \ref{thm: cabled algebra} when $L=O_n$ is the $n$-component unlink, and prove that 
$$
\Tw_D(O_n)=\bigoplus_{m=0}^{\infty}\cHFL(T(n,mn))
$$
is a free rank one module over $\CA_n.$
We then use naturality properties of link Floer homology \cite{Zemke} to complete the proof for general $K$ in Theorem \ref{thm: cabled module}.

Theorem \ref{thm: intro module} can then be deduced from Theorem \ref{thm: intro cabled module} and the following result which we prove as Theorem \ref{thm: colored module}.

\begin{theorem}
\label{thm: intro localization}
The algebra $\CA_n^{\col}$ is isomorphic to (graded) localization $\CA_n[a_0^{-1}]$. Under this isomorphism, the generator $\AAA$ corresponds to $a_1/a_0$.
\end{theorem}









\subsection{Algebro-geometric interpretation}

Motivated by \cite{GH,GNR}, we can give a geometric interpretation of the algebras $\CA_n$ and $\CA_n^{\col}$. It is a special case of the so-called $\mathrm{Proj}$ construction in algebraic geometry\footnote{We work over the field $\F$ which is not algebraically closed. For more precise statements, one needs to either consider the algebraic closure of $\F$ or assume that all results above have characteristic zero analogues and work over $\C$.}.


Consider the $2n$-dimensional affine space $\mathbb{A}^{2n}_{\F}$ over the field $\F$ with coordinates $U_1,\ldots,U_n,V_1,\ldots,V_n.$ The equations $U_iV_i=U_jV_j$ for all $i\neq j$ define the affine algebraic variety 
$X_0\subset \mathbb{A}^{2n}_{\F}$ such that 
$$
\F[X_0]=\cHFL(O_n)
$$
where $O_n$ is the $n$-component unlink. Note that $\cHFL(O_n)$ is also isomorphic to the component of $\CA_n$ of twist degree 0.

To give a geometric interpretation to the whole algebra $\CA_n$, we consider an auxiliary projective space $\mathbb{P}^{n-1}_{\F}$ with homogeneous coordinates $[a_0:\ldots:a_{n-1}]$. Then  equations \eqref{eq: linear intro} and \eqref{eq: quadratic intro} define the algebraic variety
$$
X\subset X_0\times \mathbb{P}^{n-1}_{\F}\subset \mathbb{A}^{2n}_{\F}\times \mathbb{P}^{n-1}_{\F}
$$
By definition, the homogeneous coordinate ring of $X$ is isomorphic to $\CA_n$.

Next, consider the open chart $\{a_0\neq 0\}\simeq \mathbb{A}^{n-1}_{\F}$ in $\mathbb{P}^{n-1}_{\F}$. By intersecting it with  $X$, we obtain an open subset $\mathcal{U}_0$. This is again an affine algebraic variety with the coordinate ring
$$
\F[\mathcal{U}_0]=\CA_n[a_0^{-1}]=\CA_n^{\col}.
$$
 
Finally, we can interpret Theorem \ref{thm: intro cabled module} by saying that $\Tw_D(L)$ defines a quasi-coherent sheaf on $X$, and colored homology $\HD_{D}(L)$ corresponds to the restriction of this sheaf to the open chart $\mathcal{U}_0.$ 

Note that restricting to a different chart $\{a_k\neq 0\}$ would correspond to colimit \eqref{eq: colored def intro} with connecting maps $\phi_k$.

\subsection{Comparison with Khovanov-Rozansky homology}
\label{sec: HY comparison}

As mentioned above, our definition of colored homology is inspired by the $S^n$-colored Khovanov and Khovanov-Rozansky homology, it would be interesting to find any relation to these.
First, we recall some known computations of colored homology, starting with  triply graded homology $\HHH$ corresponding to (colored) HOMFLY-PT polynomial.

\begin{theorem}[\cite{Hogancamp}]
 The $S^n$-colored triply graded Khovanov-Rozansky homology of the unknot  is a free polynomial algebra $\HHH_{S^n}(O_1)$ with even generators  $u_0,\ldots,u_{n-1}$ and odd generators $\xi_0,\ldots,\xi_{n-1}$.
\end{theorem}

Rasmussen \cite{RasDiff} constructed a spectral sequence from $\HHH$ to Khovanov homology. It is expected that it extends to colored homology, more precisely, we get the following.

\begin{conjecture}[\cite{GOR}]
The Rasmussen spectral sequence from $\HHH_{S^n}(O_1)$ to $\Kh_{S^n}(O_1)$ has only one nontrivial differential $d_2$ given by
$
d_2(u_i)=0,\ d_2(\xi_i)=\sum_{j=0}^{i}u_ju_{i-j}.
$
As a consequence,
$$
\Kh_{S^n}(O_1)\simeq H^*
\left(\Z[u_0,\ldots,u_{n-1},\xi_0,\ldots,\xi_{n-1}],d_2\right).
$$
\end{conjecture}

Beliakova, Putyra, Robert and Wagner \cite{BPRW} recently defined another spectral sequence from (reduced) triply graded homology to the hat-version of Heegaard Floer homology, confirming a conjecture of Dunfield, Gukov and Rasmussen \cite{DGR}. It would be interesting to extend this spectral sequence to the ``full" and colored versions of knot Floer homology. 
Following the computations in \cite{AGL}, we expect that $\cHFL$ is in fact related to the so-called ``$y$-ified homology" $\HY$ defined in \cite{GHog}.
We have the following result.

\begin{theorem}[\cite{BGHW}]
\label{thm:colored HY}
The $S^n$--colored $y$-ified homology of the unknot with rational coefficients is a free polynomial algebra:
$$
\HY_{S^n}(O)\simeq \mathbb{Q}[u_0,\ldots,u_{n-1},\xi_0,\ldots,\xi_{n-1},y_1,\ldots,y_n].
$$
It is a module over  $\mathbb{Q}[x_1,\ldots,x_n,y_1,\ldots,y_n]$ where
\begin{equation}
\label{eq: x HY}
x_i=u_0+u_{1}y_i+\ldots+u_{n-1}y_i^{n-1},\ i=1,\ldots,n.
\end{equation}
\end{theorem}

One could hope that Theorem \ref{thm:colored HY} holds over $\Z$ (or over $\F$) and speculate that the hypothetical spectral sequence from $\HY$ to $\cHFL$ sends $x_i$ to $U_i$ and $y_i$ to $V_i$, see \cite{AGL} for more details and the comparison of  homology of $T(n,n)$ in both theories. 
Quite surprisingly, it turns out that the equations \eqref{eq: colored algebra intro} are a specialization of \eqref{eq: x HY}, and one can hope that the colored homologies are compatible as well.

\begin{lemma}
\label{lem: HY specialization}
Let $e_{m}(V_1,\ldots,V_n)$ be the elementary symmetric polynomials in $V_i$. Under the homomorphism defined by
$$
x_i\mapsto U_i,\ y_i\mapsto V_i,\ u_k\mapsto (-1)^{k}e_{n-1-k}(V_1,\ldots,V_n)\AAA
$$
the equation \eqref{eq: x HY} transforms into \eqref{eq: colored algebra intro}. 
\end{lemma}

\begin{example}
For $n=3$ we get
$$
u_0\mapsto (V_1V_2+V_1V_3+V_2V_3)\AAA,\ u_1\mapsto -(V_1+V_2+V_3)\AAA,\ u_3\mapsto \AAA.
$$
\end{example}

\noindent \textbf{Proof of Lemma \ref{lem: HY specialization}:}
Let $e_m=e_{m}(V_1,\ldots,V_n)$ and $
\widehat{e}_m=e_{m}(V_1,\ldots,\widehat{V_i},\ldots,V_n).$
It is easy to see that 
$
e_m=\widehat{e}_m+V_i\widehat{e}_{m-1}.
$
After applying the homomorphism to the right hand side of \eqref{eq: x HY} we get
$$
e_{n-1}\AAA-e_{n-2}V_i\AAA+\ldots+(-1)^{n-1}V_i^{n-1}\AAA=
$$
$$
\left((\widehat{e}_{n-1}+V_i\widehat{e}_{n-2})-(\widehat{e}_{n-2}+V_i\widehat{e}_{n-3})V_i+\ldots+(-1)^{n-1}V_i^{n-1}\right)\AAA.
$$
All of the terms cancel out except for
$$
\widehat{e}_{n-1}\AAA=\prod_{j\neq i}V_j\cdot \AAA
$$
which agrees with the right hand side of \eqref{eq: colored algebra intro}.
\qed


\subsection{Relation to other works}

In this section, we briefly discuss the relation of our results to other work.

\begin{itemize}
\item[(1)] The study of cables in Heegaard Floer homology has a long and rich history, starting from the pioneering works of Hedden \cite{Hedden, Hedden2} and including very recent paper of Chen, Zemke and Zhou \cite{Chen}. We refer to \cite{Chen} for more references and context.
We note, however, that most references discuss cables where the pattern has one component (so the cable of $K$ is again a knot) while we prefer working with cables which have $n$ components.

In particular, our proof of Theorem \ref{thm: intro well defined} uses Heegaard diagrams of $K_{n,mn}$ which are inspired by, but slightly different from the ones used in \cite{Hedden} for $K_{n,mn+1}$.

The paper \cite{Chen} uses a powerful machinery of bordered bimodules \cite{Zemke:bordered} and their tensor products, it would be interesting to find a relation with our results. In particular, the answer in Theorem \ref{thm: intro L space} above seems to have a similar flavor to the constructions in  \cite{Chen} but we do not know any precise connection.

\item[(2)] Recent works of Cooper--Deyeso \cite{Cooper} and Wildi \cite{Wildi} also define colored variants of knot Floer homology. In particular, \cite[Corollary 4.6]{Cooper} defines it by considering a similar ``infinite full twist" limit, but uses hat-version of the homology and $(n,mn+1)$ cables which have one component. It would be interesting to relate their construction to ours.

By contrast, \cite{Wildi} defines an analogue of ``$\wedge^n$-colored homology" (as opposed to ``$S^n$-colored"), so the construction is very different. Still, one might expect an interesting symmetry between $S^n$- and $\wedge^n$-colored homology, see \cite{Conners}. 
\end{itemize}

\section*{Acknowledgments}

The authors would like to thank Anna Beliakova, Daren Chen, Robert Lipshitz, Jacob Rasmussen, Arno Wildi and Ian Zemke for useful discussions. 


\section{Background}

\subsection{Colimits}

We recall some useful definitions and properties of colimits that will be used throughout the paper. 

Let $\{V_i\ :\ i\ge 1\}$ be a collection of vector spaces, and $f_i:V_i\to V_{i+1}$ are some linear maps (called connecting maps). We call this data a directed system and draw it as follows: 
\begin{equation}
\label{eq: directed system}
(V_{\bullet},f_{\bullet}):=\left[V_1\xrightarrow{f_1} V_2\xrightarrow{f_2} \cdots \right].
\end{equation}
\begin{definition}
The colimit $\varinjlim(V_{\bullet},f_{\bullet})$ of the directed system $(V_{\bullet},f_{\bullet})$ is defined as the vector space spanned by pairs $(v,i)$ where $v\in V_i$ modulo the equivalence relation $(v,i)\sim (f_i(v),i+1)$.
\end{definition}

The following easy observation will be useful.

\begin{lemma}
\label{lem: colimit bounded}
Let $\{V_i\}$ be a sequence of vector spaces together with some connecting maps $f_i:V_i\to V_{i+1}$. Let $L=\varinjlim(V_i,f_i)$ be the colimit of this directed system.

If $\dim V_i\le D$ for sufficiently large $i$ then $\dim L\le D$ (in particular, $L$ is finite-dimensional).
\end{lemma}

\begin{proof}
Recall that $L$ is spanned by the equivalence classes of vectors $v_i\in V_i$ modulo relation $v_i\sim f_i(v_i)$. We will abbreviate $f_i$ to $f$ when it is clear from context, so $v_i\sim f(v_i)$ and $v_i\sim f^k(v_i)$ for all $k>0$.

Assume that $v_{i_1}\in V_{i_1},\ldots, v_{i_{D+1}}\in V_{i_{D+1}}$ represent equivalence classes of some vectors in $L$. Furthermore, assume that for $i>N$ we have $\dim V_i\le D$, pick $n=\max(i_1,\ldots,i_{D+1},N+1)$. Then the vectors
$$
f^{n-i_1}(v_{i_1}),\ldots, f^{n-i_{D+1}}(v_{i_{D+1}})
$$
all belong to $V_n$ and $\dim V_n\le D$. Therefore these vectors are linearly dependent and the respective vectors in $L$ (represented by $v_{i_1},\ldots,v_{i_{D+1}}$) are linearly dependent in $L$.

We conclude that any $D+1$ vectors in $L$ are linearly dependent, so $\dim L\le D$.
\end{proof}

\begin{lemma}
\label{lem: colimit action}
Suppose that $(V_{\bullet},f_{\bullet})$ and  $(W_{\bullet},g_{\bullet})$ are two directed systems, and $h_i:V_i\to W_{i+s}$ is a collection of maps such that $g_{i+s}\circ h_i=h_{i+1}\circ f_i$. Then there is a well defined map
$$
h:\varinjlim(V_{\bullet},f_{\bullet})\to \varinjlim(W_{\bullet},g_{\bullet}).
$$
\end{lemma}

\begin{proof}
We define $h(v,i)=(h_i(v),i+s)$ for $v\in V_i$. By our assumption, $h$ preserves the equivalence relation and hence defines a map of colimits.    
\end{proof}

\begin{example}
For $s=1$ we require that all squares in this diagram are commutative:
$$
\begin{tikzcd}
V_1 \arrow{r}{f_1} \arrow{dr}{h_1}& V_2 \arrow{r}{f_2} \arrow{dr}{h_2}& V_3 \arrow{r}{f_3} \arrow{dr}{h_3}& \cdots\\
W_1 \arrow{r}{g_1}& W_2 \arrow{r}{g_2}& W_3 \arrow{r}{g_3}& \cdots
\end{tikzcd}
$$
\end{example}




\subsection{Heegaard Floer homology}

In this subsection, we review the necessary background for link Floer homology. We assume some familiarity with the basic Heegaard Floer homology and its refinement to knots, see \cite{OSProperties, OSKnots, OSLinks, RasmussenKnots} for details.  Given a link $\L$ with $n$ components in a 3-manifold $Y$, one can define the ``full" version of Heegaard Floer complex $\cCFL(Y, L)$ from the Heegaard diagram by using some version of Lagrangian Floer complexes. The multi-pointed Heegaard link diagram $(\Sigma, \alphas, \betas,\w, \z)$ for such pair $(Y, L)$ is defined as follows:
\begin{enumerate}
    \item  $\Sigma$ is a closed oriented genus $g$ surface;
    \item $\alphas=\{ \alpha_1, \cdots, \alpha_{g+n-1} \}$, and $\betas=\{ \beta_1, \cdots, \beta_{g+n-1}\}$ are collections of simple closed curves on $\Sigma$ such that $\alphas$ and $\betas$ each span a $g$-dimensional subspace of $H_1(\Sigma, \Z)$ and  curves in $\alphas$ (and $\betas$) are pairwise disjoint.
    \item $\w=\{w_1, \cdots, w_n\}$ and $\z=\{z_1, \cdots, z_n\}$ denote sequences of basepoints on $\Sigma$ such that each component of $\Sigma\setminus \alphas$ (respectively of $\Sigma\setminus \betas$) contains a single point of $\w$ and a single point of $\z$. 
\end{enumerate}

The generators of the link Floer chain complex $\cCFL(Y, L)$ are tuples of intersection points $\x\in \mathbb{T}_\alpha\cap \mathbb{T}_{\beta}$ of the Lagrangian tori
$$\mathbb{T}_\alpha=\alpha_1\times \cdots \times \alpha_{g+n-1}, \quad \mathbb{T}_{\beta}=\beta_1\times \cdots \times \beta_{g+n-1}$$
in the symmetric product $\textup{Sym}^{g+n-1}(\Sigma)$. 
 
We work over the coefficient field $\F=\Z/ 2\Z$ in the paper. The marked points $z_i$ and $w_i$ are associated with formal variables $V_i$ and $U_i$, and the chain complex $\cCFL(Y,L)$ is a module over  the polynomial ring $R=\F[U_1, \cdots, U_n, V_1, \cdots, V_n]$. The differential in $\cCFL(Y, L)$ is obtained by counting pseudo-holomorphic disks in $\textup{Sym}^{g+n-1}(\Sigma)$ via:
\begin{equation}
\label{eq: HF differential}
\partial \x=\sum_{\y\in \mathbb{T}_\alpha\cap \mathbb{T}_\beta}\sum_{\varphi\in \pi_2(\x, \y),\mu(\varphi)=1}(\# \mathcal{M}(\varphi)/\R) U_1^{n_{w_1}(\varphi)}\cdots U_n^{n_{w_{n}}(\varphi)}V_1^{n_{z_{1}}(\varphi)}\cdots V_n^{n_{z_n}(\varphi)} \y. 
\end{equation}
Here the sum is taken over all holomorphic disks from $\x$ to $\y$ in $\pi_2(\x, \y)$ representing the homotopy classes of maps $\varphi$ from the unit disk $\mathbb{D}\subset \mathbb{C} $ to $\textup{Sym}^{g+n-1}(\Sigma)$ satisfying some boundary condition.  The Maslov index of $\varphi$ is denoted by $\mu(\varphi)$. For any point $x\in \Sigma\setminus (\alphas\cup \betas)$, $n_x(\varphi)$ denotes the intersection number of the divisor $\{x\} \times \textup{Sym}^{g+n-2}(\Sigma)\subset \textup{Sym}^{g+n-1}(\Sigma)$ with $\varphi(\mathbb{D})$.  The moduli space $\mathcal{M}$ consists of all pseudo-holomorphic curves representing $\varphi$ for a generic 1-parameter family of almost complex structures on $\textup{Sym}^{g+n-1}(\Sigma)$. For more details, see \cite{OSProperties}. 

We denote the homology of the link Floer chain complex $\cCFL(Y, L)$ by $\cHFL(Y, L)$, which is an invariant of the pair $(Y, L)$.  The actions of $U_iV_i$ on the complex $\cCFL(L)$ are pairwise homotopic, so we introduce the following relation:
$$U_1V_1=\cdots =U_nV_n=\UU,$$
and let $R_{UV}$ denote the ring generated by $U_1, \cdots, U_n, V_1, \cdots, V_n$ modulo this relation.  Then the action of $R$ on $\cHFL(Y, L)$ factors through $R_{UV}$. We get different versions of Heegaard Floer homology by setting extra constraints on the  variables $U_i, V_i$. For example, $H_\ast(\cCFL(Y, L)/(V_1-1, \cdots, V_n-1))$ is isomorphic to $\mathit{HF}^-(Y)$, the minus version of Heegaard Floer homology of the 3-manifold $Y$. 

In this paper, we focus on links $L=L_1\cup \ldots \cup L_n$ in the three-sphere and, for simplicity, we write the complex as $\cCFL(L)$. The complex $\cCFL(L)$ and its homology $\cHFL(L)$ have several gradings. Firstly, there is the $\mathbb{Q}\times \mathbb{Q}$-valued Maslov bigrading, denoted by $(\gr_{\w}, \gr_{\z})$, as well as the Alexander multi-grading $A=(A_1, \ldots, A_n)$ valued in the lattice
$$\HH_L=\Z^n+\dfrac{1}{2}(\ell_1, \ldots, \ell_n)$$
where $\ell_i$ is the linking number of $L_i$ with the rest of the components, i.e., $\ell_i=\sum_{j\neq i} lk(L_i, L_j)$. These gradings are relatively determined by the following equations:
\begin{equation}\label{eq:relgradings}
\begin{split}
&\gr_{\w}(\x)-\gr_{\w}(\y)=\mu(\varphi)-2n_{\w}(\varphi)\\
&\gr_{\z}(\x)-\gr_{\z}(\y)=\mu(\varphi)-2n_{\z}(\varphi)\\
&A_i(\x)-A_{i}(\y)=n_{z_i}(\varphi)-n_{w_i}(\varphi)
\end{split}
\end{equation}
where $\varphi\in\pi_2(\x,\y)$ and $n_{\w}(\varphi)=\sum_{i=1}^nn_{w_i}(\varphi)$ whereas $n_{\z}(\varphi)=\sum_{i=1}^nn_{z_i}(\varphi)$. Furthermore, the actions of $U_i,V_i$ are homogeneous with respect to these gradings with weights
\begin{equation}
\label{eq: U V gradings}    
(\gr_{\w}, \gr_{\z})(U_i)=(-2, 0), \quad (\gr_{\w}, \gr_{\z})(V_i)=(0, -2), \quad \text{and } A(V_i)=-A(U_i)=\ee_i
\end{equation}
where $\ee_i$ is the standard $i$-th coordinate vector in $\R^n$. 
The differential on $\cCFL(L)$ preserves the Alexander grading, while dropping both $\gr_{\w}$ and $\gr_{\z}$ by  one. In fact, these gradings satisfy the relation
$$A_1+\ldots+A_n=\frac{1}{2} (\gr_{\w}-\gr_{\z}).$$


For any $\ss\in \HH_L$, let $\cCFL(L, \ss)$ denote the subcomplex of $\cCFL(L)$ generated by all elements $x$ such that $A(x)=\ss$. By the grading formula \eqref{eq: U V gradings}, it is straightforward to see that the product $U_iV_i$ preserves the Alexander gradings, and $\cCFL(L, \ss)$ is a module over the subring $\F[U_1V_1, U_2V_2, \cdots, U_nV_n]$, and so $\cHFL(L, \ss)$ is an $\F[\UU]$-module. In particular, for any link $L$ we have
$$\cHFL(L)=\oplus_{s\in \HH_L} \cHFL(L, \ss).$$
and $\cHFL(L, \ss)$ is the direct sum of one copy of $\F[\UU]$ with some $\UU$-torsion.

Recall that a rational homology sphere $M$ is an \emph{L-space} if for each $\Spin$-structure $\mathfrak{s}$,  $HF^-(M, \mathfrak{s})$ is isomorphic to $\F[\UU]$. A link $L$ in the three-sphere $S^3$ is an \emph{L-space link} if the Dehn surgery $S^3_{\mathbf{d}}(L)$ is an L-space for all $\mathbf{d}\gg 0$. Another way to characterize the L-space links is the condition that $\cHFL(L, \ss)\cong \F[\UU]$ for every lattice point $\ss\in \HH_L$. For any link $L$ in the three sphere, we define the link invariant,  known as the $h$-function.

\begin{definition}
    The $h$-function $h:\HH_L\to \Z$ of a link $L$ in the three-sphere is defined so that  for a given $\ss\in \HH_L$ the value $h(\ss)$ is  $-\dfrac{1}{2} \gr_{\w}$ for the maximal Maslov grading of non-torsion elements in $\cHFL(L, \ss)$. 
\end{definition}

For L-space links $L$, its $h$-function can be computed from the Alexander polynomials of the link and its sublinks, see \cite{BorodzikGorskyImmersed}. Moreover, for such links, $\cHFL(L)$ is determined by its $h$-function (see \cite{BLZ}), which is determined by these Alexander polynomials.

\subsection{Cobordism maps}
\label{sec: cobordism maps}

Recall that for an $n$-component link $L\subset Y$, each component is associated with two base points $z_i, w_i$  in the multi-pointed Heegaard diagram. In this subsection, we review cobordisms between two pairs $(Y_1, L_1, \w_1, \z_1)$ and $(Y_2, L_2, \w_2, \z_2)$ and the induced map between $\cHFL(Y_1, L_1)$ and $\cHFL(Y_2, L_2)$. 

A \emph{coloring} of a multi-based link $(L, \w, \z)$ is a map $\sigma: \w\cup \z\rightarrow P$ where $P=\{p_1, \cdots, p_k\}$ is a finite set, considered as the set of colors. We associate a free polynomial ring with a set of colors 
$$\mathcal{R}_P^-:=\F[X_{p_1}, X_{p_2}, \cdots X_{p_k}]$$
generated by the formal variables $X_{p_1}, X_{p_2}, \cdots X_{p_k}$. A coloring  $\sigma$ gives the ring $\mathcal{R}_p^-$ the structure of an $\F[U_{\w}, V_{\z}]$-module. So for a colored multi-based link $(L, \w, \z, \sigma)$, we define
$$\cCFL(L, \w, \z, \sigma)=\cCFL(L, \w, \z)\otimes_{\F[U_{\w}, V_{\z}]} \mathcal{R}_p^-. $$

\begin{definition}\cite[Definition 1.3]{Zemke2}
A decorated link cobordism from a 3-manifold with multi-based link $(Y_1, (L_1, \w_1, \z_1))$ to another one $(Y_2, (L_2, \w_2, \z_2))$ consists of a pair $(W, \mathcal{F})$ such that 
\begin{enumerate}
    \item $W$ is a compact 4-manifold with $\partial W=-Y_1\sqcup Y_2$; 
    \item $\mathcal{F}=(\Sigma, A)$ is an oriented, properly embedded surface $\Sigma$ in $W$, along with a properly embedded 1-manifold $A$ in $\Sigma$, called \emph{dividing arcs}. Furthermore, $\Sigma\setminus A$ consists of two disjoint (possibly disconnected) subsurfaces, $\Sigma_\w, \Sigma_\z$, such that the intersection of the closure of $\Sigma_\w$ and $\Sigma_\z$ is $A$; 
    \item $\partial \Sigma=-L_1\cup L_2$; 
    \item Each component of $L_1\setminus A$ (and $L_2\setminus A)$ contains exactly one basepoint;
    \item The $\w$ basepoints are all in $\Sigma_\w$ and the $\z$ basepoints are all in $\Sigma_\z$;
    \item $\mathcal{F}$ is equipped with a coloring, i.e. a map $\sigma: C(\Sigma\setminus A)\rightarrow P$, where $C(\Sigma\setminus A)$ denotes the set of components of $\Sigma\setminus A$. 

\end{enumerate}
    
\end{definition}

Zemke \cite[Theorem A]{Zemke} associated a $\Spin$ functorial chain map to a decorated cobordims $(W, \mathcal{F})$ with a $\Spin$ structure $\s$ on $W$:
$$F_{W, \mathcal{F}, \s}: \cCFL(Y_1, L_1, \w_1, \z_1, \sigma_1, \s|_{Y_1})\rightarrow \cCFL(Y_2, L_2, \w_2, \z_2, \sigma_2, \s|_{Y_2}),$$
where $\sigma_i$ denotes the coloring on $L_i$ induced by restricting $\sigma$ for $i=1, 2$. The maps are $\mathcal{R}_P^-$ equivariant, $\Z^p$-filtered, and are invariant up to $\mathcal{R}_P^-$-equivariant, $\Z^P$-filterd chain homotopies. The Maslov grading change and Alexander multi-grading changes for the cobordism maps are given in \cite[Theorem 1.4]{Zemke2}. Here we review the grading change formula for the cobordism induced by blowing down a $(-1)$-framed unknot, which is needed to define the colored knot Floer homology.

Let $L$ be an $n$-component link in the three-sphere, and $\overline{L}$ be another link obtained from $L$ by 
adding a full twist on $n$ strands, that is, $\overline{L}$ is  obtained from $L$ by blowing down an $(-1)$-framed unknot circling the $n$ strands of $L$ where we assume that all strands are oriented the same way. Consider the corresponding cobordism from $(S^3, L)$ to $(S^3, \overline{L})$ obtained by attaching a $2$-handle to $S^3\times I$ along the aforementioned $(-1)-$framed unknot in $S^3\times\{1\}$. Let $[S^2]$ denote the generator of the second homology class of the cobordism, and let $\phi_k:\cHFL(L)\to \cHFL(\overline{L})$ denote the induced cobordism maps in homology in the $\spinc$-structure $\mathfrak{s}_k$ such that  $\langle \mathfrak{s}_k, [S^2]\rangle =2k+1$.  See \cite{AGL} for a detailed discussion of the properties of $\phi_k$.  We have

\begin{proposition}\cite{AGL}\label{prop:gradingshifts}
Given an n-component link $L$ in the three sphere, the cobordism maps $\phi_k:\cHFL(L)\to\cHFL(\overline{L}) $ induced from the $(-1)$-surgery on the unknot with $0\leq k\leq n-1$ satisfy the following grading properties:
$$
\gr_{\w}(\phi_k)=-k^2-k,\ A_i(\phi_k)=-k+(n-1)/2,
$$
where $i=1, \cdots, n$. 
\end{proposition}

More generally, assume that we blow down a $(-1)$-framed unknot $M$ circling the link with arbitrary orientation of the strands. Then  we still get maps $\phi_{k}$ with $\gr_{\w}(\phi_k)=-k^2-k$ and
\begin{equation}
\label{eq: Alexander shift generalized}
A_i(\phi_{k})=\frac{1}{2}\left[-2k-1+\lk(L,M)\right]\lk(L_i,M)
\end{equation}
The degrees of $\phi_k$ are computed using \cite{Zemke2}.
 Proposition \ref{prop:gradingshifts} is a special case when $\lk(L,M)=n$ and $\lk(L_i,M)=1$  for all $i$.


\subsection{Alexander polynomials of cable links}
\label{sec: alex cable}
The link Floer homology $\HFL^-(L)$ is the categorification of the multi-variable Alexander polynomials of links $L$ in the three sphere. That is, 
$$\chi_L(t_1, \cdots, t_n)=\sum_{\ss\in \HH_L} \chi(\HFL^-(L, \ss))t_1^{s_1} t_2^{s_2}\cdots t_n^{s_n}.$$
Here $\chi_L(t_1, \cdots, t_n)$ is a normalization of Alexander polynomial $\Delta_L(t_1, \cdots, t_n)$ of $L$ as follows. 

If $L$ is a knot, then
$$
\chi_{L}(t)=\frac{\Delta_{L}(t)}{1-t^{-1}}.
$$
It is a polynomial in $t$ and an infinite series in $t^{-1}$. For a link $L$ with $n>1$ components, we get
$$
\chi_{L}(t_1,\ldots,t_n)=(t_1 t_2\cdots t_n)^{1/2}\Delta_{L}(t_1,\ldots,t_n).
$$

Given a knot $K$, recall that the multivariable Alexander polynomial of the $(n,mn)$ cable $K_{n,mn}$ (with $n$ components) is given  by the equation:
$$
\Delta_{K_{n,mn}}(t_1,\ldots,t_n)=\Delta_{K}(\ttt)\Delta_{T(n,mn)}(t_1,\ldots,t_n),
$$
$$
 \chi_{K_{n,mn}}(t_1,\ldots,t_n)=(1-\ttt^{-1})\chi_{K}(\ttt)\chi_{T(n,mn)}(t_1,\ldots,t_n).
$$
where $\ttt=t_1\cdots t_n$. Furthermore,
$$
\chi_{T(n,mn)}=\frac{(\ttt^{m/2}-\ttt^{-m/2})^{n-1}}{(\ttt^{1/2}-\ttt^{-1/2})},
$$ 
so we conclude
\begin{equation}
\chi_{K_{n,mn}}(t_1,\ldots,t_n)=\ttt^{-1/2}\chi_{K}(\ttt)(\ttt^{m/2}-\ttt^{-m/2})^{n-1}.
\end{equation}
This immediately implies the following:
\begin{lemma}
\label{lem: Alex stable}
Let $K$ be a knot of Seifert genus $g(K)$, and $m\gg 1$, define $\cc_m=\frac{m(n-1)}{2}$. Then 
$$
\ttt^{-\cc_m}\chi_{K_{n,mn}}(t_1,\ldots,t_n)=\ttt^{-1/2}\chi_{K}(\ttt)\mod \ttt^{-m}\mathbb{Z}[[\ttt^{-1}]].
$$
In particular,
$$
\lim_{m\to \infty}\ttt^{-\cc_m}\chi_{K_{n,mn}}(t_1,\ldots,t_n)=\ttt^{-1/2}\chi_{K}(\ttt).
$$
Furthermore, the normalized Euler characteristic $t^{-\cc_m}\chi_{K_{n,mn}}(t)$ stabilizes in the region 
$$\{s\preceq g(K)(1,\ldots,1)\}\setminus \{s\preceq (g(K)-m)(1,\ldots,1)\}$$ and vanishes for $s\not\preceq g(K)(1,\ldots,1)$.
\end{lemma}

Here two vectors $\mathrm{v}\preceq \mathrm{w}$ means that each component $v_i\leq w_i$ for $i=1, \cdots, n$. 

\section{Link Floer homology and infinite twists}
\label{sec: construction}

Let $L=\amalg_{i=1}^nL_i$ be an oriented $n$-component link in $S^3$ and $M$ 
  be an unknot bounding a disk $D$ that intersects every $L_i$ positively at exactly one point. Then, $(-1)$-surgery on $M$ will result in inserting a positive full twist in $L$. Let $L_m$ denote the link obtained by performing this operation $m$ times, that is inserting $m$ full twists in $L$. 




\subsection{Special Heegaard diagrams}

First, we consider some special Heegaard diagrams for $L_m$.

Corresponding to $L$ and $D$, let $G_{L,D}$ be the graph obtained by pinching $L$ along $D$ and creating a thick edge as in Figure \ref{fig:graph}. So $G_{L,D}$ has two vertices and $n+1$ edges connecting them. Conversely, any two pairs $(L,D)$ associated to such a bipartite graph $G$ with one distinguished edge, differ by inserting a pure braid in a neighborhood of the disk $D$. 

\begin{figure}[h]
\centering
\includegraphics[width=3in]{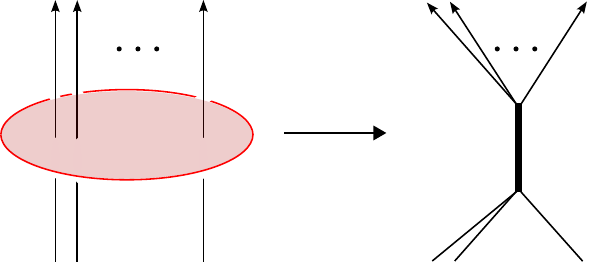}
   \caption{Graph associated to the link $L$ and the disk $D$, shaded red.} \label{fig:graph}
\end{figure}

Let 
\[H_G=\left(\Sigma,\alphas_G=\{\alpha_1,\cdots,\alpha_g\},\betas_G=\{\beta_1,\cdots,\beta_g\},z,\w=\{w_1,\cdots,w_n\}\right)\]
be a Heegaard diagram for $G_{L,D}$ where $z$ is the basepoint corresponding to the thick edge.   From $H_{G}$ we construct a Heegaard diagram for $L$ as follows. First, replace $z$ with $n$ basepoints $\z=\{z_1,\cdots,z_n\}$ in the same connected component of $\Sigma\setminus \{\alphas,\betas\}$. Let $D'\subset \Sigma$ be a small disk containing $z_1,\cdots,z_n$ and disjoint from $\alpha$- and $\beta$-circles.  For $1\le i\le n$ connect $z_i$ to $w_i$ with arcs $a_i$ and $b_i$ such that
\begin{enumerate}
\item $a_i$ is disjoint from the $\alpha$-circles and $b_i$ is disjoint from the $\beta$-circles,
\item $a_1,a_2,\cdots,a_n$ (resp. $b_1,b_2,\cdots,b_n$) are pairwise disjoint
\item pushing $a_1,\cdots,a_n$ and $b_1,\cdots,b_n$ in the $\alpha$- and $\beta$-handlebody, respectively, gives a link isotopic to $L$ with an isotopy that maps $D'$ to $D$. 
\end{enumerate}
Note that it is easy to arrange for conditions (1) and (2), however, $a_1,\cdots,a_n$ and $b_1,\cdots,b_n$ might represent a different link $L'$. Note that the graph associated to $(L',D')$ is the same as $G_{L,D}$. So $L$ is obtained from $L'$ by inserting a pure braid in a neighborhood of $D'$, which can be arranged by applying the corresponding diffeomorphism of $D'$ to $b_1,b_2,\cdots,b_n$. 

Next, by stabilizing the diagram $H_G$ as in Figure \ref{fig:stab}, we arrange for each $a_i$ to be disjoint from $b_j$ for $j\neq i$, while $a_i\cap b_i=\{z_i,w_i\}$. We keep using the same notation for the stabilized diagram.

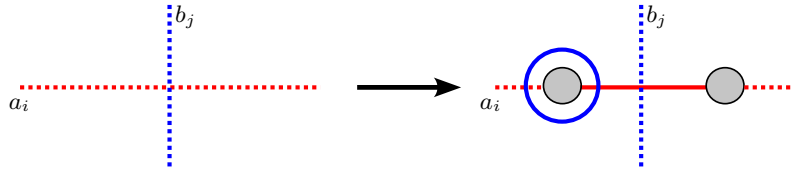
\begin{figure}[h]
\centering
\begin{tikzpicture}
    \node[anchor=south west,inner sep=0] at (0,0) {\includegraphics[width=4in]{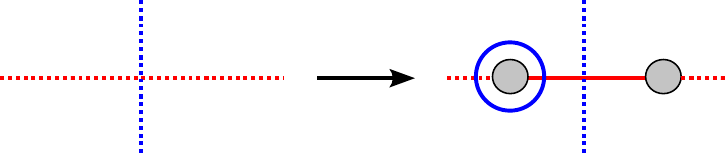}};
    \node[label=below:{\small{$a_i$}}] at (0,1.2){};
    \node[label=right:{\small{$b_j$}}] at (1.8,2.03){};
    \node[label=below:{\small{$a_i$}}] at (6.2,1.2){};
    \node[label=right:{\small{$b_j$}}] at (8,2.03){};
\end{tikzpicture}
\caption{Trading an intersection point of $a_i$ and $b_j$ with extra genus, and $\alpha$ and $\beta$ circles } \label{fig:stab}
\end{figure}

Then, add $\alpha$- and $\beta$-circles $\alpha_{g+1},\cdots,\alpha_{g+n-1}$ and $\beta_{g+1},\cdots,\beta_{g+n-1}$ such that 
\begin{itemize}
\item $\alpha_{g+i}$ and $\beta_{g+i}$ bound disks containing $a_i$ and $b_i$, respectively,
\item For any $1\le i,j \le n-1$, if $i\neq j$ then $\alpha_{g+i}\cap\beta_{g+j}=\emptyset$. Otherwise, $\alpha_{g+i}$ intersects $\beta_{g+i}$ in exactly $4$ points, as vertices of two bigons, one containing $z_i$ and another containing $w_i$.
\end{itemize}
See Figure \ref{fig:link-ddiag} for a local picture with $n=3$. Denote the final diagram of $L$ by $H$, that is
\[H=(\Sigma,\alphas=\{\alpha_1,\alpha_2,\cdots,\alpha_{g+n-1}\},\betas=\{\beta_1,\beta_2,\cdots,\beta_{g+n-1}\},\z=\{z_1,\cdots,z_n\},\w=\{w_1,\cdots,w_n\})\]

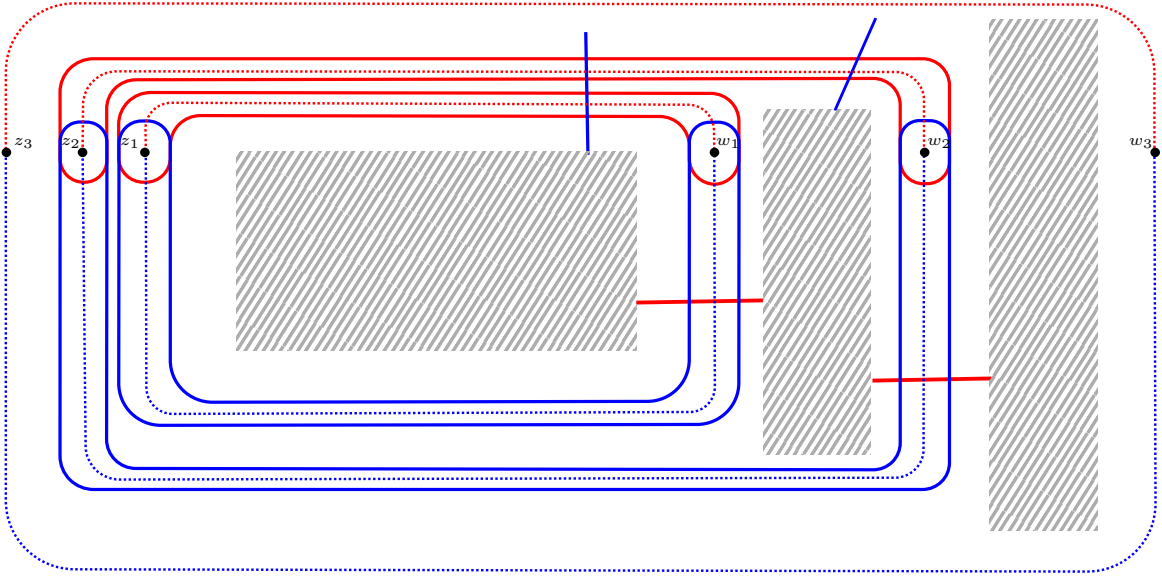
\begin{figure}[h]

\begin{tikzpicture}
    \node[anchor=south west,inner sep=0] at (0,0) {\includegraphics[width=6in]{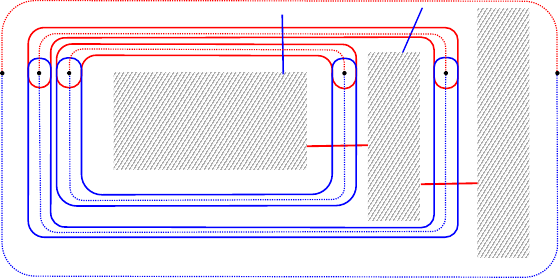}};
    \node at (1.7,5.5)[label={[yshift=-3pt]above: \tiny{$z_1$}}]{};
     \node at (0.92,5.5)[label={[yshift=-3pt]above: \tiny{$z_2$}}]{};
     \node at (0.3,5.5)[label={[yshift=-3pt]above: \tiny{$z_3$}}]{};
    \node at (9.57,5.5)[label={[yshift=-3pt]above: \tiny{$w_1$}}]{};
     \node at (12.35,5.5)[label={[yshift=-3pt]above: \tiny{$w_2$}}]{};
     \node at (15,5.5)[label={[yshift=-3pt]above: \tiny{$w_3$}}]{};


\end{tikzpicture}
\caption{In this figure $n=3$ and arcs $a_i$ and $b_i$ connecting $z_i$ and $w_i$ are dashed red and blue, respectively. Shaded areas are not necessarily disks and could include $\alpha$- or $\beta$-circles.}\label{fig:link-ddiag}
\end{figure}

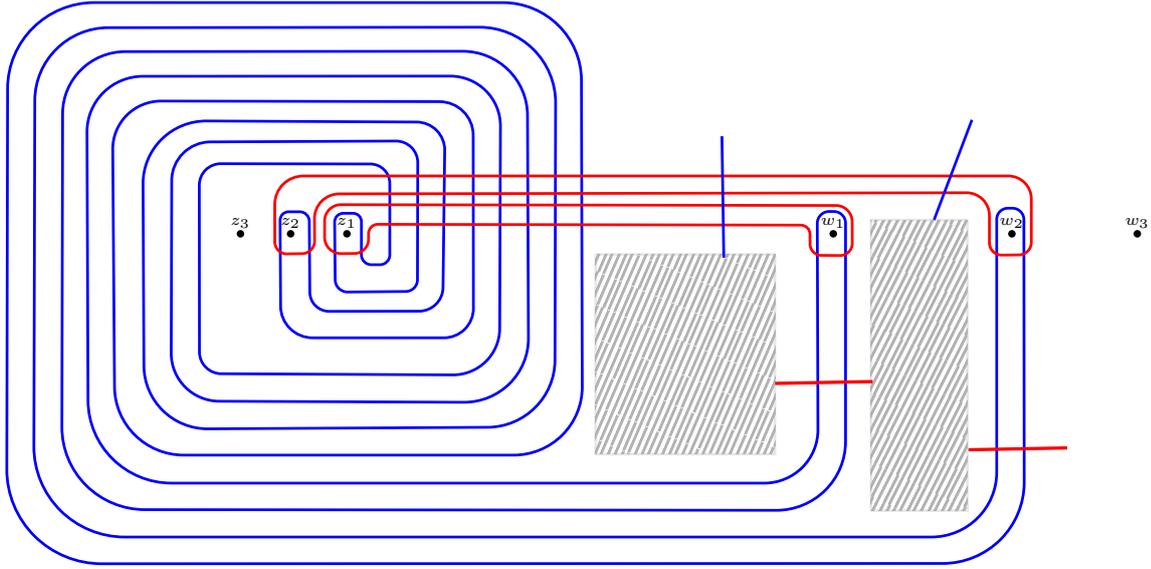
\begin{figure}[h]

\begin{tikzpicture}
    \node[anchor=south west,inner sep=0] at (0,0) {\includegraphics[width=5.5in]{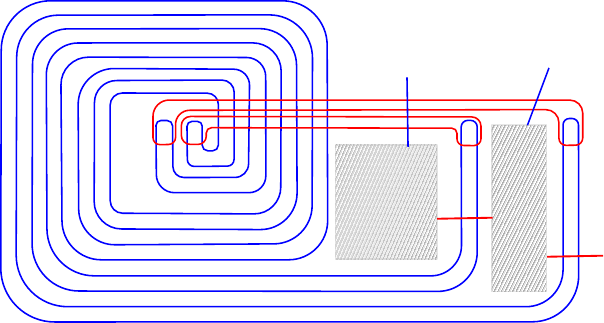}};
    \node at (4.5,4.4)[circle,fill,inner sep=1pt,label={[yshift=-3pt]above: \tiny{$z_1$}}]{};
     \node at (3.76,4.4)[circle,fill,inner sep=1pt,label={[yshift=-3pt]above: \tiny{$z_2$}}]{};
     \node at (3.1,4.4)[circle,fill,inner sep=1pt,label={[yshift=-3pt]above: \tiny{$z_3$}}]{};
    \node at (10.9,4.4)[circle,fill,inner sep=1pt,label={[yshift=-3pt]above: \tiny{$w_1$}}]{};
     \node at (13.25,4.4)[circle,fill,inner sep=1pt,label={[yshift=-3pt]above: \tiny{$w_2$}}]{};
     \node at (14.9,4.4)[circle,fill,inner sep=1pt,label={[yshift=-3pt]above: \tiny{$w_3$}}]{};


\end{tikzpicture}
\caption{Local figure for $m=2$.}\label{fig:cableddiag}
\end{figure}

A Heegaard diagram $H_m$ for $L_m$ is obtained from $H$ by twisting $\beta_{g+1},\beta_{g+2},\cdots,\beta_{g+n-1}$ around $z$-basepoints  $m$ times as in Figure \ref{fig:cableddiag}. We write
\[H_m=\left(\Sigma,\alphas=\{\alpha_1,\cdots,\alpha_{g+n-1}\},\betas_{m}=\{\beta_1,\cdots,\beta_g,\beta^{m}_{g+1},\cdots,\beta^m_{g+n-1}\},\z,\w\right).\]
 Every generator $\x$ of the link Floer chain complex $\cCFL(H_{m})$ is a $(g+n-1)$-tuple of intersection points and can be denoted as $\x=\{x_1,x_2,\cdots, x_{g+n-1}\}$ where $x_i\in \alpha_i$. Denote the set of generators for $\cCFL(H_m)$ by $\mathcal{G}(H_m)$. 


For every distinct $1\le i,j\le n-1$, $\beta^m_{g+j}$ intersects $\alpha_{g+i}$ in $4m$ points, while $\alpha_{g+i}\cap\beta^m_{g+i}$ consists of $4m+4$ points. We label these intersection points (similar to \cite{Licata}) as follows. First, we label the intersection points of $\alpha_{g+i}\cap \beta_{g+i}$ on the boundary of the bigons containing $w_i$ and $z_i$ basepoints by  $E_i,E_i'$ and $I_i,I_i'$, respectively, as in Figure \ref{fig:labelint}. The  $\ell$-th winding block of intersection points is the $2(n-1)\times 2(n-1)$  grid of intersection points that appears when we wind $\beta^{\ell-1}_{g+1},\cdots,\beta^{\ell-1}_{g+n-1}$ one more time about the $\z$ basepoints to obtain $\beta^{\ell}_{g+1},\cdots,\beta^{\ell}_{g+n-1}$, respectively. Then, we denote the four intersection points $\alpha_{g+i}\cap\beta_{g+j}$ in the $\ell$-th winding block by $G_{ij,\ell}$, $G_{ij,\ell}''$, $G_{ij,\ell}^r$ and $G_{ij,\ell}^l$ as in Figure \ref{fig:labelint}. 

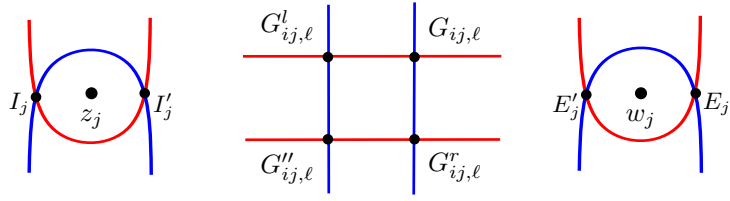
\begin{figure}[h]
\begin{tikzpicture}
    \node[anchor=south west,inner sep=0] at (0,0) {\includegraphics[width=3.5in]{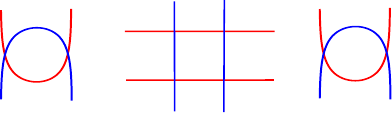}};
    \node[label=below:{\small{$I_j$}}] at (-0.1,1.6){};
    \node[label=below:{\small{$I_j'$}}] at (1.8,1.6){};
     \node[label=below:{\small{$E_j'$}}] at (7.1,1.6){};
    \node[label=below:{\small{$E_j$}}] at (9.08,1.6){};
    \node at (0.12,1.3)[circle,fill,inner sep=1.4pt]{};
    \node at (1.55,1.35)[circle,fill,inner sep=1.4pt]{};
 \node at (7.36,1.33)[circle,fill,inner sep=1.4pt]{};
    \node at (8.8,1.35)[circle,fill,inner sep=1.4pt]{};
    \node at (8.08,1.35)[circle,fill,inner sep=1.6pt,label=below: $w_j$]{};
     \node at (0.85,1.35)[circle,fill,inner sep=1.6pt,label=below: $z_j$]{};
    \node at (3.96,0.74)[circle,fill,inner sep=1.4pt,label=below left: $G_{ij,\ell}''$]{};
    \node at (3.96,1.83)[circle,fill,inner sep=1.4pt,label=above left: $G_{ij,\ell}^{l}$]{};
    \node at (5.1,1.83)[circle,fill,inner sep=1.4pt,label=above right: $G_{ij,\ell}$]{};
     \node at (5.1,0.74)[circle,fill,inner sep=1.4pt,label=below right: $G_{ij,\ell}^r$]{};
         
\end{tikzpicture}
\caption{labeling of intersection points between $\alpha_{g+1},\cdots,\alpha_{g+n-1}$ and $\beta_{g+1},\cdots,\beta_{g+n-1}$. }\label{fig:labelint}
\end{figure}

\begin{lemma}\label{lem:HDmovez}
For any $0\le \ell\le m$, let $H^\ell_{m+1}$ be the Heegaard diagram obtained from $H_{m+1}$ by moving the basepoints $\z$ to the $(\ell+1)-$th winding block as in Figure \ref{fig:movez}, and denote them by $\z^\ell=\{z^\ell_1,z^\ell_2,\cdots,z^\ell_n\}$. Then, $H^\ell_{m+1}$ is a Heegaard diagram for $L_{\ell}$.
\end{lemma}
\begin{proof}
It is straightforward to check that $H_{\ell}$ can be obtained from $H^\ell_{m+1}$ by isotopy on the curves $\alpha_{g+1},\cdots,\alpha_{g+n-1}$ and $\beta_{g+1},\cdots,\beta_{g+n-1}$. For instance, the local picture for $\ell=m-1$ (and $n=3$) is depicted in Figure \ref{fig:HD-movez}. In this figure all intersection points except for the four that are highlighted green can be removed by isotopy and the result is $H_{m-1}$.
\end{proof}

 Obviously, $\cCFL(H_{m+1}^\ell)$ and $\cCFL(H_{m+1})$ have the same set of generators, that is $\mathcal{G}(H_{m+1}^\ell)=\mathcal{G}(H_{m+1})$. For any $\x\in\mathcal{G}(H_{m+1})$ we denote the Alexander grading of $\x$ as a generator of $\cCFL(H_{m+1}^\ell)$ by $A^\ell(\x)=(A_1^\ell(\x),\cdots,A_{n}^\ell(\x))$.


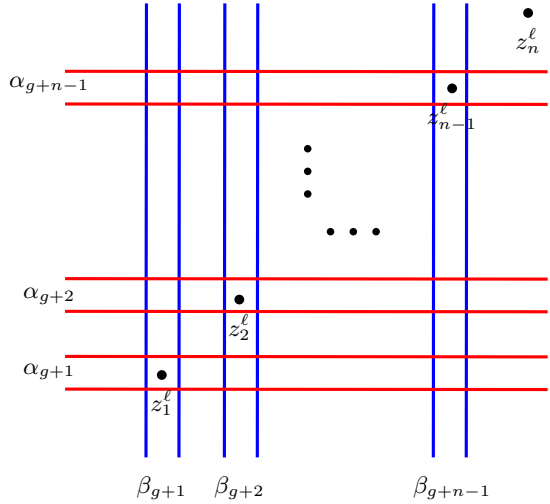
\begin{figure}[h]
\begin{tikzpicture}
    \node[anchor=south west,inner sep=0] at (0,0) {\includegraphics[width=2.5in]{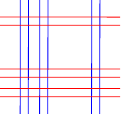}};
    \node[label=below:{\small{$\alpha_{g+1}$}}] at (-0.2,1.5){};
    \node[label=below:{\small{$\alpha_{g+2}$}}] at (-0.2,2.5){};
    \node[label=below:{\small{$\alpha_{g+n-1}$}}] at (-0.2,5.3){};

     \node[label=below:{\small{$\beta_{g+1}$}}] at (1.28,0){};
    \node[label=below:{\small{$\beta_{g+2}$}}] at (2.3,0){};
    \node[label=below:{\small{$\beta_{g+n-1}$}}] at (5.1,0){};
 
   \node at (1.28,1.1)[circle,fill,inner sep=1.3pt,label=below: {\small $z_1^{\ell}$}]{};
    \node at (2.3,2.1)[circle,fill,inner sep=1.3pt,label=below: {\small $z_{2}^{\ell}$}]{};
   \node at (5.1,4.9)[circle,fill,inner sep=1.3pt,label=below: {\small $z_{n-1}^{\ell}$}]{};
   \node at (6.1,5.9)[circle,fill,inner sep=1.3pt,label=below: {\small $z_{n}^{\ell}$}]{};

    \node at (3.5,3)[circle,fill,inner sep=1pt]{};
    \node at (3.8,3)[circle,fill,inner sep=1pt]{};
    \node at (4.1,3)[circle,fill,inner sep=1pt]{};
    \node at (3.2,3.5)[circle,fill,inner sep=1pt]{};
    \node at (3.2,3.8)[circle,fill,inner sep=1pt]{};
    \node at (3.2,4.1)[circle,fill,inner sep=1pt]{};

\end{tikzpicture}
\caption{$\ell$-th winding block and placement of $\z$ basepoints in the diagram $H_{m+1}^{\ell}$}\label{fig:movez}
\end{figure}

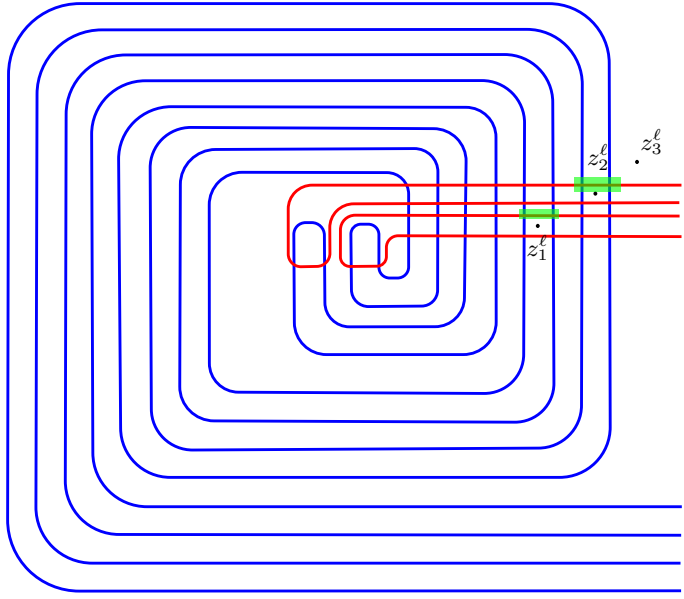
\begin{figure}[h]
\begin{tikzpicture}
 \node[anchor=south west,inner sep=0] at (0,0) {\includegraphics[width=3.5in]{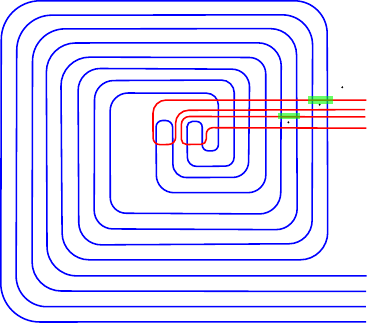}};
    \node[label=below:{\small{$z_1^{\ell}$}}] at (7,5){};
    \node[label=below:{\small{$z_2^{\ell}$}}] at (7.8,6.2){};
    \node[label=below:{\small{$z_3^{\ell}$}}] at (8.5,6.4){};
\end{tikzpicture}
\caption{Special case: $\ell=m-1$}\label{fig:HD-movez}\label{fig:HD-movez}
\end{figure}

Let $\G^{o}_{\ell}(H_{m+1})\subset\G(H_{m+1})$ denote the set of intersection points $\x$ that do not contain any points in the $i$-th winding block for $i\ge \ell+1$ and $\x\cap \left(\cup_{j=1}^{n-1}\{I_j,I_j'\}\right)=\emptyset$. Let $\mathcal{G}_{\ell}^{\iota}(H_{m+1})=\mathcal{G}(H_{m+1})\setminus \mathcal{G}^o_\ell(H_{m+1})$. 

\subsection{Stabilization of homology}

In this subsection, we prove the main result of Section \ref{sec: construction} modulo some technical lemmas whose proofs  are postponed to the next subsection.  

\begin{theorem}
\label{thm: HFL stabilization}
Let $C$ be the constant from Lemma \ref{lem:agrading2} below.
For any $\ovl{\ss}=(\ovl{s}_1,\ovl{s}_2,\cdots,\ovl{s}_n)$ with $\ovl{s}_i\ge C-m$ for all $i$ we have
\[\cHFL^{\stab}(L_m,\ovl{\ss})\cong \cHFL^{\stab}\left(L_{m+1},\ovl{\ss}\right)\]
In particular, for fixed $\ovl{\ss}$ and $m\gg 0$ the homology $\cHFL^{\stab}(L_m,\ovl{s})$ stabilizes. 
\end{theorem}

\begin{proof}
Let $s_i=\ovl{s}_i+\frac{(m+1)(n-1)}{2}$, $\widetilde{s}_i=\ovl{s}_i+\frac{m(n-1)}{2}$ and set $\ss=(s_1,s_2,\cdots,s_n)$ and $\widetilde{\ss}=(\widetilde{s}_1,\widetilde{s}_2,\cdots,\widetilde{s}_n)$. Note that $s_i-\widetilde{s}_i=\frac{n-1}{2}$.
As before, consider Heegaard diagrams $H_{m+1}$ and $\widetilde{H}_m:=H^{m}_{m+1}$ for $L_{m+1}$ and $L_m$, respectively. For simplicity, we denote the Alexander grading of $\cCFL(\widetilde{H}_m)$ by $\widetilde{A}=(\widetilde{A}_1,\cdots,\widetilde{A}_m)$ i.e. $\widetilde{A}_i=A^m_i$.

The chain complex $\cCFL(H_{m+1},\ss)$ is generated by $p_{\x}\x$ where the monomial $p_{\x}$ is equal to \[p_{\x}=\left(\prod_{i\in I_{\ss}(\x)}U_i^{A_i(\x)-s_i}\right)\left(\prod_{i\in J_{\ss}(\x)}V_i^{s_i-A_i(\x)}\right).\]
Here, $I_{\ss}(\x)=\{1\le i\le n\ |\ A_i(\x)>s_i\}$ and $J_{\ss}(\x)=\{1,2,\cdots,n\}\setminus I_{\ss}$. Similarly, $\cCFL(\widetilde{H}_{m},\tilde{\ss})$ is generated by $\widetilde{p}_{\x}\x$ where  
\[\widetilde{p}_{\x}=\left(\prod_{i\in I_{\widetilde{\ss}}(\x)}U_i^{\widetilde{A}_i(\x)-\widetilde{s}_i}\right)\left(\prod_{i\in J_{\widetilde{\ss}}(\x)}V_i^{\widetilde{s}_i-\widetilde{A}_i(\x)}\right)\]
for $I_{\widetilde{\ss}}(\x)=\{1\le i\le n\ |\ \widetilde{A}_i(\x)>\widetilde{s}_i\}$ and $J_{\widetilde{\ss}}(\x)=\{1,2,\cdots,n\}\setminus I_{\widetilde{\ss}}$.


 We define an isomorphism $F$ from $\cCFL(\widetilde{H}_{m},\widetilde{\ss})$ to $\cCFL(H_{m+1},\ss)$
by setting $F(\widetilde{p}_\x\x)=p_\x\x$. We need to check that $F$ is a chain map. Denote the differential on $\cCFL(\widetilde{H}_{m},\widetilde{\ss})$ and $\cCFL(H_{m+1},\ss)$ by $\widetilde{\partial}$ and $\partial$, respectively.
Assume $\varphi\in\pi_2(\x,\x')$ be a Maslov index one disk contributing nontrivially to $\partial \x$ and to $\widetilde{\partial}\x$. Then, 
the contributions of $\varphi$ to $\partial(p_{\x}\x)$ and $\widetilde{\partial}(\widetilde{p}_{\x}\x)$ are equal to $\UU^{n(\varphi)}(p_{\x'}\x')$ and  $\UU^{\widetilde{n}(\varphi)}(\widetilde{p}_{\x'}\x')$, respectively, where
\begin{equation}
\label{eq: def n(phi)}
\UU^{n(\varphi)}p_{\x'}=p_{\x}\prod_{i=1}^nU_i^{n_{w_i}(\varphi)}V_i^{n_{z_i}(\varphi)}\quad\quad \text{and}\quad\quad\UU^{\widetilde{n}(\varphi)}\widetilde{p}_{\x'}=\widetilde{p}_{\x}\prod_{i=1}^nU_i^{n_{w_i}(\varphi)}V_i^{n_{\widetilde{z}_i}(\varphi)} 
\end{equation}
in $R_{UV}$. Here, $\widetilde{z}_i=z_i^m$.  We need to prove $n(\varphi)=\widetilde{n}(\varphi)$, this would imply that $F$ is indeed a chain map.

Note that by Lemmas \ref{lem:topshift} and \ref{lem:agrading1} for any $\x\in\G^o_{m}(H_{m+1})$ we have that 
 $A_i(\x)-\widetilde{A}_i(\x)=\frac{n-1}{2}$ and so $p_{\x}=\widetilde{p}_{\x}$. Now we consider several cases:

1) If both $\x,\x'\in\G^o_{m}(H_{m+1})$ then $p_{\x}=\widetilde{p}_{\x}$,  $p_{\x'}=\widetilde{p}_{\x'}$ and $n_{z_i}(\varphi)=n_{\widetilde{z}_i}(\varphi)$. Thus, $n(\varphi)=\widetilde{n}(\varphi)$. 

2) Suppose $\x\in\G^{o}_{m}(H_{m+1})$ while $\x'\in \G^{\iota}_{m}(H_{m+1})$. Then, by Lemma \ref{lem:agrading2} 
\[\widetilde{A}_i(\x')\le C+\frac{m(n-1)}{2}-m\le \widetilde{s}_i\]
and $\widetilde{p}_{\x'}=\prod_{i=1}^nV_i^{\widetilde{s}_i-\widetilde{A}_i(\x')}$. Similarly, $p_{\x'}=\prod_{i=1}^nV_i^{s_i-A_i(\x')}$.
Now we can compare the total $U$-degrees in both sides of \eqref{eq: def n(phi)}, let $n_{U_i}(p_{\x})$ denote the exponent of $U_i$ in $p_{\x}$.
We have
$$
n(\varphi)=\sum_{i=1}^n(n_{w_i}(\varphi)+n_{U_i}(p_{\x}))\quad\quad\text{and}\quad\quad\widetilde{n}(\varphi)=\sum_{i=1}^n(n_{w_i}(\varphi)+n_{U_i}(\widetilde{p}_{\x}))
$$
and so $p_{\x}=\widetilde{p}_{\x}$ implies that $n(\varphi)=\widetilde{n}(\varphi)$.

3) If $\x\in\G^{\iota}_{m}(H_{m+1})$ and $\x'\in\G^{o}_{m}(H_{m+1})$, then an analogous argument implies that
\[n(\varphi)=\sum_{i=1}^n(n_{w_i}(\varphi)-n_{U_i}(p_{\x'}))\quad\quad\text{and}\quad\quad \widetilde{n}(\varphi)=\sum_{i=1}^n(n_{w_i}(\varphi)-n_{U_i}(\widetilde{p}_{\x'}))\]
and so $p_{\x'}=\widetilde{p}_{\x'}$ implies  that $n(\varphi)=\widetilde{n}(\varphi)$.

4) Finally, assume $\x,\x'\in\mathcal{G}_{m+1}^{\iota}(H_{m+1})$. In this case, $n(\varphi)=\sum_{i=1}^nn_{w_i}(\varphi)=\widetilde{n}(\varphi)$ and so we are done.

\end{proof}

\subsection{Estimates for the Alexander gradings}

In the following lemmas we  compare the Alexander degrees of the generators for different choices of $\z$ basepoints.

\begin{lemma}\label{lem:agrading1}
For any $\x,\x'\in\mathcal{G}^o_{\ell}(H_{m+1})$ we have $A(\x)-A(\x')=A^{i}(\x)-A^{i}(\x')$ for all $\ell\le i\le m$.
\end{lemma}
\begin{proof}
By Equation \ref{eq:relgradings}, we have $A(\x)-A(\x')=A^{i}(\x)-A^{i}(\x')$ if and only if for any $1\le j\le n$ and any $\varphi\in\pi_2(\x,\x')$ we have
$n_{z_j}(\varphi)=n_{z_j^i}(\varphi)$. On the other hand, as discussed in the proof of Lemma \ref{lem:HDmovez} the Heegaard diagram $H_{\ell}$ is obtained from $H_{m+1}^{\ell}$ by certain isotopies that do not move the intersection points in $\mathcal{G}^{o}_{\ell}(H_{m+1})$. So, every $\x\in\mathcal{G}^o_{\ell}(H_{m+1})$ has a canonical corresponding intersection point in $\mathcal{G}(H_{\ell})$. Moreover, every Whitney disk $\varphi\in\pi_2(\x,\x')$ comes from extending a disk between corresponding intersection points in $\mathcal{G}(H_{\ell})$. In fact, any such disk can be uniquely extended to a disk from $\x$ to $\x'$, and will have the same coefficient at $z_{j^{i}}$ for all $\ell\le i\le m$. For instance, see Figure \ref{fig:localcoef} for the local coefficient of such a disk when $n=3$ at $z_1,z_2,z_3$ and the $(m+1)$-th winding block. So, $n_{z_j}=n_{z_j^m}$. 
The general case is similar.


\begin{figure}[h]
\begin{tikzpicture}
    \node[anchor=south west,inner sep=0] at (0,0) {\includegraphics[width=3.5in]{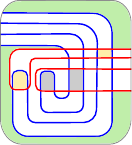}};

    \node[label=below:{\small{$a$}}] at (8.2,7.5){};
    \node[label=below:{\small{$b$}}] at (8.2,6.5){};
    \node[label=below:{\small{$a$}}] at (8.2,5.8){};
    \node[label=below:{\small{$c$}}] at (8.2,4.7){};
    \node[label=below:{\small{$a$}}] at (8.2,3.1){};
    
    \node[label=below:{\small{$d$}}] at (7.2,7.5){};
    \node[label=below:{\tiny{$d\!+\!b\!-\!a$}}] at (7.1,6.5){};
    \node[label=below:{\small{$d$}}] at (7.1,5.8){};
    \node[label=below:{\tiny{$d\!+\!c\!-\!a$}}] at (7.1,4.7){};
 \node[text width=3mm, font=\small] at (7.1,2.8){$d$};

    \node[label=below:{\small{$a$}}] at (6.1,7.5){};
    \node[label=below:{\small{$b$}}] at (6.1,6.5){};
    \node[label=below:{\small{$a$}}] at (6.1,5.8){};
    \node[label=below:{\small{$c$}}] at (6.1,4.7){};
    \node[label=below:{\small{$a$}}] at (6.1,3.1){};
    
   \node[label=below:{\small{$e$}}] at (5,7.3){};
    \node[label=below:{\tiny{$e\!+\!b\!-\!a$}}] at (5.1,6.5){};
    \node[label=below:{\small{$e$}}] at (5.1,5.8){};
    \node[label=below:{\tiny{$e\!+\!c\!-\!a$}}] at (5.15,4.7){};
 \node[text width=3mm, font=\small] at (5.1,2.8){$e$};
 \node[label=below:{\small{$b$}}] at (4.2,6.5){};
    \node[label=below:{\small{$a$}}] at (4.2,5.8){};
    \node[label=below:{\small{$c$}}] at (4.2,4.7){};
    \node[label=below:{\small{$a$}}] at (4.2,3.7){};
  \node[label=below:{\tiny{$e\!+\!c\!-\!a$}}] at (3.22,4.7){};
 \node[label=below:{\tiny{$d\!+\!b\!-\!a$}}] at (1.38,4.7){};
\end{tikzpicture}

\caption{Local coefficients of a disk  $\varphi\in\pi_2(\x,\x')$ for $n=3$; the regions containing $z_3$ and $z_{3}^{m}$ are colored green, the regions containing $z_2$ and $z_2^{m}$ are colored yellow, and the regions containing $z_1$ and $z_1^m$ are colored gray.}\label{fig:localcoef}
\end{figure}
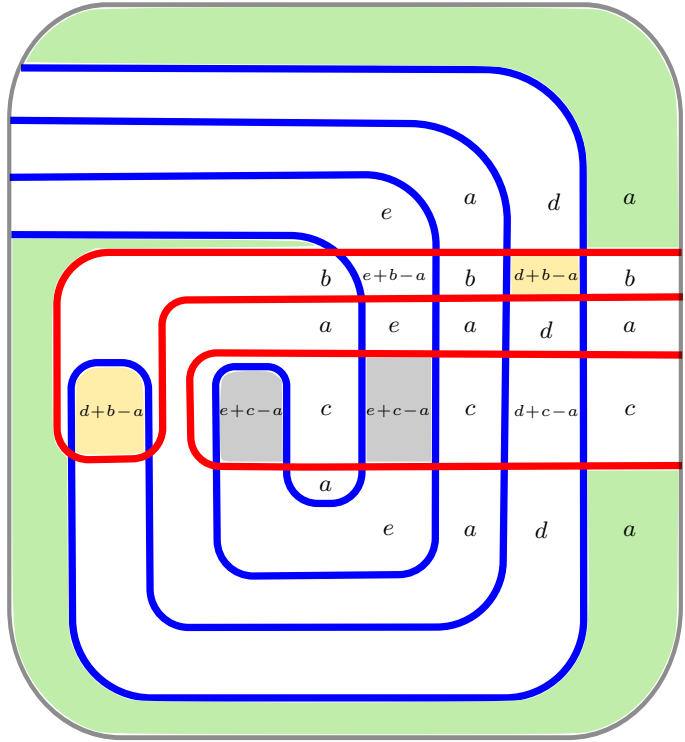

\end{proof}

\begin{lemma}\label{lem:topshift}
For any intersection point of the form $\x=(\x',\{E_1,E_2,\cdots,E_{n-1}\})$ in $\G(H_{m+1})$ we have
\[A^{\ell+1}(\x)-A^{\ell}(\x)=\left(\frac{n-1}{2},\cdots,\frac{n-1}{2}\right)\]
for all $0\le \ell\le m$ where $A^{m+1}(\x):=A(\x)$.
\end{lemma}

\begin{proof}

 Let $H_G$ denote the underlying Heegaard diagram for $G_{L,D}$. Following the discussions in \cite[Section 5.4]{Zemke2} if 
 this property holds, then it will still hold if we change $H_G$ (and so $H_{m+1}$) by a Heegaard move. Therefore, it is enough to prove it for one Heegaard diagram $H_G$. 

First, we prove this for the special case where $L$ is the unlink and so $L_{\ell}=T(n,\ell n)$. In this case, the corresponding graph is trivial (i.e. embeds in $S^2$), and thus it has a genus zero Heegaard diagram with no $\alpha$- and $\beta$-circles, that is $H_G=(S^2,z,w_1,\cdots,w_n)$. Starting with this special diagram, the corresponding diagram for $T(n,(m+1)n)$ has genus zero and is of the form depicted in Figure \ref{fig:cableddiag}. In \cite{Licata}, Licata shows that for the special case of $\ell=1$ i.e. $T(n,n)$, the Alexander grading $A^1(E_1,\cdots,E_{n-1})=\left(\frac{n-1}{2},\frac{n-1}{2},\cdots,\frac{n-1}{2}\right)$. In general, a similar computation of relative Alexander gradings shows that $A^{\ell}_i(E_1,\cdots,E_{n-1})\ge A^{\ell}_i(\x)$ for any other intersection point $\x$, so considering $\widehat{\HFL}(T(n,\ell n))$ and the Alexander polynomial of $T(n,\ell n)$ we have $\{E_1,\cdots, E_{n-1}\}$ is the generator of $\widehat{\HFL}(T(n,\ell n),\ss)$ for $\ss=\left(\frac{\ell(n-1)}{2},\frac{\ell(n-1)}{2},\cdots,\frac{\ell(n-1)}{2}\right)$. So, $A^{\ell}(E_1,E_2,\cdots,E_{n-1})=\left(\frac{\ell(n-1)}{2},\frac{\ell(n-1)}{2},\cdots,\frac{\ell(n-1)}{2}\right)$, and the claim holds for the unlink.


Let  $O_n\subset S^3$ be the $n$-component unlink in $S^3$ and $D'$ be a small disk whose interior intersects each component of $O_n$ is exactly one point with positive sign. There is a framed link $\mathbb{S}$ in the complement of $O_n$ and disjoint from $D'$ such that surgery on $S^3$ along $\mathbb{S}$ is diffeomorphic to $S^3$ and this diffeomorphism maps $(O_n,D')$ to $(L,D)$. Attaching $2$-handles along $\mathbb{S}$ gives a cobordisms from $O_n$ to $L$ and also $G_{O_n,D'}$ to $G_{L,D}$. So, we may consider a Heegaard triple $\mathcal{T}_G=(\Sigma,\alphas,\betas,\gammas,z,\w)$ corresponding to the cobordism between graphs that can be \emph{modified} to a triple for the cobordism between links. More precisely,  $H_{\alphas\betas}=(\Sigma,\alphas,\betas,z,\w)$, $H_{\alphas\gammas}=(\Sigma,\alphas,\gammas,z,\w)$ and $H_{\betas\gammas}=(\Sigma,\betas,\gammas,z,\w)$ are Heegaard diagrams for $G_{O_n,D'}$, $G_{L,D}$ and the trivial graph in some connected sum of $S^1\times S^2$s. Moreover, $\w$ and $z$ are in the same connected component of $\Sigma\setminus(\betas\cup\gammas)$, and a Heegaard triple $\mathcal{T}$ for the link cobordism from $O_n$ to $L$ is obtained from $\mathcal{T}_G$ by replacing $z$ with $z_1,\cdots,z_n$ and adding $\alpha_{g+1},\cdots,\alpha_{g+n-1}$, $\beta_{g+1},\cdots,\beta_{g+n-1}$ and $\gamma_{g+1},\cdots,\gamma_{g+n-1}$ similar to Figure \ref{fig:cableddiag}. Note that each $\gamma_{g+i}$ is a small Hamiltonian isotope of $\beta_{g+i}$ and intersects it in exactly two points. 

Note that twisting $\beta_{g+1},\cdots,\beta_{g+n-1}$ (and correspondingly $\gamma_{g+1},\cdots,\gamma_{g+n-1}$) around $\z$ basepoints $(m+1)$-times gives a Heegaard triple $\mathcal{T}_{m+1}$ for the corresponding cobordism from $T(n,(m+1)n)$ to $L_{m+1}$. Moreover, for any $1\le \ell\le m$ moving the basepoints $\z$ to the $(\ell+1)-$th winding block will result a Heegaard triple $\mathcal{T}_{m+1}^{\ell}$ for the corresponding cobordism from $T(n,\ell n)$ to $L_{\ell}$.

Assume $\varphi'\in\pi_2(\y',\theta',\x')$ is a triangle in  $\mathcal{T}_{G}$ connecting intersection points $\y'\in\mathcal{G}(H_{\alphas\betas})$, $\theta'\in\mathcal{G}(H_{\betas\gammas})$ and $\x'\in \mathcal{G}(H_{\alphas\gammas})$. Let $\x=(\x',\{E_1,E_2,\cdots, E_{n-1}\})$, $\theta=(\theta',\{\theta_1^+,\cdots,\theta_{n-1}^+\})$ and $\y=(\y',\{\bar{E}_1,\bar{E}_2,\cdots,\bar{E}_{n-1}\})$ where $\theta_j^+, E_j$ and $\bar{E}_j$ are depicted in Figure  \ref{fig:triangles}. Corresponding to $\varphi'$ there is a triangle $\varphi\in\pi_2(\y,\theta,\x)$ obtained by adding small triangles on $\Sigma$ connecting $E_j, \theta_j^+$ and $\bar{E}_j$ as in  Figure \ref{fig:triangles}.

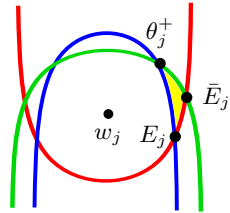
\begin{figure}[h]
\begin{tikzpicture}
    \node[anchor=south west,inner sep=0] at (0,0) {\includegraphics[width=1in]{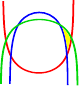}};

   \node at (1.98,1.97)[circle,fill,inner sep=1.5pt,label=above: {\small $\theta_j^+$}]{};
    \node at (2.18,1)[circle,fill,inner sep=1.5pt,label={[xshift=3pt]left: {\small $E_j$}}]{};
   \node at (2.33,1.52)[circle,fill,inner sep=1.5pt,label=right: {\small $\bar{E}_j$}]{};

   \node at (1.3,1.3)[circle,fill,inner sep=1.3pt,label=below: {\small $w_j$}]{};

\end{tikzpicture}

\caption{Intersection points $\theta_j^+$, $E_j$ and $\bar{E}_j$ and the small triangular domain connecting them, colored yellow}\label{fig:triangles}
\end{figure}

Then, by \cite[Formula 5.2]{Zemke2} we have
\[A^\ell_i(\x)=A^\ell_i(\y)+A^{\ell}_i(\theta)+n_{w_i}(\varphi)-n_{z^{\ell}_i}(\varphi)+\frac{\langle c_1(\s_{\w}(\varphi)),[\hat{F}^{\ell}_i]\rangle-[\hat{F}^{\ell}]\cdot[\hat{F}^{\ell}_i]}{2}.\]
 Here, $F_i^{\ell}$ denotes the component of $F^{\ell}$ with boundary on the $i$-th components of $T(n,\ell n)$ and $L_{\ell}$ i.e. the components containing $z_i$ and $w_i$. Moreover, $\hat{F}^{\ell}=\amalg_i\hat{F}^{\ell}_i$ is the result of capping $F^{\ell}$ with a Seifert surface for $T(n,\ell n)$ in $S^3$ and a Seifert surface for $L_{\ell}$ in $S^3(\mathbb{S})$. 
Note that $A_i^{\ell}(\theta)=A_i^{\ell+1}(\theta)$ and $n_{z_i^{\ell}}(\varphi)=n_{z_i^{\ell+1}}(\varphi)$. Therefore, 
\[A_i^{\ell+1}(\x)-A_i^{\ell}(\x)=A_i^{\ell+1}(\y)-A_i^{\ell}(\y)+\frac{\langle c_1(\s_{\w}(\varphi)),[\hat{F}^{\ell+1}_i]\rangle-[\hat{F}^{\ell+1}]\cdot[\hat{F}^{\ell+1}_i]}{2}-\frac{\langle c_1(\s_{\w}(\varphi)),[\hat{F}^{\ell}_i]\rangle-[\hat{F}^{\ell}]\cdot[\hat{F}^{\ell}_i]}{2}.\]

Note that changing $\ell$ will change the surface $\hat{F}_{i}^{\ell}$ but it  will not change its homology class i.e. $[\hat{F}_i^{\ell}]$. That is because $[\hat{F}_i^{\ell}]$ depends on the linking number of the $i$-th components of $T(n,\ell n)$ with $\mathbb{S}$ which does not change by changing $\ell$ as the disk $D$ is disjoint from $\mathbb{S}$. Consequently, both 
\[\frac{\langle c_1(\s_{\w}(\varphi)),[\hat{F}^{\ell+1}_i]\rangle-\langle c_1(\s_{\w}(\varphi)),[\hat{F}^{\ell}_i]\rangle}{2}\ \ \ \text{and}\ \ \ \frac{[\hat{F}^{\ell+1}]\cdot[\hat{F}^{\ell+1}_i]-[\hat{F}^{\ell}]\cdot[\hat{F}^{\ell}_i]}{2}\]
vanish. 
Thus, $A_i^{\ell+1}(\x)-A_i^{\ell}(\x)=A_i^{\ell+1}(\y)-A_i^{\ell}(\y)$. For such generators $\y$, we have shown the following
$$A_i^{\ell+1}(\y)-A_i^{\ell}(\y)=\dfrac{n-1}{2}$$
for all $1\leq i\leq n$. Thus, the claim holds for $L$.

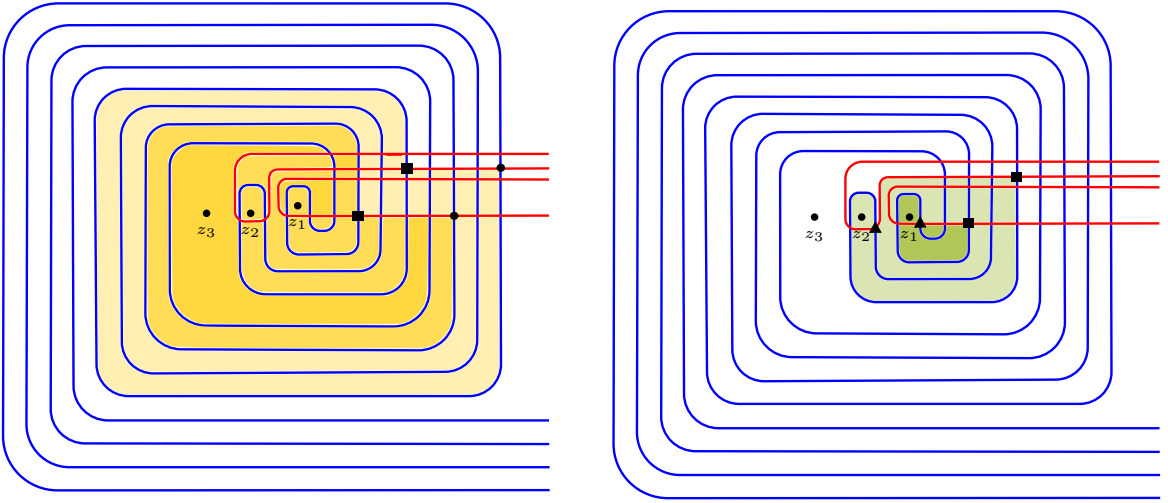
\begin{figure}[h]
\begin{tikzpicture}
    \node[anchor=south west,inner sep=0] at (0,0) {\includegraphics[width=6.0in]{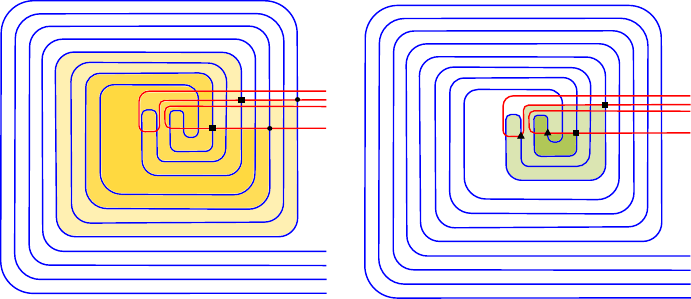}};
      \node at (3.9,3.9)[circle,fill,inner sep=1pt,label=below: {\tiny $z_1$}]{};
    \node at (3.28,3.8)[circle,fill,inner sep=1pt,label=below: {\tiny $z_{2}$}]{};
   \node at (2.7,3.8)[circle,fill,inner sep=1pt,label=below: {\tiny $z_{3}$}]{};
 \node at (11.95,3.75)[circle,fill,inner sep=1pt,label=below: {\tiny $z_1$}]{};
    \node at (11.32,3.75)[circle,fill,inner sep=1pt,label=below: {\tiny $z_{2}$}]{};
   \node at (10.7,3.75)[circle,fill,inner sep=1pt,label=below: {\tiny $z_{3}$}]{};

   \end{tikzpicture}
\caption{Suppose $n=3$ and $m=1$. In the Heegaard diagram $H_{2}$, if $\x$ is the intersection point represented by the black triangles in the winding region, then, $\x^1$ is obtained from $\x$ by replacing black triangles with black dots. Concatenating the Whitney with the  yellow shaded domain (left) with the Whitney disk with the green shaded domain (right) we get a Whitney disk in $\pi_2(\x^1,\x)$.
}\label{fig:disks}
\end{figure}

\end{proof}

\begin{lemma}\label{lem:agrading2} There exists a constant $C$ independent of $m$ and $\ell$ such that for any $\x\in\mathcal{G}_{\ell}^{\iota}(H_{m+1})$ 
\[{A}^\ell_i(\x)\le C+\frac{\ell(n-1)}{2}-\ell\]
for all $i$. 
\end{lemma}

\begin{proof}
Suppose $\x\in\G^\iota_\ell(H_{m+1})$. Let $\x^1\in\mathcal{G}^{o}_1(H_{m+1})$ be the intersection point obtained from $\x$ by replacing every $I_i$, $I_i'$ and $G^*_{ij,\ell}$ with $\ell>1$ in $\x$ by $G_{ii,1}^l$, $G^r_{ii,1}$ and $G^*_{ij,1}$, respectively. Then, there is a disk $\varphi\in\pi_2(\x^1,\x)$ whose domain is a union/concatenation of disk domains similar to the one in Figure \ref{fig:disks}, such that for all $1\le i\le n$
\[n_{w_i}(\varphi)=0 \quad\quad\text{and}\quad\quad n_{z_i^\ell}(\varphi)\ge \ell-1.\]
Thus, $A_i^\ell(\x)\le A_i^\ell(\x^1)-(\ell-1)$.

Consider an intersection point of the form $\y=(\y',\{E_1,E_2,\cdots, E_{n-1}\})$. By Lemma \ref{lem:agrading1} we have
\[A_i^\ell(\x^1)-A_i^\ell(\y)=A_i^1(\x^1)-A_i^1(\y).\]
Note that $A_i^1(\x^1)-A_i^1(\y)$ is equal to the $i$-th relative Alexander grading of the corresponding points of $\x^1$ and $\y$ in the Heegaard diagram $H_1$ for  $L_1$. So there is a constant $C'$ independent of $m$ such that
\[A_i^1(\x^1)-A_i^1(\y)\le C'\]
for all $\x$.  Consequently, $A_i^\ell(\x)\le A_i^\ell(\y)+C'-(\ell-1)$. Then, by Lemma \ref{lem:topshift}, $A_i^\ell(\y)=A_i^1(\y)+\frac{(\ell-1)(n-1)}{2}$ and so the claim holds by setting
\[C=C'+\max_{\y',i}\left\{A_i^1(\y)\right\}-\frac{n-3}{2}.\]
\end{proof}

\section{Colored knot Floer homology of the unknot}

In this section, we compute the colored knot Floer homology of the unknot.  

\subsection{Homology of $T(n,mn)$} The first step is computing the link Floer homology of the torus link $T(n,nm)$. Since $m$ is fixed in this subsection, we will denote $\cc_m=\frac{m(n-1)}{2}$ by $\cc$ to simplify notations. 

\begin{theorem}
\label{thm: Tnm}
The homology $\cHFL(T(n,mn))$ as a module over $R_{UV}$ is generated by $Y_0,\ldots,Y_{m(n-1)}$ with Alexander and Maslov degrees 
$$
A(Y_i)=(\cc-i,\ldots,\cc-i),\ \gr_{\w}(Y_i)=-(q+1)(qm+2r),\ \ i=qm+r,\ 0\le r\le m-1.
$$
Furthermore, the relations between $Y_i$ are spanned by $U_IY_i=V_{\overline{I}}Y_{i+1}$ for all subsets $I\subset\{1,2,\cdots,n\}$ with $q+1$ elements. Here, $\overline{I}=\{1,2,\cdots,n\}\setminus I$, $U_I=\prod_{j\in I}U_i$ and $V_{\overline{I}}=\prod_{j\in \overline{I}}V_j$.
\end{theorem}


\begin{proof}
The proof is similar to the proof of \cite[Theorem 7.3]{BLZ}. By  \cite[Lemma 7.1]{BLZ}, to write down the generators we need to compute the $h$-function for $T(n,mn)$.  Since the Alexander polynomial for $T(n,mn)$ is symmetric in $t_1,\ldots,t_n$, the $h$-function is symmetric in $s_1,s_2,\cdots,s_n$ so we consider the case $s_1\le s_2\le\cdots \le s_n$. By \cite[Theorem 4.3]{GH} for $s_1\le s_2\le\cdots \le s_n$ we get
\begin{equation}
\label{eq: h function cable}
h(s_1,\ldots,s_n)=h_0(s_1-\cc)+h_0(s_2-\cc+m)+\ldots+h_0(s_n-\cc+m(n-1))
\end{equation}
where  
$$
h_0(x)=\begin{cases}
0 & \mathrm{if}\ x\ge 0\\
-x & \mathrm{if}\ x<0
\end{cases}
$$
In particular, for $\cc\le s_1\le \ldots\le s_n$ we get $h(s_1,\ldots,s_n)=0$. Suppose that $s=\cc-i$ where $i=mq+r\ge 0$ then (assuming $i\le m(n-1)=2\cc)$ we get
$$
s-\cc\le 0, s-\cc+m\le 0,\ldots,s-\cc+qm\le 0\ \mathrm{but}\ s-\cc+(q+1)m>0
$$
so 
$$
h(s,\ldots,s)=(q+1)i-\frac{1}{2}q(q+1)m=\frac{1}{2}(q+1)(qm+2r).
$$
We claim that 
\begin{equation}
\label{eq: change h diagonal}
h(s-1,\ldots,s-1)-h(s,\ldots,s)=q+1.
\end{equation}
 If $r<m-1$ then going from $s$ to $s-1$ changes $i$ to $i+1$ and $r$ to $r+1$ while preserving $q$, so \eqref{eq: change h diagonal} is clear. If $r=m-1$ then $i=qm+(m-1),i+1=(q+1)m+0$ and
$$
h(s,\ldots,s)=\frac{1}{2}(q+1)(qm+2m-2),\ h(s-1,\ldots,s-1)=\frac{1}{2}(q+2)(q+1)m
$$
hence \eqref{eq: change h diagonal} holds as well.

Furthermore, one can check that
$$
h(\underbrace{s-1,\ldots,s-1}_{k},s,\ldots,s)=
\begin{cases}
h(s,\ldots,s)+k & k\le q+1\\
h(s-1,\ldots,s-1)=h(s,\ldots,s)+q+1 & k\ge q+1.
\end{cases}
$$
Let $Y_i$ be the generator of the $\F[\UU]$-tower in Alexander grading $(\cc-i,\ldots,\cc-i)$. It will have Maslov grading  
\[-2h(\cc-i,\ldots,\cc-i)=-(q+1)(qm+2r)\]
and by \cite[Lemma 7.1]{BLZ}, $Y_0,Y_1,\cdots,Y_{m(n-1)}$ generate $\cHFL(T(n,mn))$.

By \cite[Lemma 7.1]{BLZ}, the relations are spanned by the relations between $Y_a$ and $Y_b$ for any $0\leq a\leq b \leq (n-1)m$ which are the ones spanned by $(U_iV_i+U_jV_j)Y_a$ and $(U_iV_i+U_jV_j)Y_b$ as well as sums
$$U^I Y_a+V^J Y_b$$
ranging over sequences of nonnegative integer $I=(I_1,\cdots,I_n)$ and $J=(J_1,\cdots,J_n)$ such that 
$$I+J=(b-a, \cdots, b-a) \text{ and } |I|_{L^1}:=\sum_{k=1}^n I_k=(q_b+1)(q_b m+2r_b)/2-(q_a+1)(q_a m+2r_a)/2$$
where $a=q_a m+r_a$ and $b=q_b m+r_b$. Here, $U^I=\prod_{k}U_k^{I_k}$ and $V^J=\prod_kV_k^{J_k}$.

If $a=i$ and $b=i+1$, then $(q_b+1)(q_b m+r_b)/2-(q_a+1)(q_a m+r_a)/2=q+1$ and so these relations are exactly the ones described in the statement.


Suppose $b>a+1$. We show that the relations between $Y_a$ and $Y_b$ are in the span of the relations between consecutive $Y_i$ and $Y_{i+1}$. We prove this claim by induction on $b$. First, we show that there is a tuple $I'=(I'_1,\cdots,I'_{n})$ such that for all $1\leq i\leq n$, $0\leq I'_i\leq I_i$, and
$$|I'|_{L^1}=q_{b-1}+1, \quad |I'|_{L^\infty}=\max\{I'_1,\cdots,I'_{n}\}\leq 2 \text{ and } |I-I'|_{L_{\infty}}\le b-a-1$$

For $k\in \mathbb{N}$, we write
$$a_k:=\# \{i: I_i\geq k\}.$$
Then 
$$\sum_{k=1}^{b-a} a_k=|I|_{L^1}=(q_b+1)(q_b m+2r_b)/2-(q_a+1)(q_a m+2r_a)/2, \quad \text{ and } a_1\geq a_2\geq \cdots \geq a_{b-a},$$
and so by Equation (\ref{eq: change h diagonal}) $\sum_{k=1}^{b-a}a_k=\sum_{i=a}^{b-1}(q_i+1)$. Consequently, the number of entries $I_i$ that are equal to $b-a$ is at most $q_{b-1}+1$. So, such $I'$ exists if and only if $a_1+a_2\geq q_{b-1}+1$. Suppose to the contrary that 
$a_1+a_2< q_{b-1}+1$. Then $a_2<(q_{b-1}+1)/2$ and therefore

$$\sum_{i=a}^{b-1}(q_i+1)=\sum_{k=1}^{b-a} a_k=a_1+a_2+\sum_{k=3}^{b-a}a_k< q_{b-1}+1+(b-a-2) \dfrac{q_{b-1}+1}{2}=(b-a)\frac{q_{b-1}+1}{2}.$$

On the other hand, $q_{b-1}+1\le q_{b+a-1}+1\le (q_{a+i}+1)+(q_{b-i-1}+1)$ implies that 
\[\sum_{i=a}^{b-1}(q_i+1)\ge \frac{b-a}{2}(q_{b-1}+1)\]
which is a contradiction.

Therefore such a tuple $I'$ exists. By induction, for $J'=(b-a-1, \cdots, b-a-1)-(I-I')$, the relation 
$$U^{I-I'}Y_a=V^{J'} Y_{b-1}$$
is in the span of the claimed relations between consecutive generators. Then, this relation implies that 
$$U^I Y_a=U^{I'} U^{I-I'}Y_a=U^{I'} V^{J'}Y_{b-1}.$$

If $|I'|_{L^\infty}=1$, the relations between $Y_{b-1}$ and $Y_b$ imply that
$$U^I Y_a=U^{I'} V^{J'}Y_{b-1}=V^J Y_{b}.$$
Otherwise, for any  index $i$ with  $I'_i=2$, we observe that $(I-I')_i<b-a-1$, hence $J'_i>0$. Then $U^{I'} V^{J'}$ contains a factor of $U_iV_i$. Since, $|I'_i|_{L_1}=q_{b-1}+1\le n$ there is an index $j$ with $I'_j=0$. We will trade the $U_iV_i$ factor with $U_jV_j$.  By repeating this process enough times, we obtain another tuple $I''$ from $I'$ such that $|I''|_{L^1}=|I'|_{L^1}$ and $|I''|_{L^\infty}=1$. That is
$$U^{I'} V^{J'} Y_{b-1}=U^{I''}V^{J''}V_{b-1}=V^{J} Y_{b}.$$
\end{proof}

\begin{lemma}
\label{lem: stable range} Let $\ss=(s_1,s_2,\cdots,s_{n})$ and $\min\{s_1,s_2,\cdots,s_n\}\ge \cc-m$. The $\F[\UU]$-tower $\cHFL(T_{n,mn},\ss)$ is generated by
\begin{itemize}
\item[(a)]  $V^{s_1-\cc}_1\cdots V^{s_n-\cc}_nY_0$ if $\min(s_1,\ldots,s_n)\ge \cc$,
\item[(b)]  $V^{s_1-\cc+k}_1\cdots V^{s_n-\cc+k}_nY_k$ if $\min(s_1,\ldots,s_n)=\cc-k$ for $0\le k\le m$.
\end{itemize}
\end{lemma}

\begin{proof}
By symmetry assume $s_1\le \ldots\le s_n$, and so $\min(s_1,\ldots,s_n)=s_1$. 

(a) If $s_1\ge \cc$, then $s_i-\cc+(i-1)m\ge 0$ for all $i\geq 1$ and so by Equation \eqref{eq: h function cable} $h(s_1,s_2,\cdots,s_n)=0$. Therefore, the generator of $\cHFL(T_{n,mn},\ss)$ has Maslov grading equal to $0$. On the other hand, $$
\gr_{\w}(V^{s_1-\cc}_1\cdots V^{s_n-\cc}_nY_0)=\gr_{\w}(Y_0)=0,\ A(V^{s_1-\cc}_1\cdots V^{s_n-\cc}_nY_0)=(s_1,\ldots,s_n),
$$
 so
$V^{s_1-\cc}_1\cdots V^{s_n-\cc}_nY_0$ generates $\cHFL(T_{n,mn},\ss)$.



(b) If $\cc-m\le s_1\le \cc$, then $s_1-\cc\le 0$, while $s_i-\cc+(i-1)m\ge 0$ for $i>1$. Therefore, by Equation \eqref{eq: h function cable} $h(s_1,\ldots,s_n)=\cc-s_1=k$ and so the generator of $\cHFL(T_{n,mn},\ss)$ has Maslov grading $-2k$.
On the other hand, for $k\le m$ we get  
$$
\gr_{\w}(V^{s_1-\cc+k}_1\cdots V^{s_n-\cc+k}_nY_k)=\gr_{\w}(Y_k)=-2k,\ A(V^{s_1-\cc+k}_1\cdots V^{s_n-\cc+k}_nY_k)=(s_1,\ldots,s_n),
$$
so $V^{s_1-\cc+k}_1\cdots V^{s_n-\cc+k}_nY_k$ generates $\cHFL(T_{n,mn},\ss)$.
\end{proof}

\begin{lemma}
\label{lem: stable range iso} For any $\overline{\ss}=(\ovl{s}_1,\ovl{s}_2,\cdots,\ovl{s}_n)$ with $\min(\ovl{s}_1,\ldots,\ovl{s}_n)\ge -m$, the map
$$\phi_0: \cHFL^{\stab}\left(T(n,mn),\ovl{\ss}\right)\to \cHFL^{\stab}\left(T(n,(m+1)n),\ovl{\ss}\right)
$$ 
is an isomorphism.
\end{lemma}

\begin{proof}
The map $\phi_0$ has Alexander degree $\left(\frac{n-1}{2},\ldots,\frac{n-1}{2}\right)$ and Maslov degree 0. Since $T(n,mn)$ and $T(n,(m+1)n)$ are both L-space links and $\phi_0$ is injective,  it is sufficient to check that the Maslov degrees of the generators of the two $\F[\UU]$-towers $\cHFL^{\stab}\left(T(n,mn),\ovl{\ss}\right)$ and $\cHFL^{\stab}\left(T(n,(m+1)n),\ovl{\ss}\right)$ match.

If $\min(\ovl{s}_1,\ldots,\ovl{s}_n)\ge 0$, then by part (a) of Lemma \ref{lem: stable range} the generators of $\cHFL^{\stab}\left(T(n,mn),\ovl{\ss}\right)$ and $\cHFL^{\stab}\left(T(n,(m+1)n),\ovl{\ss}\right)$ have Maslov degree $0$. If $\min(\ovl{s}_1,\ldots,\ovl{s}_n)\ge -k$ for some $0<k\le m$, by part (b) of Lemma  \ref{lem: stable range}, the generators of $\cHFL^{\stab}\left(T(n,mn),\ovl{\ss}\right)$ and $\cHFL^{\stab}\left(T(n,(m+1)n),\ovl{\ss}\right)$ have Maslov degree $-2k$. So $\phi_0$ is an isomorphism in both cases.

\end{proof}

\begin{corollary}
\label{cor: homology stabilize}
For any $\overline{\ss}=(\ovl{s}_1,\ovl{s}_2,\cdots,\ovl{s}_n)$ with $\min(\ovl{s}_1,\ldots,\ovl{s}_n)\ge -m$
\[\HD_n(O,\ovl{\ss})\cong \cHFL^{\stab}(T(n,mn),\ovl{\ss}).\]

\end{corollary}

\begin{proof}
It follows from the definition of the colimit: suppose that $Y\in \cHFL^{\stab}(T(n,m'n),\ovl{\ss})$ for some $m'$. Choose $m$ such that $m\ge m'$ and $\min(\overline{s}_1,\ldots,\overline{s}_n)\ge -m$. Then, $Y$ is equivalent to $\left(\phi_0\right)^{m-m'}(Y)$ in the colimit which is in $\cHFL^{\stab}(T(n,mn),\ovl{\ss})$. 
\end{proof}

\subsection{Cable algebra and cobordism maps} Next, we explicitly describe the cobordism maps $\phi_k$ corresponding to adding full twists  for $T(n,mn)$. In order to do that, we consider the direct sum of homology for all $n$-strand cables of the unknot with nonnegative integer slopes:
\begin{equation}\label{eq:CabU}
\Cab_n(O): =\Tw(O_n)=\bigoplus_{m=0}^{\infty}\cHFL(T(n,nm))
\end{equation}
where $O$ is the unknot and  $O_n$ is the $n$-component unlink. 

Each summand is $\Z^n\oplus \Z$ graded by Alexander multi-grading ($\Z^n$) and Maslov grading ($\Z$). Moreover, $\Cab_n(O)$ has an additional $\Z$ grading given by the cabling slope $m$, we call it \emph{twist grading} and denote it by $\tw$. Note that cobordism maps $\phi_k$, corresponding to adding a full twist, act naturally on $\Cab_n(O)$ with twist grading $1$ i.e. they map $\cHFL(T(n,mn))$ to $\cHFL(T(n,(m+1)n))$. 

\begin{definition}
The $n$-strand \newword{cable algebra} $\CA_n$ is defined as a $\Z^n\oplus \Z\oplus \Z$ graded algebra  over $R_{UV}$ with commuting generators $a_0,\ldots,a_{n-1}$ and the following relations:
\begin{itemize}
\item Linear relations
\begin{equation}
\label{eq: linear}
U_Ia_{k-1}=V_{\overline{I}}a_k
\end{equation}
where $I\subset \{1,\ldots,n\}$ is an arbitrary subset with $|I|=k$ and $\overline{I}=\{1,\ldots,n\}\setminus I.$ 
\item Quadratic relations
\begin{equation}
\label{eq: quadratic}
a_ia_j=\UU^{k\ell-ij}a_{k}a_{\ell}
\end{equation} 
whenever $i+j=k+\ell$ and $i\le k\le \ell\le j$.
\end{itemize}
As above, the gradings correspond to the Alexander ($\Z^n$) and Maslov ($\Z$) gradings, along with an additional  $\Z$ \emph{twist grading}. The generator $a_k$ has Alexander grading 
$$
A(a_k)=\left(\frac{n-1}{2}-k,\ldots, \frac{n-1}{2}-k\right),
$$
Maslov grading $\gr_{\w}(a_k)=-k^2-k$ and twist grading $\tw(a_k)=1$.
\end{definition}

It is easy to check that the relations are homogeneous with respect to all three gradings.

\begin{theorem}
\label{thm: cabled algebra}
The homology $\Cab_n(O)$ is a $\Z^n\oplus \Z\oplus \Z$ graded module over the algebra $\CA_n$ where the generators $a_k$ act by cobordism maps $\phi_k$. Furthermore, $\Cab_n(O)$ is a free $\CA_n$-module generated by a single class $1\in \cHFL(T(n,0))$.
\end{theorem}

\begin{proof}
We proceed in several steps.

{\bf Step 1:} We show that there is a well-defined action of algebra $\CA_n$ on $\Cab_n(O)$, where $a_k$ acts by the cobordism map $\phi_k$. First, we prove that the actions of $\phi_{k}$ and $\phi_{k'}$ on $\Cab_n(O)$ commute. By \cite{GH} all links $T(n,nm)$ are L-space links, so $\cHFL(T(n,mn))$ is free of rank 1 over $\F[\UU]$ in each Alexander degree. The maps $\phi_{k}$ and $\phi_{k'}$ correspond to negative definite cobordisms, so they are both injective \cite{AGL}, and hence determined by their Alexander and Maslov degrees. Now we observe that both compositions $\phi_{k}\circ \phi_{k'}$ and $\phi_{k'}\circ \phi_{k}$
have the same Alexander and Maslov degrees by Proposition \ref{prop:gradingshifts}, and hence coincide.

Next, we check that these cobordism maps satisfy  relations \eqref{eq: linear} and \eqref{eq: quadratic}. Note that $a_k$ has the same Alexander, Maslov and twist grading as $\phi_k$, thus relations \eqref{eq: linear} and \eqref{eq: quadratic} being homogeneous, implies that the cobordism maps satisfy  these relations as well and the action is well-defined. 




{\bf Step 2:} We prove that $\Cab_n(O)$ is generated by a single class $1\in \cHFL(T(n,0))$ under the action of $\CA_n$. 
We follow Theorem \ref{thm: Tnm} and claim that the generator $Y_i\in \cHFL(T(n,mn))$ can be identified with the monomial $a_{q}^{m-r}a_{q+1}^{r}$ (where $i=mq+r$) applied to $1$.
By the above, the element $a_{q}^{m-r}a_{q+1}^{r}(1)$ is nonzero, and it is sufficient to verify that it has the same Alexander and Maslov degree as $Y_i$. Indeed, the Alexander grading equals
$$
A_j(a_{q}^{m-r}a_{q+1}^{r})=(m-r)\left(\frac{n-1}{2}-q\right)+r\left(\frac{n-1}{2}-q-1\right)=m\frac{n-1}{2}-mq-r=\cc-i
$$
and the Maslov grading equals
$$
\gr_{\w}(a_{q}^{m-r}a_{q+1}^{r})=-(m-r)q(q+1)-r(q+1)(q+2)=-(q+1)(mq+2r).
$$
To sum up, we have a surjective map of $\CA_n$ modules from $\CA_n$ to $\Cab_n(O)$.

{\bf Step 3:} Finally, we need to prove that this map is an isomorphism and $\Cab_n(O)\simeq \CA_n$. For this, we need to find the generators of $\CA_n$ over $R_{UV}$, relations between them and compare these with Theorem \ref{thm: Tnm}. 

Define $\widetilde{Y_i}=a_{q}^{m-r}a_{q+1}^{r}\in \CA_n$ where $i=mq+r$. Let us prove that $\widetilde{Y_i}$ generate $\CA_n$ over $R_{UV}$. Indeed, consider a monomial $M=a_{i_1}\cdots a_{i_m}$ with $i_1\le i_2\le \cdots \le i_m$. If $i_m-i_1\le 1$ then $M=\widetilde{Y}_{i_1+\ldots+i_m}$ by the above. Otherwise we can use \eqref{eq: quadratic} to write
$$
a_{i_1}a_{i_m}=\UU^{i_m-i_1-1}a_{i_1+1}a_{i_m-1}.
$$ 
since $(i_1+1)(i_m-1)-i_1i_m=i_m-i_1-1$.
This decreases the power of $a_{i_1}$ and we can proceed by induction.

Finally, we need to verify that the relations from Theorem \ref{thm: Tnm}  hold for $\widetilde{Y_i}$. This is clear since the relations between
$\widetilde{Y}_i=a_{q}^{m-r}a_{q+1}^{r}$ and $\widetilde{Y}_{i+1}=a_{q}^{m-r-1}a_{q+1}^{r+1}$ agree with the relations between $a_{q}$ and $a_{q+1}$,
up to multiplication by $a_{q}^{m-r-1}a_{q+1}^{r}$. 
\end{proof}






\begin{example}
The homology of $T(2,2m)$ is generated by $a_0^m,a_0^{m-1}a_1,\ldots,a_1^{m}$ modulo relations $U_1a_0=V_2a_1$ and $U_2a_0=V_1a_1$, times a degree $(m-1)$ monomial. 
\end{example}

\begin{example}
The homology of $T(3,6)$ is generated by  six degree 2 monomials  in $a_0,a_1,a_2$, but in fact we have a relation $a_0a_2=\UU a_1^2$ since $a_0a_2$ and $a_1^2$ have the same Alexander degree. So $\cHFL(T(3,6))$ is generated by $a_0^2,a_0a_1,a_1^2,a_1a_2,a_2^2$.
\end{example}

\begin{example}
Similarly, the homology of $T(3,3m)$ is generated by $2m+1$ monomials
$$
a_0^{m},a_0^{m-1}a_1,\ldots,a_0a_1^{m-1},a_1^{m},a_1^{m-1}a_2,\ldots,a_1a_2^{m-1},a_2^{m}.
$$
\end{example}

\subsection{Colored knot Floer homology of the unknot} Finally, we are ready to compute the colored homology $\HD_{n}(O)$. Recall that this invariant is defined as the limit of the directed system of $\cHFL(T(n,mn))$ along the directed system of maps
$\phi_{0}$. We define the algebraic analogue of this as follows:

\begin{definition}
Let the algebra $\CA_n^{\col}$
be the (graded) localization of $\CA_n$ in $a_0$:
$$
\CA_n^{\col}:=\CA_n[a_0^{-1}]:=\spann\left\{\frac{Y}{a_0^m}:\tw(Y)=m\right\}\subset \mathrm{Frac}(\CA_n).
$$

\end{definition}

Note that since $\phi_0:\Cab_n(O)\to \Cab_n(O)$ is injective, by Theorem \ref{thm: cabled algebra} the multiplication by $a_0$ in $\CA_n$ is injective as well. Therefore $a_0$ is not a zero divisor in $\CA_n$, and $\CA_n[a_0^{-1}]$ embeds into $\mathrm{Frac}(\CA_n).$

\begin{lemma}
\label{lem: colored unknot as localization}
The colored homology of the unknot $\HD_{n}(O)$ is a free rank 1 module over $\CA_n^{\col}$.
\end{lemma}

\begin{proof}
This is an easy consequence of Theorem \ref{thm: cabled algebra}. We identify an element $Y\in \CA_n$ such that $\tw(Y)=m$ with a class $Y(1)\in \cHFL(T(n,mn))$. By Theorem \ref{thm: cabled algebra} this yields an isomorphism $\cHFL(T(n,mn))\simeq \spann\left\{Y\in \CA_n\mid \tw(Y)=m\right\}$.

The colored homology $\HD_{n}(O)$ is spanned by classes $Y(1)\in \cHFL(T(n,mn))$ for all $m$ modulo relations $Y(1)\sim \phi_0(Y(1))$, equivalently, by the elements $Y\in \CA_n$ modulo relations $Y\sim a_0Y$. Given such an element $Y$, we can associate to it the fraction $\frac{Y}{a_0^m}\in \mathrm{Frac}(\CA_n)$. This agrees with our equivalence relation since  
$$
\frac{Y}{a_0^m}=\frac{a_0Y}{a_0^{m+1}}.
$$
and the result follows. 
\end{proof}

 Next, we would like to have some concrete description of $\CA_n[a_0^{-1}]$ and $\HD_n(O)$.

\begin{theorem}
\label{thm: colored unknot}
The   colored homology of unknot is isomorphic to 
\begin{equation}
\label{eq: localized algebra}
\HD_n(O)\simeq \CA_n^{\col}\simeq 
\frac{\mathbb{F}[U_1,\ldots,U_n,V_1,\ldots,V_n,\AAA]}{(U_i=V_{\overline{i}}\AAA)}.
\end{equation}
where $\AAA=a_1/a_0$. The generator $\AAA$ has Maslov grading $\gr_{\w}=-2$ and renormalized Alexander grading $(-1,\ldots,-1)$.
\end{theorem}

\begin{proof}
The first isomorphism follows from the Lemma \ref{lem: colored unknot as localization}, so we focus on $\CA_n^{\col}=\CA_n[a_0^{-1}]$.
The algebra  $\CA_n$ is generated by $a_0,\ldots,a_{n-1}$, so $\CA_n[a_0^{-1}]$ is generated by $a_k/a_0$. We claim that
\begin{equation}
\label{eq: ak from A}
\frac{a_k}{a_0}=\UU^{\frac{k(k-1)}{2}}\left(\frac{a_1}{a_0}\right)^k=\UU^{\frac{k(k-1)}{2}}\AAA^k.
\end{equation}
This is equivalent to the identity $a_ka_0^{k-1}=\UU^{\frac{k(k-1)}{2}}a_1^{k}$ which can be proved by repeated application of \eqref{eq: quadratic} and induction in $k$:
$$
a_ka_0^{k-1}=(a_ka_0)a_0^{k-2}=(\UU^{k-1}a_{k-1}a_1)a_0^{k-2}=\UU^{k-1}a_1(a_{k-1}a_0^{k-2}).
$$
Therefore $\CA_n[a_0^{-1}]$ is generated by the powers of $\AAA$ over $R_{UV}$. 

The relations \eqref{eq: linear} for $k=1$ have the form $U_ia_0=V_{\overline{i}}a_1$ and imply $U_i=V_{\overline{i}}\AAA$. Let us check that these relations in turn imply all other relations in $\CA_n[a_0^{-1}]$. First note that 
$$
U_iV_i=V_1\cdots V_k\AAA
$$
does not depend on $i$, and we do not need to add this as an additional relation.

Let $|I|=k>1$. We claim that 
\begin{equation}
\label{eq: linear localized}
U_I=\UU^{k-1}V_{\overline{I}}\AAA.
\end{equation}
Indeed, if $I=\{i_1,\ldots,i_k\}$ then 
$$
\UU^{k-1}V_{\overline{I}}=U_{i_2}\cdots U_{i_k}V_{i_2}\cdots V_{i_k}V_{\overline{I}}=U_{i_2}\cdots U_{i_k}V_{\overline{i_1}}
$$
hence
$$
\UU^{k-1}V_{\overline{I}}\AAA=U_{i_2}\cdots U_{i_k}(V_{\overline{i_1}}\AAA)=U_{i_2}\cdots U_{i_k}U_{i_1}=U_I.
$$
Now \eqref{eq: linear localized} implies (after multiplication by $\UU^{\frac{(k-1)(k-2)}{2}}\AAA^{k-1}$):
$$
\UU^{\frac{(k-1)(k-2)}{2}}U_I\AAA^{k-1}=\UU^{\frac{k(k-1)}{2}}V_{\overline{I}}\AAA^{k}
$$
By applying \eqref{eq: ak from A} this can be rewritten as $U_I\frac{a_{k-1}}{a_0}=V_{\overline{I}}\frac{a_k}{a_0}$, and after multiplication by $a_0$ we get \eqref{eq: linear}.

To prove \eqref{eq: quadratic}, we divide both sides by $a_0^2$ and get 
$$
\left(\frac{a_i}{a_0}\right)\left(\frac{a_j}{a_0}\right)=\UU^{k\ell-ij}\left(\frac{a_k}{a_0}\right)\left(\frac{a_{\ell}}{a_0}\right)
$$
and this is immediate from \eqref{eq: ak from A}.

\end{proof}

\begin{remark}

We can also directly check the isomorphism \eqref{eq: localized algebra} using Lemma \ref{lem: stable range} and Corollary \ref{cor: homology stabilize}. By Lemma \ref{lem: stable range}, we have two cases. If $\min(\overline{s}_1,\ldots,\overline{s}_n)\ge 0$ then $\cHFL^{\stab}(T(n,mn),\overline{s}_1,\ldots,\overline{s}_n)$ is generated by 
$$
V_1^{\overline{s}_1}\cdots V_n^{\overline{s}_n}Y_0=V_1^{\overline{s}_1}\cdots V_n^{\overline{s}_n}a_0^m\sim V_1^{\overline{s}_1}\cdots V_n^{\overline{s}_n}.
$$
If  $-m\le \min(\overline{s}_1,\ldots,\overline{s}_n)\le 0$ then $\cHFL^{\stab}(T(n,mn),\ovl{s}_1,\ldots,\ovl{s}_n)$ is generated by 
$$
V_1^{\overline{s}_1+k}\cdots V_n^{\overline{s}_n+k}Y_k=V_1^{\overline{s}_1+k}\cdots V_n^{\overline{s}_n+k}a_0^{m-k}a_1^k\sim V_1^{\overline{s}_1+k}\cdots V_n^{\overline{s}_n+k}\AAA^k.
$$
Therefore the limiting homology is indeed isomorphic to 
$$\mathbb{F}[U_1,\ldots,U_n,V_1,\ldots,V_n,\AAA]/(U_i=V_{\hat{i}}\AAA)\simeq \mathbb{F}[V_1,\ldots,V_n,\AAA]$$ 
as a graded vector space.

\end{remark}


\section{Module structures of Colored knot Floer homology}

In this section, we use the algebra $\CA_n$ and its localization to study the module structure for cabled and colored knot Floer homology. 

\subsection{Action of cobordism maps: general case}
Recall that for an $n$-component link $L$ in $S^3$ such that every component intersects a disk $D$ positively exactly once, 
$$
\Tw_D(L)=\bigoplus_{m=0}^{\infty}\cHFL(L_m).
$$
As before, this is a $\Z^n\oplus \Z\oplus \Z$ graded vector space with Alexander, Maslov and twist gradings.


\begin{theorem}
\label{thm: cabled module}
The space $\Tw_D(L)$ is a triply graded module over the algebra $\CA_n$.
\end{theorem}

\begin{proof}
We have cobordism maps $\phi_k:\cHFL(L_m)\to \cHFL(L_{m+1})$. We construct the action of the algebra $\CA_n$ such that the generators $a_k$ act on $\Tw_D(L)$ via the maps $\phi_k$. For this to make sense, we need to verify that the maps $\phi_k$ pairwise commute and satisfy the relations \eqref{eq: linear} and \eqref{eq: quadratic}. 
We closely follow the proof of \cite[Proposition 3.13]{AGL}. 

We would like to compare the cobordism maps from $\cHFL(L_m)$ to $\cHFL(L_{m+1})$ and the similar maps from $\cHFL(T(n,mn))$ to $\cHFL(T(n,(m+1)n))$. To distinguish them, we denote them respectively by $\phi_k^{L}$ and $\phi_k^{O}$. 

Let us denote by $L'_{m}$ the link $L_{m}$ with an extra pair of basepoints $w'_i,z'_i$ per component. We extend the coloring $\sigma$ on $L_{m}$ to $L'_{m}$ so that its codomain is $P'=P\sqcup \{p'_1, \cdots, p'_n\}$
and $\sigma'(w'_i)=p'_i, \sigma'(z'_i)=\sigma(z_i)$.
Then $\sigma'$ restricts to a coloring of $L_{m}$ with codomain $P'$, and for simplicity, we still denote it by $\sigma'$. The ring isomorphism $\mathcal{R}_P^-\cong \F[U_1, \cdots, U_n, V_1, \cdots V_n]$ extends to an isomorphism 
$$\mathcal{R}_{P'}^-\cong \F[U_1, \cdots, U_n, V_1, \cdots, V_n, U'_1, \cdots , U'_n]$$
by sending formal variables $X_{p'_i}$ to $U'_i$. Then, we have the following isomorphisms
$$\cHFL(L_m^{\sigma'})\cong \cHFL(L_m)\otimes_{\F} \F[U'_1, \cdots, U'_n] $$
and 
$$\cHFL(L_m^{'\sigma'})\cong \cHFL(L_m^{\sigma'})/ \langle U_1-U'_1, \cdots, U_n-U'_n\rangle \cong \cHFL(L_m).$$
Let $\mathcal{S}_m$ be the quasi-stabilization cobordism from $L_m$ to $L'_{m}$. Then, the composition of the induced cobordism map from $\cHFL(L_m)$ to $\cHFL(L'_m)$ with the latter isomorphism is identity.



Observe that $L_m$ can be obtained as band connected sum of $L$ and $T(n,mn)$ where we add one band per link component. Consider the following cobordisms: 

\begin{itemize}
\item $\mathcal{C}_{m}$ is the decorated cobordism from $L_m$ to $L_{m+1}$ obtained from the $(-1)$-surgery on an unknot bounding a disk intersecting each component of $L_m$ positively at exactly one point,
\item $\mathcal{C}'_m$ is the induced decorated cobordism from $L'_{m}$ to $L'_{m+1}$ similar to $\mathcal{C}_m$,
\item $\mathcal{B}_m$ from $L_{m}$ to $L_{m}\sqcup O_n$ corresponding to birth of an $n$-component unlink,
\item $\mathcal{C}_{b,m}$ from $L_{m}\sqcup O_n$ to $L'_{m}$ is given by (decorated) band attachment. More generally, band attachment gives a cobordism
$\mathcal{C}_{b,m,j}$ from $L_{m}\sqcup T(n,jm)$ to $L'_{m+j}$,
\item $\mathcal{C}_{O_n,m}$  is the cobordism from $L_{m}\sqcup O_n$ to $L_{m}\sqcup T(n,n)$ given by the 2-handle attachment along the (-1)-framed unknot for inserting a full twist in $O_n$.
\end{itemize}

Define
$$\tilde{\mathcal{C}}_m=\mathcal{C}'_m\circ \mathcal{S}_m=\mathcal{S}_{m+1} \circ \mathcal{C}_m.$$
Under the aforementioned isomorphism $\cHFL(L_{m+1}^{'\sigma'})\cong \cHFL(L_{m+1})$, the homomorphism induced by the cobordism map $F_{\tilde{\mathcal{C}}_m, \s_k}$
 from $\cHFL(L_{m})$ to $\cHFL(L_{m+1})$ is equal to $\phi^{L}_k$. Note that $\mathcal{S}_m$ can be decomposed as the cobordism $\mathcal{B}_m$ containing $n$ births from $L_{m}$ to $L_{ m}\sqcup O_n$ followed by $n$ band attachments cobordism $\mathcal{C}_{b,m}$. 

By arguing as in \cite[Proposition 3.13]{AGL} (see also Proposition 3.8 in \cite{AGL}) we can isotope the attaching circle of the 2-handle in $\tilde{\mathcal{C}}_m$ so that we can
change the order of 2-handle attachment and band attachments, which shows that the composition
$$
\tilde{\mathcal{C}}_m=\mathcal{C}'_{m}\circ \mathcal{C}_{b,m}\circ \mathcal{B}_{m}
$$
is isotopic to the composition
$$
\mathcal{C}_{b,m,1}\circ \mathcal{C}_{O_n,m}\circ \mathcal{B}_{m}.
$$

Therefore by naturality (\cite{Zemke}) we get
$$
\phi_k^{L}=F_{\tilde{\mathcal{C}}_m, \s_k}=F_{\mathcal{C}_{b,m,1}}\circ F_{\mathcal{C}_{O_n,m},\mathfrak{s}_k}\circ F_{\mathcal{B}_{m}}=F_{\mathcal{C}_{b,m,1}}\circ \left(\Id\otimes \phi_k^{O}\right)\circ F_{\mathcal{B}_{m}}.
$$
Here, $\phi_k^{O}: \cHFL(O_{n})\rightarrow \cHFL(T(n,n))$ denotes the full twist cobordism map. Since all the maps are $R_{UV}$-linear and  $\phi_k^{O}$ satisfy linear relations \eqref{eq: linear} by Theorem \ref{thm: cabled algebra}, the maps $\phi_k^{L}$ satisfy \eqref{eq: linear} as well. 

Next, we need to study the compositions of two maps $\phi_{\ell}^{L}: \cHFL(L_{m})\rightarrow\cHFL(L_{m+1})$ and $\phi_{k}^{L}: \cHFL(L_{ m+1})\rightarrow\cHFL(L_{ m+2})$
$$
\phi_{k}^{L}\circ \phi_{\ell}^{L}=
\phi_{k}^{L}\circ
F_{\mathcal{C}_{b,m,1}}\circ \left(\Id\otimes \phi_k^{O}\right)\circ F_{\mathcal{B}_{m}}.
$$

By isotopying the attaching circle of the second $2$-handle, changing the order of band attachments and the second 2-handle attachment, and using the naturality of link Floer homology we have
$$
\phi_{k}^{L}\circ \phi_{\ell}^{L}=F_{\mathcal{C}_{b,m,2}}\circ \left(\Id\otimes \left(\phi_{k}^{O}\circ\phi_{\ell}^{O}\right)\right)\circ F_{\mathcal{B}_{m}}.
$$
Consequently, since \eqref{eq: quadratic} and commutation relations hold for the unlink by Theorem \ref{thm: cabled algebra}, they hold for $L$ as well. 
\end{proof}

\subsection{Colored homology: general case}

Next, we turn to the limit $\HD_D(L)$ for a link $L$ with the decorated disk $D$.

\begin{theorem}
\label{thm: colored module}
The colored homology $\CH_{D}(L)$ is a module over the algebra $\CA_n^{\col}$.
\end{theorem}

\begin{proof}
This is a consequence of Theorem \ref{thm: cabled module}. Consider an element $b=\frac{a}{a_0^t}\in \CA_n[a_0^{-1}]$ with $\tw(a)=t$, and a representative in an equivalence class $Y\in \cHFL(L_{m})$. We define $b(Y):=[a(Y)]$  where the action of $a$ is given by Theorem \ref{thm: cabled module}.

We need to check that this is well-defined under the equivalence relations for $b$ and $Y$. First, we can write 
$$
b=\frac{a}{a_0^t}=\frac{aa_0}{a_0^{t+1}}.
$$
Then 
$$
aa_0(Y)=a_0(a(Y))=\phi_0(a(Y))\sim a(Y).
$$
Second, we have $Y\sim \phi_0(Y)$ and $$a(\phi_0(Y))=a(a_0(Y))=a_0(a(Y))=\phi_0(a(Y))\sim a(Y).$$

Note that in both cases we need to use the fact that the actions of $a$ and $\phi_0$ commute. This follows from the fact (proven in Theorem \ref{thm: cabled module}) that the actions of $\phi_k$ and $\phi_0$ commute. In fact, by Lemma \ref{lem: colimit action} this is enough to ensure the action of $a$ on the colimit, as in the second case.
\end{proof}

\subsection{Braid group action}

The homology of cables has a natural braid group action by cobordisms that swap different components. We summarize its properties in the following proposition. 

\begin{proposition}
The braid group $\Br_n$ acts on $\cHFL(L_{m})$ for all $m$. Given a braid $\beta\in \Br_n$ and the corresponding permutation $w_{\beta}\in S_n$ we have
$$
\beta\circ U_i=U_{w_{\beta}(i)}\beta,\ \beta\circ V_i=V_{w_{\beta}(i)}\beta.
$$
Furthermore, the action of $\Br_n$ commutes with the maps $\phi_k:\cHFL(L_{m})\to \cHFL(L_{m+1})$:
$$
\beta\circ \phi_k=\phi_k\circ \beta.
$$
\end{proposition}

\begin{proof}
All statements follow from the naturality of cobordism maps in link Floer homology \cite{Zemke}. For the last equation, note that $\phi_k$ corresponds to blowing down the meridian which commutes (up to isotopy) with permuting the components of the link.

\end{proof}

\begin{corollary}
There is a natural $S_n$ action on the algebra $\CA_n$ which permutes $U_i,V_i$ and fixes $\UU$ and $a_0,\ldots,a_{n-1}$.  Furthermore, the action of $\Br_n$ on $\Tw_D(L)$ and the action of $S_n$ on $\CA_n$ are compatible via the action of $\CA_n$ on $\Tw_D(L)$ from Theorem \ref{thm: cabled module}.
\end{corollary}

Note that the relations \eqref{eq: linear} and \eqref{eq: quadratic} are clearly invariant under the $S_n$ action. We can apply these results to colored homology.

\begin{proposition}
There is a natural action of $\Br_n$ on the colored homology $\HD_{D}(L)$. It is compatible with the action of $S_n$ on $\CA_n^{\col}=\CA_n[a_0^{-1}]$ which permutes $U_i,V_i$ and fixes $\AAA$.
\end{proposition}

\begin{proof}
This follows from the fact that braid group action fixes $a_0$, and Lemma \ref{lem: colimit action}.
\end{proof}

\section{Examples: L-space knots}
\label{sec: L space}

As in \eqref{eq: def normalized Alexander},
we define $\cc_m=m(n-1)/2$ and
$s_i=\overline{s}_i+\cc_m$. Similarly, we define $h^{\stab}_{n,mn}(\overline{\ss})=h_{n,mn}(\ss)$. If $K_{n,mn}$ is an L-space link then 
$$
\cHFL^{\stab}(K_{n,mn},\overline{\ss})\simeq \F[\UU]
\left[-2h_{n,mn}^{\stab}(\overline{\ss})\right].
$$

\subsection{Colored homology: description}

We study cables $K_{n,mn}$ with $n$ components. Each of these components is isotopic to $K_{1,m}=K$. In this section, we assume that $K$ is an L-space knot and $m$ is large, so that $K_{n,mn}$ is an $L$-space link by \cite{GH}. 
First, we can describe the colored homology as a graded $\F[\UU]$-module, one Alexander grading at a time.   

\begin{lemma}
\label{lem: stable h}
Let us fix the renormalized Alexander grading $\overline{\ss}=(\overline{s}_1,\ldots,\overline{s}_n)$. Then for $m$ large enough we have 
$$
\cHFL^{\stab}(K_{n,mn},\overline{\ss})\simeq \F[\UU][-2h_K(\min(\overline{\ss}))]
$$  
where $\min(\overline{\ss})=\min\{\overline{s}_1,\cdots,\overline{s}_n\}$, and the map $\phi_0:\cHFL^{\stab}(K_{n,mn},\overline{\ss})\to \cHFL^{\stab}(K_{n,(m+1)n},\overline{\ss})$ is an isomorphism. In particular,  for $m$ large enough we get
$
\cHFL^{\stab}(K_{n,mn},\overline{\ss})\simeq \CH_n(K,\overline{\ss}).
$
\end{lemma}

\begin{proof}
As in Theorem \ref{thm: cabled algebra}, we argue that the map $\phi_0:\cHFL^{\stab}(K_{n,mn})\to \cHFL^{\stab}(K_{n,(m+1)n})$ is nonzero, has (stable) Alexander degree zero and its action on the tower $\cHFL^{\stab}(K_{n,mn},\overline{\ss})\simeq \F[\UU]$ is determined by its Maslov degree.

Note that the $h$-function is symmetric, and by \cite[Theorem 4.3]{GH} we can write it for $K_{n,mn}$ as
$$
h_{n,mn}(s_1,\ldots,s_n)=h_K(s_1-\cc_m)+h_K(s_2-\cc_m+m)+\ldots+h_K(s_n-\cc_m+(n-1)m).
$$ 
where $s_1\le s_2\le \ldots\le s_n$ and $h_K$ denotes the $h$-function of $K$. So,
$$
h_{n,mn}^{\stab}(\overline{s}_1,\ldots,\overline{s}_n)=h_K(\overline{s}_1)+h_K(\overline{s}_2+m)+\ldots+h_K(\overline{s}_n+(n-1)m).
$$ 
We can choose $m$ large enough so that $\overline{s}_i>g(K)-m$ for all $i$, then $h_K(\overline{s_k}+(k-1)m)=0$ for $k>1$. We conclude that for large $m$
$$
h_{n,mn}^{\stab}(\overline{s}_1,\ldots,\overline{s}_n)=h_K(\overline{s}_1)=h_K(\min(\overline{\ss})).
$$
and the result follows.
\end{proof}

Next, we describe the module structure of this homology over $\F[U_1,\ldots,U_n,V_1,\ldots,V_n]$.
First, we need to introduce some notations.
Following \cite{OSlens} we have
$$
\chi_K(t)=\frac{\Delta_K(t)}{1-t^{-1}}=\sum_{\sigma\in S}t^{\sigma}=t^{\sigma_1}+t^{\sigma_2}+\ldots
$$
where $S=\{\sigma_1,\sigma_2,\ldots\}$ is some bounded above subset of $\mathbb{Z}$ and we assume $\sigma_1>\sigma_2>\ldots$ where $\sigma_i=1-i$ for $i\ge g+1$. We can describe $\cHFL(K)$ by generators and relations: 
\begin{itemize}
\item The generators are $z_{\sigma_i}$ for $\sigma_i\in S$, with Alexander degree $\sigma_i$ and homological degree $-2h_K(\sigma_i)=-2(i-1)$.
\item The relations are 
\begin{equation}
\label{eq: zigzag}
Uz_{\sigma_i}=V^{\sigma_i-\sigma_{i+1}-1}z_{\sigma_{i+1}}.
\end{equation}
\end{itemize}
Indeed, both $Uz_{\sigma_i}$ and $V^{\sigma_i-\sigma_{i+1}-1}z_{\sigma_{i+1}}$ have Alexander degree $\sigma_i-1$ and homological degree $-2i=-2(i-1)-2$. Note that when $\sigma_{i+1}=\sigma_i-1$  we have $Uz_{\sigma_i}=z_{\sigma_{i+1}}$, so we can in principle eliminate $z_{\sigma_{i+1}}$. In particular, for $i\ge g+1$ we can write $z_{\sigma_i}=U^{i-g-1}z_{-g}$.

\begin{theorem}
\label{thm: colored L space}
The colored homology $\CH_{n}(K)$ has generators 
$\widetilde{z}_{\sigma_i}$ of Alexander degree
$$
A\left(\widetilde{z}_{\sigma_i}\right)=(\sigma_i,\ldots,\sigma_i),
$$
and Maslov degree $-2(i-1)$. The relations are given by
\begin{equation}
\label{eq: colored relations L space}
U_j\widetilde{z}_{\sigma_i}=V^{(\sigma_i-\sigma_{i+1})\ee-\ee_j}\widetilde{z}_{\sigma_{i+1}}, 1\le j\le n,
\end{equation}
where $\ee=\sum_{j=1}^n\ee_j$, and $\ee_j$ is the unit vector with $j$-th entry equal one. Here, $$V^{(\sigma_i-\sigma_{i+1})\ee-\ee_j}=(V_1\cdots V_n)^{\sigma_i-\sigma_{i+1}-1}\left(\prod_{l\neq j}V_l\right).$$
\end{theorem}

\begin{proof}
Let $z(\overline{\ss})$ denote the generator of the $\F[\UU]$ tower in $\CH_{n}(K)$ of renormalized Alexander degree $\overline{\ss}$. By Lemma \ref{lem: stable h} the Maslov degree of  $z(\overline{\ss})$ equals $-2h_K(\min(\overline{\ss}))$.

We can describe the action of $U_i$ and $V_i$ on generators $z(\overline{\ss})$ as follows:

{\bf Case 1:} $\min(\overline{\ss})=\min(\overline{\ss}+\ee_i)=\sigma$. In this case 
\begin{equation}
\label{eq: U V action case 1}
V_iz(\overline{\ss})=z(\overline{\ss}+\ee_i),\ U_iz(\overline{\ss}+\ee_i)=\UU z(\overline{\ss}).
\end{equation}
Note that this implies
$$
z(\overline{\ss})=V_1^{\overline{s}_1-\sigma}\cdots V_n^{\overline{s}_n-\sigma}z(\sigma,\ldots,\sigma)
$$
and the diagonal elements generate $\CH_{n}(K)$ under the action of $V_i$. 

{\bf Case 2:}
$\min(\overline{\ss})=\sigma,\min(\overline{\ss}+\ee_i)=\sigma+1$. In this case
\begin{equation}
\label{eq: U V action case 2}
V_iz(\overline{\ss})=\UU^{h_K(\sigma)-h_K(\sigma+1)}z(\overline{\ss}+\ee_i),\ U_iz(\overline{\ss}+\ee_i)=\UU^{1-h_K(\sigma)+h_K(\sigma+1)}z(\overline{\ss}).
\end{equation}
By combining \eqref{eq: U V action case 1} and \eqref{eq: U V action case 2} we get 
$$
V_1\cdots V_nz(\sigma,\ldots,\sigma)=\UU^{h_K(\sigma)-h_K(\sigma+1)}z(\sigma+1,\ldots,\sigma+1).
$$
In particular, $
V_1\cdots V_nz(\sigma,\ldots,\sigma)=z(\sigma+1,\ldots,\sigma+1)$ whenever $h_K(\sigma)=h_K(\sigma+1)$, so the homology is generated by 
$$
\widetilde{z}_{\sigma_i}:=z(\sigma_i,\ldots,\sigma_i).
$$
The relations \eqref{eq: colored relations L space} follow directly from grading computation and by a similar argument as Theorem \ref{thm: Tnm}, the relations \eqref{eq: colored relations L space} generate all relations among the generators.

\end{proof}

\begin{remark}
\label{rmk: infinitely generated}
We get an infinite chain of generators $\widetilde{z}_{\sigma_i}$ for $i\ge g+1$, these have Alexander degrees $A(\widetilde{z}_{\sigma_i})=(1-i,\ldots,1-i)$ and Maslov degrees $\gr_{\w}(\widetilde{z}_{\sigma_i})=-2(i-1)$.  The relations between these generators are nontrivial and given by
$$
U_i\widetilde{z}_{\sigma_i}=V_{\overline{i}}\widetilde{z}_{\sigma_{i+1}}.
$$
In particular, $\CH_{n}(K)$ is infinitely generated over $\F[U_1,\ldots,U_n,V_1,\ldots,V_n]$.
\end{remark}

\subsection{Colored homology: module structure}

In this subsection we determine the module structure of $\CH_{n}(K)$ over the cabled algebra $\CA_n[a_0^{-1}]$ assuming that $K$ is an L-space knot. 


\begin{lemma}
\label{lem: A action}
In the notations of Theorem \ref{thm: colored L space} the action of $\AAA=a_1/a_0$ is given by
$$
\AAA\widetilde{z}_{\sigma_i}=(V_1\cdots V_n)^{\sigma_i-\sigma_{i+1}-1}\widetilde{z}_{\sigma_{i+1}}.
$$
\end{lemma}

\begin{proof}
The action of $\phi_1:\cHFL(K_{n,mn})\to \cHFL(K_{n,(m+1)n})$ is nontrivial and sends an $\F[\UU]$ tower inside another $\F[\UU]$ tower. Hence by construction in Theorem \ref{thm: colored module} the action of $\AAA=a_1/a_0$ on colored homology is nontrivial as well, and is determined by its shift of Alexander and Maslov degrees.

The element $\AAA\widetilde{z}_{\sigma_i}$  has (renormalized) Alexander degree $(\sigma_i-1,\ldots,\sigma_i-1)$ and Maslov degree $-2i$ which agrees with the Alexander and Maslov degrees of 
$(V_1\cdots V_n)^{\sigma_i-\sigma_{i+1}-1}\widetilde{z}_{\sigma_{i+1}}.$
\end{proof}

\begin{remark}
For $i\ge g+1$ we have 
$
\AAA\widetilde{z}_{\sigma_i}=\widetilde{z}_{\sigma_{i+1}},
$
so we can compactly write $$\widetilde{z}_{\sigma_i}=\AAA^{i-g-1}\widetilde{z}_{-g}.$$
In particular, $\CH_{n}(K)$ is finitely generated over $\F[U_1,\ldots,U_n,V_1,\ldots,V_n,\AAA]$, in contrast with Remark \ref{rmk: infinitely generated}.
\end{remark}



\begin{theorem}
\label{thm: colored tensor product}
Assume $K$ is an L-space knot, then
$$
\CH_{n}(K)\simeq \cHFL(K)\otimes_{\F[U,V]}\CA_n^{\col}
$$
where we regard $\CA_n^{\col}$ as a $\F[U,V]$-module via the homomorphism $\varepsilon_n:\F[U,V]\to \F[U_1,\ldots,U_n,V_1,\ldots,V_n,\AAA]$
defined by
$$
\varepsilon_n(U)=\AAA,\ \varepsilon_n(V)=V_1\cdots V_n.
$$

\end{theorem}
 
\begin{proof}
We would like to compare the right hand side (as an $R_{UV}$-module) with Theorem \ref{thm: colored L space}. 

First, we claim that the elements $z_{\sigma_i}\otimes 1$ generate the right hand side over $R_{UV}$, and can be identified with $\widetilde{z}_{\sigma_i}$ in the left hand side. Indeed, $z_{\sigma_i}$ generate $\cHFL(K)$ over $\F[U,V]$ and by Theorem \ref{thm: colored unknot},  $\AAA^k,k\ge 0$ generate $\CA_n^{\col}$ over $R_{UV}$. Furthermore,
\begin{multline}
\label{eq: A tensor product}
z_{\sigma_i}\otimes \AAA f=z_{\sigma_i}\otimes \varepsilon_n(U)f=(Uz_{\sigma_i})\otimes f=V^{\sigma_i-\sigma_{i+1}-1}z_{\sigma_{i+1}}\otimes f=
\\
z_{\sigma_{i+1}}\otimes \varepsilon_n(V^{\sigma_i-\sigma_{i+1}-1})f=z_{\sigma_{i+1}}\otimes (V_1\cdots V_n)^{\sigma_i-\sigma_{i+1}-1}f
\end{multline}
so the powers of $\AAA$ can be eliminated. 

Next, we need to check that the relations on the right hand side match the ones on the left hand side. By \eqref{eq: A tensor product} and Lemma \ref{lem: A action} the relation \eqref{eq: zigzag} in $\cHFL(K)$ translates to 
$$
z_{\sigma_i}\otimes \AAA=z_{\sigma_{i+1}}\otimes (V_1\cdots V_n)^{\sigma_i-\sigma_{i+1}-1}
$$
which holds in $\CH_{n}(K)$. Next, on the right hand side we have the relation 
$$
z_{\sigma_i}\otimes U_j=z_{\sigma_j}\otimes \AAA V_{\overline{j}}. 
$$
By again applying \eqref{eq: A tensor product} this is equivalent to
$$
z_{\sigma_i}\otimes U_j=z_{\sigma_{i+1}}\otimes (V_1\cdots V_n)^{\sigma_i-\sigma_{i+1}-1}V_{\overline{j}}
$$
which coincides with \eqref{eq: colored relations L space}.
\end{proof}


\begin{remark}
If $K$ is a plumbed L-space knot, the sufficiently large cables $K_{n,mn}$ are plumbed L-space links and by the main result of \cite{BLZ} the chain complexes $\cCFL(K_{n,mn})$ are homotopy equivalent to the resolutions of $\cHFL(K_{n,mn})$. It would be interesting to lift Theorem \ref{thm: colored tensor product} to the level of chain complexes by considering the minimal resolution of colored homology defined by \eqref{eq: colored relations L space}.
\end{remark}

\subsection{Example: $(3,4)$ torus knot}
 
Suppose that $K=T(3,4)$, then $\Delta_{K}(t)=t^3-t^2+1-t^{-2}+t^{-3}$ and $\chi_K(t)=t^3+1+t^{-1}+t^{-3}+t^{-4}+\ldots$, so 
that $S=\{3,0,-1,-3,-4,\ldots\}$. 
The $h$-function of $T(3,4)$ is given by the following table (where $\sigma_i\in S$ are marked in bold):
\begin{center}
\begin{tabular}{|c|c|c|c|c|c|c|c|c|c|c|c|c|c|}
\hline
$s$ & \dots & \bf{-5} & \bf{-4} & \bf{-3} & -2 & \bf{-1} & \bf{0} & 1 & 2 & \bf{3} & 4 & 5\\
\hline
$h(s)$ & \dots & 5 & 4 & 3 & 3 & 2 & 1 & 1 & 1 & 0 & 0 & 0\\
\hline
\end{tabular}
\end{center}
The homology $\cHFL(K)$ has generators $z_{3},z_0,z_{-1},z_{-3},z_{-4},z_{-5},\ldots$ and relations
\begin{equation}
\label{eq: T34 relations}
Uz_3=V^2z_0,\ Uz_0=z_{-1},\ Uz_{-1}=Vz_{-3},\ Uz_{-3}=z_{-4},\ Uz_{-4}=z_{-5},\ldots
\end{equation}

Now let us proceed to the $(2,2m)$ cables of $T(3,4)$ for sufficiently large $m$.
Let $m=6$, then the $h$-function (in normalized Alexander grading) of the $(2,12)$ cable of $T(3,4)$ is given by 
$$
h^{\stab}_{2,12}(\overline{s}_1,\overline{s}_2)=\begin{cases}
h(\overline{s}_1)+h(\overline{s}_2+6) & s_1\le s_2\\
h(\overline{s}_2)+h(\overline{s}_1+6) & s_2\le s_1\\
\end{cases}
$$
and shown in the following table:
\begin{center}
\begin{tikzpicture}
\draw[step=1] (0,0) grid (11,11);
\draw (11,-0.3) node {\small{$\overline{s}_1$}};
\draw (10.5,-0.3) node {\small{5}};
\draw (9.5,-0.3) node {\small{4}};
\draw (8.5,-0.3) node {\small{3}};
\draw (7.5,-0.3) node {\small{2}};
\draw (6.5,-0.3) node {\small{1}};
\draw (5.5,-0.3) node {\small{0}};
\draw (4.5,-0.3) node {\small{-1}};
\draw (3.5,-0.3) node {\small{-2}};
\draw (2.5,-0.3) node {\small{-3}};
\draw (1.5,-0.3) node {\small{-4}};
\draw (0.5,-0.3) node {\small{-5}};
\draw (-0.3,11) node {\small{$\overline{s}_2$}};
\draw (-0.3,10.5) node {\small{5}};
\draw (-0.3,9.5) node {\small{4}};
\draw (-0.3,8.5) node {\small{3}};
\draw (-0.3,7.5) node {\small{2}};
\draw (-0.3,6.5) node {\small{1}};
\draw (-0.3,5.5) node {\small{0}};
\draw (-0.3,4.5) node {\small{-1}};
\draw (-0.3,3.5) node {\small{-2}};
\draw (-0.3,2.5) node {\small{-3}};
\draw (-0.3,1.5) node {\small{-4}};
\draw (-0.3,0.5) node {\small{-5}};

\draw (10.5,10.5) node {0};
\draw (9.5,10.5) node {0};
\draw (8.5,10.5) node {0};
\draw (7.5,10.5) node {1};
\draw (6.5,10.5) node {1};
\draw (5.5,10.5) node {1};
\draw (4.5,10.5) node {2};
\draw (3.5,10.5) node {3};
\draw (2.5,10.5) node {3};
\draw (1.5,10.5) node {4};
\draw (0.5,10.5) node {5};

\draw (10.5,9.5) node {0};
\draw (9.5,9.5) node {0};
\draw (8.5,9.5) node {0};
\draw (7.5,9.5) node {1};
\draw (6.5,9.5) node {1};
\draw (5.5,9.5) node {1};
\draw (4.5,9.5) node {2};
\draw (3.5,9.5) node {3};
\draw (2.5,9.5) node {3};
\draw (1.5,9.5) node {4};
\draw (0.5,9.5) node {5};

\draw (10.5,8.5) node {0};
\draw (9.5,8.5) node {0};
\draw (8.5,8.5) node {0};
\draw (7.5,8.5) node {1};
\draw (6.5,8.5) node {1};
\draw (5.5,8.5) node {1};
\draw (4.5,8.5) node {2};
\draw (3.5,8.5) node {3};
\draw (2.5,8.5) node {3};
\draw (1.5,8.5) node {4};
\draw (0.5,8.5) node {5};

\draw (10.5,7.5) node {1};
\draw (9.5,7.5) node {1};
\draw (8.5,7.5) node {1};
\draw (7.5,7.5) node {1};
\draw (6.5,7.5) node {1};
\draw (5.5,7.5) node {1};
\draw (4.5,7.5) node {2};
\draw (3.5,7.5) node {3};
\draw (2.5,7.5) node {3};
\draw (1.5,7.5) node {4};
\draw (0.5,7.5) node {5};

\draw (10.5,6.5) node {1};
\draw (9.5,6.5) node {1};
\draw (8.5,6.5) node {1};
\draw (7.5,6.5) node {1};
\draw (6.5,6.5) node {1};
\draw (5.5,6.5) node {1};
\draw (4.5,6.5) node {2};
\draw (3.5,6.5) node {3};
\draw (2.5,6.5) node {3};
\draw (1.5,6.5) node {4};
\draw (0.5,6.5) node {5};

\draw (10.5,5.5) node {1};
\draw (9.5,5.5) node {1};
\draw (8.5,5.5) node {1};
\draw (7.5,5.5) node {1};
\draw (6.5,5.5) node {1};
\draw (5.5,5.5) node {1};
\draw (4.5,5.5) node {2};
\draw (3.5,5.5) node {3};
\draw (2.5,5.5) node {3};
\draw (1.5,5.5) node {4};
\draw (0.5,5.5) node {5};

\draw (10.5,4.5) node {2};
\draw (9.5,4.5) node {2};
\draw (8.5,4.5) node {2};
\draw (7.5,4.5) node {2};
\draw (6.5,4.5) node {2};
\draw (5.5,4.5) node {2};
\draw (4.5,4.5) node {2};
\draw (3.5,4.5) node {3};
\draw (2.5,4.5) node {3};
\draw (1.5,4.5) node {4};
\draw (0.5,4.5) node {5};

\draw (10.5,3.5) node {3};
\draw (9.5,3.5) node {3};
\draw (8.5,3.5) node {3};
\draw (7.5,3.5) node {3};
\draw (6.5,3.5) node {3};
\draw (5.5,3.5) node {3};
\draw (4.5,3.5) node {3};
\draw (3.5,3.5) node {3};
\draw (2.5,3.5) node {3};
\draw (1.5,3.5) node {4};
\draw (0.5,3.5) node {5};

\draw (10.5,2.5) node {3};
\draw (9.5,2.5) node {3};
\draw (8.5,2.5) node {3};
\draw (7.5,2.5) node {3};
\draw (6.5,2.5) node {3};
\draw (5.5,2.5) node {3};
\draw (4.5,2.5) node {3};
\draw (3.5,2.5) node {3};
\draw (2.5,2.5) node {3};
\draw (1.5,2.5) node {4};
\draw (0.5,2.5) node {5};

\draw (10.5,1.5) node {4};
\draw (9.5,1.5) node {4};
\draw (8.5,1.5) node {4};
\draw (7.5,1.5) node {4};
\draw (6.5,1.5) node {4};
\draw (5.5,1.5) node {4};
\draw (4.5,1.5) node {4};
\draw (3.5,1.5) node {4};
\draw (2.5,1.5) node {4};
\draw (1.5,1.5) node {5};
\draw (0.5,1.5) node {6};

\draw (10.5,0.5) node {5};
\draw (9.5,0.5) node {5};
\draw (8.5,0.5) node {5};
\draw (7.5,0.5) node {5};
\draw (6.5,0.5) node {5};
\draw (5.5,0.5) node {5};
\draw (4.5,0.5) node {5};
\draw (3.5,0.5) node {5};
\draw (2.5,0.5) node {5};
\draw (1.5,0.5) node {6};
\draw (0.5,0.5) node {6};

\draw[line width=3] (2,11)--(2,2)--(11,2);

\draw (8.5,8.5) circle (0.5);
\draw (5.5,5.5) circle (0.5);
\draw (4.5,4.5) circle (0.5);
\draw (2.5,2.5) circle (0.5);
\end{tikzpicture}
\end{center}

The stable range is given by $\overline{s}_1,\overline{s}_2\ge g(K)-m=3-6=-3,$
and the stable generators $\widetilde{z}_{\sigma_i}$ are marked by circles. 

For $m=7$ the $h$-function is very similar, except for
$$
h^{\stab}_{2,14}(-4,-4)=4,\
h^{\stab}_{2,14}(-4,-5)=h^{\stab}_{2,14}(-5,-4)=5,\ h^{\stab}_{2,14}(-5,-5)=6,
$$
and there is another stable generator $\widetilde{z}_{-4}$ with $\gr_{\w}(\widetilde{z}_{-4})=-8$.

For $m\ge 8$ we get
$$
h^{\stab}_{2,2m}(-4,-4)=4,\
h^{\stab}_{2,2m}(-4,-5)=h^{\stab}_{2,14}(-5,-4)=5,\ h^{\stab}_{2,2m}(-5,-5)=5,
$$
and there are stable generators $\widetilde{z}_{-4}$, $\widetilde{z}_{-5}$ with with $\gr_{\w}(\widetilde{z}_{-4})=-8,\gr_{\w}(\widetilde{z}_{-5})=-10$.
The stable relations between $\widetilde{z}_{\sigma_i}$ can be written in two ways:

1) Following \cite{BLZ} and Theorem \ref{thm: colored L space}, we can compute for $m\ge 6$:
$$
U_1\widetilde{z}_{3}=V_2(V_1V_2)^2\widetilde{z}_{0},\quad U_2\widetilde{z}_{3}=V_1(V_1V_2)^2\widetilde{z}_{0},
$$
$$
U_1\widetilde{z}_{0}=V_2\widetilde{z}_{-1},\quad U_2\widetilde{z}_{0}=V_1\widetilde{z}_{-1},
$$
$$
U_1\widetilde{z}_{-1}=V_2(V_1V_2)\widetilde{z}_{-3},\quad U_2\widetilde{z}_{-1}=V_1(V_1V_2)\widetilde{z}_{-3},
$$
For $m\ge 7$ we get additional relations
$$
U_1\widetilde{z}_{-3}=V_2\widetilde{z}_{-4},\quad U_2\widetilde{z}_{-3}=V_1\widetilde{z}_{-4},
$$
for $m\ge 8$ we get additional relations
$$
U_1\widetilde{z}_{-4}=V_2\widetilde{z}_{-5},\quad U_2\widetilde{z}_{-4}=V_1\widetilde{z}_{-5},
$$
and so on.

2) Following Lemma \ref{lem: A action}, we can instead describe the action of $\AAA$. For $m\ge 6$ we get 
\begin{equation}
\label{eq: T34 cable A relations}
\AAA\widetilde{z}_3=(V_1V_2)^2\widetilde{z}_0,\ 
\AAA\widetilde{z}_0=\widetilde{z}_{-1},\
\AAA\widetilde{z}_{-1}=(V_1V_2)\widetilde{z}_{-3},\
\AAA\widetilde{z}_{-3}=\widetilde{z}_{-4},\
\AAA\widetilde{z}_{-4}=\widetilde{z}_{-5}.
\end{equation}
The last two relations make sense for $m\ge 7$ and $m\ge 8$ respectively. Then all of the above relations are determined by \eqref{eq: T34 cable A relations} and
$
U_1=\AAA V_2,\ U_2=\AAA V_1.
$
Finally, note that \eqref{eq: T34 cable A relations} can be obtained from \eqref{eq: T34 relations} by changing $z_i$ to $\widetilde{z}_i$, $U$ to $\AAA$ and $V$ to $V_1V_2$, in agreement with Theorem \ref{thm: colored tensor product}.
 

\section{Application: crossing change maps}

Suppose two knot diagrams $K^+$ and $K^-$ differ at a single crossing which is positive in $K^+$ and negative in $K^-$. We would like to find a relation between their colored homologies, but first we relate the homology of cables. 

The blackboard framings for both $K^+$ and $K^-$ correspond to the writhe of their diagrams. Note that the writhe of $K^+$ is 2 more than the writhe of $K^-$.  The $n$-component cables of $K^+$ and $K^-$ are related by $n^2$ crossing changes. More precisely, changing $n^2$ negative crossings in $K^-_{n,mn}$ to positive results in $K^+_{n,(m+2)n}$ since the linking number between each pair of components increases by 2.


\begin{lemma}
\label{lem: big blowdown}
For $j\in \Z$ there is a map 
$$
G_{j,m}: \cHFL\left(K^-_{n,mn}\right)\to \cHFL\left(K^+_{n,(m+4)n}\right)
$$
of Maslov degree $-j^2-j$ and Alexander degree $A_i(G_{j,m})=-2j-1+2n,\ i=1,\ldots,n$. Furthermore, the maps $G_{j,m}$ commute with the connecting maps $\phi_k$:
$$
\begin{tikzcd}
\cHFL\left(K^-_{n,mn}\right) \arrow{r}{\phi_k} \arrow{d}{G_{j,m}}& \cHFL\left(K^-_{n,(m+1)n}\right) \arrow{d}{G_{j,m+1}}\\
\cHFL\left(K^+_{n,(m+4)n}\right) \arrow{r}{\phi_k}& \cHFL\left(K^+_{n,(m+5)n}\right).
\end{tikzcd}
$$
\end{lemma}

\begin{proof}
Consider the following braids on $2n$ strands: $\FT_{[1,n]}$ is the full twist on the left $n$ strands, $\FT_{[n+1,2n]}$ is the full twist on the right $n$ strands, and $\FT_{[1,2n]}$ is the full twist on all strands. Furthermore, consider the ``positive cabled crossing" braid $\beta_{n,n}$ which is the positive braid lift of the permutation $[n+1,\ldots,2n,1,\ldots,n]$:
\begin{center}
\begin{tikzpicture}
\draw (-1,1) node {$\beta_{n,n}=$};
\draw[->] (1.5,0)--(0,2);
\draw[->] (2,0)--(0.5,2);
\draw[->] (2.5,0)--(1,2);

\draw[->,line width=3,white] (0,0)--(1.5,2);
\draw[->,line width=3,white] (0.5,0)--(2,2);
\draw[->,line width=3,white] (1,0)--(2.5,2);
\draw[->] (0,0)--(1.5,2);
\draw[->] (0.5,0)--(2,2);
\draw[->] (1,0)--(2.5,2);
\end{tikzpicture}
\end{center}
It is easy to see that 
$
\FT_{[1,2n]}=\FT_{[1,n]}\FT_{[n+1,2n]}\beta_{n,n}^2$, and so $  \FT_{[1,2n]}\beta_{n,n}^{-1}=\FT_{[1,n]}\FT_{[n+1,2n]}\beta_{n,n}
$, that is: 
\begin{center}
\begin{tikzpicture}
\draw[->] (-5,0)--(-3.5,2);
\draw[->] (-4.5,0)--(-3,2);
\draw[->] (-4,0)--(-2.5,2);

\draw[->,line width=3,white] (-3.5,0)--(-5,2);
\draw[->,line width=3,white] (-3,0)--(-4.5,2);
\draw[->,line width=3,white] (-2.5,0)--(-4,2);
\draw[->] (-3.5,0)--(-5,2);
\draw[->] (-3,0)--(-4.5,2);
\draw[->] (-2.5,0)--(-4,2);

\draw (-5.5,-0.5)..controls (-5.5,0) and (-2,0)..(-2,-0.5);

\draw[line width=3,white] (-5,-1.5)--(-5,0);
\draw[line width=3,white] (-4.5,-1.5)--(-4.5,0);
\draw[line width=3,white] (-4,-1.5)--(-4,0);
\draw[line width=3,white] (-3.5,-1.5)--(-3.5,0);
\draw[line width=3,white] (-3,-1.5)--(-3,0);
\draw[line width=3,white] (-2.5,-1.5)--(-2.5,0);

\draw (-5,-1.5)--(-5,0);
\draw (-4.5,-1.5)--(-4.5,0);
\draw (-4,-1.5)--(-4,0);
\draw (-3.5,-1.5)--(-3.5,0);
\draw (-3,-1.5)--(-3,0);
\draw (-2.5,-1.5)--(-2.5,0);

\draw[line width=3,white] (-5.5,-0.5)..controls (-5.5,-1) and (-2,-1)..(-2,-0.5);
\draw (-5.5,-0.5)..controls (-5.5,-1) and (-2,-1)..(-2,-0.5);

\draw (-6,-0.5) node {$-1$};

\draw (-1.5,1) node {$=$};

\draw[->] (0,0)--(1.5,2);
\draw[->] (0.5,0)--(2,2);
\draw[->] (1,0)--(2.5,2);

\draw[->,line width=3,white] (1.5,0)--(0,2);
\draw[->,line width=3,white] (2,0)--(0.5,2);
\draw[->,line width=3,white] (2.5,0)--(1,2);
\draw[->] (1.5,0)--(0,2);
\draw[->] (2,0)--(0.5,2);
\draw[->] (2.5,0)--(1,2);

\draw (-0.2,0)--(2.7,0)--(2.7,-1)--(-0.2,-1)--(-0.2,0);
\draw (0,-1)--(0,-1.5);
\draw (0.5,-1)--(0.5,-1.5);
\draw (1,-1)--(1,-1.5);
\draw (1.5,-1)--(1.5,-1.5);
\draw (2,-1)--(2,-1.5);
\draw (2.5,-1)--(2.5,-1.5);
\draw (1.3,-0.5) node {$\FT$};

\draw (3.5,1) node {$=$};
 
\draw[->] (6.5,0)--(5,2);
\draw[->] (7,0)--(5.5,2);
\draw[->] (7.5,0)--(6,2);

\draw[->,line width=3,white] (5,0)--(6.5,2);
\draw[->,line width=3,white] (5.5,0)--(7,2);
\draw[->,line width=3,white] (6,0)--(7.5,2);
\draw[->] (5,0)--(6.5,2);
\draw[->] (5.5,0)--(7,2);
\draw[->] (6,0)--(7.5,2);

\draw (4.8,0)--(6.1,0)--(6.1,-1)--(4.8,-1)--(4.8,0);
\draw (6.4,0)--(7.7,0)--(7.7,-1)--(6.4,-1)--(6.4,0);
\draw (5,-1)--(5,-1.5);
\draw (5.5,-1)--(5.5,-1.5);
\draw (6,-1)--(6,-1.5);
\draw (6.5,-1)--(6.5,-1.5);
\draw (7,-1)--(7,-1.5);
\draw (7.5,-1)--(7.5,-1.5);
\draw (5.5,-0.5) node {$\FT$};
\draw (7,-0.5) node {$\FT$};
\end{tikzpicture}
\end{center}
This means that adding a full twist  to $K^{-}_{n,mn}$ results in $K^{+}_{n,(m+2)n}$ with two additional full twists which is isotopic to $K^{+}_{n,(m+4)n}$. Now the construction of $G_{j,m}$ follows from \cite[Proposition 3.10]{AGL}, this also gives the Maslov grading change.

For the Alexander grading we use equation \eqref{eq: Alexander shift generalized}. Let $M$ be the $(-1)$-framed unknot. Since $\lk(K^{-}_{n, mn},M)=2n$ and $\lk(L_i,M)=2$ for each component $L_i$, we get
$$
A_i=\frac{1}{2}\left[-2j-1+2n\right]\cdot 2=-2j-1+2n.
$$

The cobordism defining $\phi_k$ corresponds to blowing down a meridian of $K^{-}_{n,mn}$ (resp. $K^+_{n,(m+4)n}$) which can be done away from the crossing, and therefore the two cobordisms commute and induced maps in Heegaard Floer homology commute.  
\end{proof}

\begin{lemma}
\label{lem: big blowdown 2}
For $j\in \Z$ there is a map 
$$
F_{j,m}: \cHFL\left(K^+_{n,mn}\right)\to \cHFL\left(K^-_{n,(m-2)n}\right)
$$
of Maslov degree $-j^2-j$ and Alexander degree $A_i=0$. Furthermore, the maps $F_{j,m}$ commute with the connecting maps $\phi_k$:
$$
\begin{tikzcd}
\cHFL\left(K^+_{n,mn}\right) \arrow{r}{\phi_k} \arrow{d}{F_{j,m}}& \cHFL\left(K^+_{n,(m+1)n}\right) \arrow{d}{F_{j,m+1}}\\
\cHFL\left(K^-_{n,(m-2)n}\right) \arrow{r}{\phi_k}& \cHFL\left(K^-_{n,(m-1)n}\right).
\end{tikzcd}
$$
\end{lemma}

\begin{proof}
The proof is similar to the proof of Lemma \ref{lem: big blowdown}, where we now use the following diagram (compare with \cite[Proposition 3.1]{AGL}):
\begin{center}
\begin{tikzpicture}
\draw (-3,-1.5)..controls (-2.5,-1.5) and (-2.5,3.5)..(-3,3.5);

\draw[->,line width=3,white] (-2,-2)--(-6.5,4);
\draw[->,line width=3,white] (-1.5,-2)--(-6,4);
\draw[->,line width=3,white] (-1,-2)--(-5.5,4);
\draw[->] (-2,-2)--(-6.5,4);
\draw[->] (-1.5,-2)--(-6,4);
\draw[->] (-1,-2)--(-5.5,4);

\draw[->,line width=3,white] (-6.5,-2)--(-2,4);
\draw[->,line width=3,white] (-6,-2)--(-1.5,4);
\draw[->,line width=3,white] (-5.5,-2)--(-1,4);

\draw[->] (-6.5,-2)--(-2,4);
\draw[->] (-6,-2)--(-1.5,4);
\draw[->] (-5.5,-2)--(-1,4);

\draw[line width=3,white]  (-3,-1.5)..controls (-3.5,-1.5) and (-3.5,3.5)..(-3,3.5);

\draw (-3,-1.5)..controls (-3.5,-1.5) and (-3.5,3.5)..(-3,3.5);

\draw (-3,-2) node {$-1$};

\end{tikzpicture}
\end{center}
Blowing down the circle changes the positive cabled crossing in $K^+_{n,mn}$ to the negative one (resulting in $K^-_{n,(m-2)n}$), introduces a positive full twist on one set of $n$ strands and a negative full twist on the second set of $n$ strands. Since $K$ is connected, we can slide the full twists along $K$ and cancel with each other. As a result, we are left with $K^-_{n,(m-2)n}$. 

For the Alexander degree, we denote the $(-1)$-framed unknot by $M$, and  use \eqref{eq: Alexander shift generalized} again. We have $\lk(L_i,M)=0$ for all component $L_i$ of $K^+_{n, mn}$, so the Alexander grading does not change. 
 \end{proof}

\begin{theorem}
a) The maps $G_{j,m}$ induce maps 
$$
G_j:\Cab_n(K^{-})\to \Cab_n(K^+),\ \gr_{\w}(G_j)=-j^2-j,\ A_i(G_j)=-2j-1+2n
$$ 
of twist degree $4$ which commute with the action of the cable algebra $\CA_n$. 

b) We obtain maps 
$$
G_j^{\col}:\CH_n(K^{-})\to \CH_n(K^+),\ \gr_{\w}(G_j^{\col})=-j^2-j,\ A_i(G_j^{\col})=-2j+1
$$
which commute with the action of the algebra $\CA_n^{\col}$.

c) The maps $F_{j,m}$ induce maps 
$$
F_j:\Cab_n(K^{+})\to \Cab_n(K^-),\ \gr_{\w}(F_j)=-j^2-j,\ A_i(F_j)=0
$$ 
of twist degree $-2$ which commute with the action of the cable algebra $\CA_n$. 

d) We obtain maps 
$$
F_j^{\col}:\CH_n(K^{+})\to \CH_n(K^-),\ \gr_{\w}(F_j^{\col})=-j^2-j,\ A_i(F_j^{\col})=n-1
$$
which commute with the action of the algebra $\CA_n^{\col}$.
\end{theorem}

\begin{proof}
Part (a) is a restatement of Lemma \ref{lem: big blowdown}, and  (b) follows from (a) after regrading
$$
-2j-1+2n-(m+4)(n-1)/2+m(n-1)/2=-2j+1.
$$
Similarly, (c) follows from Lemma \ref{lem: big blowdown 2} and (d) follows from (c) after regrading
$$
0-(m-2)(n-1)/2+m(n-1)/2=n-1.
$$
\end{proof}

\begin{remark}
The cables $K^+_{n,(m+2)n}$ and $K^-_{n,mn}$ are related by a sequence of $n^2$ crossing changes, and we can instead consider compositions of local cobordisms maps at these crossings as in \cite[Section 5]{AGL}.

Let $Z\subset \{1,\ldots,n\}\times \{1,\ldots,n\}$ be an arbitrary subset. Then there exists a   map 
$$
\Psi_{Z,m}:\cHFL\left(K^+_{n,(m+2)n}\right)\to \cHFL\left(K^-_{n,mn}\right)
$$
of Maslov degree zero and Alexander degree 
$$
A(\Psi_{Z,m})=\frac{1}{2}\left(\sum_{(i,j)\in Z}(\ee_i-\ee_j)-\sum_{(i,j)\notin Z}(\ee_i-\ee_j)\right).
$$
It is easy to see that, as before, $\Psi_{Z,m}$ commutes with the connecting maps $\phi_k$, and hence defines a map $\CH_n(K^+)\to \CH_n(K^-)$.
\end{remark}


\begin{thebibliography}{99}

\bibitem{AGL} A. Alishahi, E. Gorsky, B. Liu. Splitting maps in link Floer homology and integer points in permutahedra. arXiv:2307.07741

\bibitem{BGHW} W. Ballinger, E. Gorsky, M. Hogancamp, J. Wang. Stable deformed $\mathfrak{gl}_N$ homology of torus knots. arXiv:2507.00175


\bibitem{BPRW}  A. Beliakova, K. Putyra, L.-H. Robert, E. Wagner. A proof of Dunfield-Gukov-Rasmussen Conjecture. J. Eur. Math. Soc. (2025), arXiv:2210.00878

\bibitem{BorodzikGorskyImmersed} M. Borodzik and E. Gorsky. Immersed concordances of links and Heegaard Floer homology. Indiana Univ. Math. J. 67 (2018), no. 3, 1039--1083. 



\bibitem{BLZ} M. Borodzik, B. Liu, I. Zemke. Lattice homology, formality,  and plumbed L-space links.  J. Eur. Math. Soc. (2024), arXiv:2210.15792

\bibitem{Chen} D. Chen, I. Zemke, H. Zhou. L-space satellite operators and knot Floer homology. arXiv:2412.05755

\bibitem{Cautis} S. Cautis. Clasp technology to knot homology via the affine Grassmannian. Mathematische Annalen 363 (2015): 1053--1115.

\bibitem{Conners} L. Conners. Row--column mirror symmetry for colored torus knot homology.
Selecta Mathematica 30 (5), article 97.

\bibitem{Cooper} B. Cooper, R. Deyeso.  Holonomicity from a Heegaard-Floer Perspective. arXiv:2501.01519

\bibitem{CK} B. Cooper, V. Krushkal. Categorification of the Jones--Wenzl projectors. Quantum Topology 3.2 (2012): 139--180.

\bibitem{DGR} N. Dunfield, S. Gukov, J. Rasmussen. The superpolynomial for knot homologies. Experimental Mathematics 15.2 (2006): 129--159.

\bibitem{GHog} E. Gorsky, M. Hogancamp. Hilbert schemes and  $y$-ification of Khovanov-Rozansky homology.
Geom. Topol. 26 (2022), no. 2, 587--678.

\bibitem{GH} E. Gorsky, J. Hom. Cable links and L-space surgeries.  
Quantum Topology 8 (2017), no.4, 629--666.

\bibitem{GNR} E. Gorsky, A. Negu\cb{t}, J. Rasmussen. Flag Hilbert schemes, colored projectors and Khovanov-Rozansky homology. Advances in Mathematics 378 (2021): 107542.

\bibitem{GOR} E. Gorsky, A. Oblomkov, J. Rasmussen. On stable Khovanov homology of torus knots. Experimental Mathematics 22.3 (2013): 265--281.

\bibitem{Hedden} M. Hedden. On knot Floer homology and cabling
Algebraic \& Geometric Topology 5 (2005), 1197--1222.

\bibitem{Hedden2}  M. Hedden. On knot Floer homology and cabling. II.
International Mathematics Research Notices, 2009, no. 12, 2248--2274.

\bibitem{Hogancamp} 
M. Hogancamp. Categorified Young symmetrizers and stable homology of torus links.
Geometry \& Topology 22 (2018) 2943--3002.

\bibitem{Hogancamp2} M. Hogancamp. A polynomial action on colored  $\mathfrak{sl}_2$  link homology.
Quantum Topol. 10 (2019), no. 1, 1--75.


\bibitem{Kh}  M. Khovanov. A categorification of the Jones polynomial.
Duke Math. J. 101 (2000), no. 3, 359--426.

\bibitem{KR1} M. Khovanov, L. Rozansky. Matrix factorizations and link homology.
Fund. Math. 199 (2008), no. 1, 1--91.

\bibitem{KR2}  M. Khovanov, L. Rozansky. Matrix factorizations and link homology. II
Geom. Topol. 12 (2008), no. 3, 1387--1425.


\bibitem{Licata} J. Licata. Heegaard Floer homology of $(n,n)$-torus links: computations and questions.  arXiv:1208.0394.

\bibitem{Lipshitz-cylindrical} R. Lipshitz. A cylindrical reformulation of Heegaard Floer homology. Geom. Topol. 10 (2006) 955-1096.

\bibitem{Yajing} Y. Liu. L-space surgeries on links. Quantum Topol. 8 (2017), no. 3, 505--570.

\bibitem{MM} P. Melvin, H. Morton. The coloured Jones function. Comm. Math. Phys. 169 (1995), no. 3, 501--520.



\bibitem{OSKnots} P. Ozsv\'ath, Z. Szab\'o. Holomorphic disks and knot invariants. Adv. Math. 186 (2004), no.1, 58--116. 



\bibitem{OSProperties} P. Ozsv\'ath, Z. Szab\'o. Holomorphic disks and three-manifold invariants: properties  and applications. Ann. of Math. (2) 159 (2004), no.3, 1159--1245. 

\bibitem{OSlens} P. Ozsv\'ath, Z. Szab\'o. On knot Floer homology and lens space surgeries. Topology 44 (2005), no.6, 1281--1300.


\bibitem{OSLinks} P. Ozsv\'ath, Z. Szab\'o. Holomorphic disks, link invariants and the multi-variable  Alexander polynomial. Algebr. Geom. Topol. 8 (2008), no. 2, 615--692.

\bibitem{RasmussenKnots} J. Rasmussen. Floer homology and knot complements. PhD thesis, Harvard University, 2003.



\bibitem{RasDiff} J. Rasmussen. Some differentials on Khovanov–Rozansky homology. Geometry \& Topology, 19(6), 3031--3104.

\bibitem{RT}  N. Y. Reshetikhin and V. G. Turaev. Ribbon graphs and their invariants derived from
quantum groups. Comm. Math. Phys. 127 (1990), 1--26.

\bibitem{RJ} M. Rosso and V. F. R. Jones. On the invariants of torus knots derived from quantum groups. Journal of Knot Theory and its 
Ramifications 2 (1993), 97--112.

\bibitem{Rozansky} L. Rozansky. An infinite torus braid yields a categorified Jones--Wenzl projector. Fundamenta Mathematicae 225 (2014): 305--326.

\bibitem{Wildi} A. Wildi. On Categorifications and Applications of Knot Invariants. PhD thesis, University of Z\"urich, 2025.


\bibitem{Zemke:graph} I. Zemke. Graph cobordisms and Heegaard Floer homology. arXiv:1512.01184.

\bibitem{Zemke} I. Zemke. Link cobordisms and functoriality in link Floer homology. J. Topol. 12 (2019), no. 1, 94--220. 

\bibitem{Zemke2} I. Zemke. Link cobordisms and absolute gradings on link Floer homology. Quantum Topol. 10 (2019), no. 2, 207--323. 

\bibitem{Zemke:basepoint} I. Zemke. Quasistabilization and basepoint moving maps in link Floer homology. Algebraic \& Geometric Topology 17 (2017), no. 6, 3461 -- 3518.

\bibitem{Zemke:bordered} I. Zemke. Bordered manifolds with torus boundary and the link surgery formula. arXiv:2109.11520.


\end{thebibliography}
\end{document}